\documentclass[12pt,twoside]{amsart}
\usepackage{mathrsfs}
\usepackage{txfonts}
\usepackage{bbm}
\usepackage{amsfonts}
\usepackage{amsmath}
\usepackage{cite}
\numberwithin{equation}{section}
\usepackage{esint}
\usepackage{amssymb}
\usepackage{wasysym}
\usepackage{psfrag}
\usepackage{graphics,epsfig,amsmath}
\usepackage{comment}
\usepackage{enumerate}
\usepackage{times}
\usepackage{fancyhdr}

\newtheorem{Theorem}{Theorem}[section]
\newtheorem{Lemma}[Theorem]{Lemma}
\newtheorem{Proposition}[Theorem]{Proposition}
\newtheorem{Definition}[Theorem]{Definition}

\newtheorem{Corollary}[Theorem]{Corollary}
\usepackage{metalogo}
\usepackage[cmyk]{xcolor}
\selectfont
\usepackage{mathtools}
\mathtoolsset{showonlyrefs}
\renewcommand{\epsilon}{\varepsilon}

\newcommand{\R}{\mathbb{R}}
\newcommand{\N}{\mathbb{N}}

\newcommand{\C}{\mathbb{C}}

\renewcommand{\le}{\leqslant}

\renewcommand{\ge}{\geqslant}

\allowdisplaybreaks

\usepackage{geometry}
\title[Mixed order Schr\"{o}dinger equation]{Qualitative properties of positive solutions \\ of a mixed order nonlinear Schr\"{o}dinger equation}

\author[S. Dipierro]{Serena Dipierro}
\address{Department of Mathematics and Statistics, University of Western Australia, 35 Stirling Highway, WA 6009 Crawley, Australia}
\email{serena.dipierro@uwa.edu.au}

\author[X. Su]{Xifeng Su}
\address{School of Mathematical Sciences, Laboratory of Mathematics and Complex Systems (Ministry of Education)\\
	Beijing Normal University,
	No. 19, XinJieKouWai St., HaiDian District, Beijing 100875, P. R. China}
\email{xfsu@bnu.edu.cn, billy3492@gmail.com}

\author[E. Valdinoci]{Enrico Valdinoci}
\address{Department of Mathematics and Statistics, University of Western Australia, 35 Stirling Highway, WA 6009 Crawley, Australia}
\email{enrico.valdinoci@uwa.edu.au}

\author[J. Zhang]{Jiwen Zhang}
\address{School of Mathematical Sciences, Laboratory of Mathematics and Complex Systems (Ministry of Education)\\
	Beijing Normal University,
	No. 19, XinJieKouWai St., HaiDian District, Beijing 100875, P. R. China}
\address{Department of Mathematics and Statistics, University of Western Australia, 35 Stirling Highway, WA 6009 Crawley, Australia}
\email{jwzhang826@mail.bnu.edu.cn, jiwen.zhang@uwa.edu.au}


\begin{document}
	\maketitle
	
	\begin{abstract}
	In this paper, we deal with the following  mixed local/nonlocal Schr\"{o}dinger equation 
		 \begin{equation*}
		 	\left\{
		 	\begin{array}{ll}
		 		- \Delta u +  (-\Delta)^s u+u = u^p \quad \hbox{in $\mathbb{R}^n$,} \\
		 		u>0 \quad \hbox{in $\mathbb{R}^n$,}\\
		 		\lim\limits_{|x|\to+\infty}u(x)=0,\\
		 	\end{array}
		 \right.
		 \end{equation*}
where~$n\geqslant2$, $s\in (0,1)$ and~$p\in\left(1,\frac{n+2}{n-2}\right)$.
		 
		The existence of { positive} solutions for the above problem is proved, relying on some new regularity results.
		In addition, we study the power-type decay and the radial symmetry properties of such solutions.
		
The methods make use also of some basic properties of the heat kernel and the Bessel kernel  associated with  the  operator~$- \Delta  +  (-\Delta)^s$: in this context, we provide self-contained proofs of these results based on Fourier analysis techniques.
	\end{abstract}
	
\tableofcontents
	
	\section{Introduction}
	
	In this paper, 
we are concerned with qualitative properties of solutions to the following mixed local/nonlocal Schr\"{o}dinger equation
	\begin{equation}\label{Maineq}
			- \Delta u +  (-\Delta)^s u+u = u^p \quad \hbox{in $\mathbb{R}^n$,}
			\end{equation}
satisfying~$u>0$ in~$\mathbb{R}^n$ and
$$  \lim\limits_{|x|\to+\infty}u(x)=0.$$
	Here above, for any~$s\in(0,1)$, the fractional Laplacian is defined as
	\begin{equation*}
		(-\Delta )^{s} u(x)=c_{n,s}\,{\text P.V.}\int_{\mathbb{R}^{n}} \frac{u(x)-u(y)}{|x-y|^{n+2s}}\, dy,
	\end{equation*}
	where~$ c_{n,s}>0$ is a suitable normalization constant, whose explicit value does not play a role here, and~${\text P.V.}$ means that the integral is taken in the Cauchy principal value sense.

	Moreover, we suppose that~$n\ge2$ and
	\begin{equation}\label{duiwegfufvckuqfhoq496712123hew}
	p\in\left(1,\frac{n+2}{n-2}\right)\; {\mbox{ if }} n>2\qquad
	{\mbox{and}} \qquad p\in(1,+\infty)\; {\mbox{ if }} n=2.\end{equation}
	
We recall that, on the one hand, the Schr\"{o}dinger equation is the fundamental equation of physics for describing quantum mechanics. 
		Feynman and Hibbs~\cite{feynman1963quantum} 
	 formulated the non-relativistic quantum mechanics as a path integral over the Brownian
		paths, and this background leads to standard (non-fractional)
		Schr\"{o}dinger equation. In this setting,
		during the last 30 years there have been important contributions to the analysis of the classical nonlinear Schr\"{o}dinger equation, see e.g.~\cite{MR867665, MR886933,MR969899,
		MR1162728, MR1201323, MR1218300}.
	
On the other hand, a nonlocal version of the Schr\"{o}dinger equation has been introduced by Laskin, see~\cite{MR1755089, MR1948569, MR2932616}. Laskin  
	constructed the fractional path integral and formulated the space fractional quantum
	mechanics on the basis of the L\'{e}vy flights. In the recent years, the fractional Schr\"{o}dinger equation has been studied, under different perspectives, by many authors, see in particular~\cite{MR1111746, MR3002595, MR3070568, MR3207007, MR3121716, MR3393677, MR3530361} and the references therein.

In this framework (see e.g. the appendix in~\cite{MR3393677}) the classical and the fractional 
Schr\"{o}dinger equations arise from a Feynman path integral over ``all possible histories of the system'', subject
to an action functional obtained as the superposition of a
diffusive operator and a potential term: specifically, when the diffusive operator is Gaussian
(i.e., as produced by classical Brownian motions), this procedure returns the classical
Schr\"{o}dinger equation, while when the diffusive process is $2s$-stable for some~$s\in(0,1)$ (e.g., as induced by L\'{e}vy flights),
one obtains
the fractional Schr\"{o}dinger equation corresponding to the fractional Laplacian~$(-\Delta)^s$.

In view of this construction, it is also interesting to consider the case in which the diffusive operator
in the action functional of the Feynman path integral is of ``mixed type'', i.e., rather than possessing a given scaling invariance,
it is obtained as the overlapping of two different diffusive operators, say a Gaussian and a $2s$-stable one
(or, similarly, that the diffusive operator acts alternately in a Gaussian fashion and in a $2s$-stable way for ``infinitesimal times''
with respect to the time scale considered). This superposition of stochastic processes in the quantum action functional give
rise precisely to the equation proposed in~\eqref{Maineq}.
\medskip
		
Our goal in this paper is to investigate existence and qualitative properties of solutions to~\eqref{Maineq}. More precisely, we will obtain results of this type:
	\begin{itemize}
		\item 	We prove the existence of nonnegative weak solutions by exploiting
		 Ekeland's variational principle (see Theorems~\ref{th existence} below). 
		 \item We then obtain H\"{o}lder continuity  and~$C^{1,\alpha}$-regularity of weak solutions based on the $L^p$-theory for the Laplacian and 
some basic properties of the kernel of the mixed Bessel potential (see Theorems~\ref{Theorem :holder regularity} and ~\ref{th:regularity} below). To overcome the difficulties caused by the fact that the equation involves operators of different orders and therefore is not scale-invariant,
		  we  develop a ``piecewise" argument presented  in Appendices~\ref{th properties of h1}-\ref{sce properties of k} which is new in the literature.
		  \item
Based on the~$C^{1,\alpha}$-regularity result, we next
establish a $C^{2,\alpha}$-regularity result  
by combining a suitable truncation method and a covering argument 
(see  Theorem~\ref{th C^2,alpha interior without boundary condition} below).
		  \item Finally, we discuss qualitative properties of classical solutions, such as positivity, power-type decay at infinity, and radial symmetry (see Theorem~\ref{th main theorem} below).
	\end{itemize}

We now state the main results of this paper.

\begin{Theorem}[Existence of nonnegative weak solutions]\label{th existence}
There exists a nontrivial and nonnegative weak solution~$u$ of~\eqref{Maineq}.
\end{Theorem}

The result above will be established
by exploiting the Mountain Pass geometry of the functional associated with~\eqref{Maineq}. Additionally, we will show that the energy of the solution in Theorem~\ref{th existence} is precisely at the Mountain Pass level, see Corollary~\ref{aggdopomeglio}.

\medskip
 
We  next  establish some regularity results which provide useful auxiliary tools in the analysis of the equation under consideration. 
For this, 
we notice that  the pseudo-differential operator~$-\Delta+(-\Delta)^{s}$ is the infinitesimal generator of a stochastic process~$X$, where~$X$ is the mixture of the Brownian motion and an independent symmetric $2s$-stable L\'{e}vy process, and, as such, can be characterized using the Fourier Transform~${\mathcal{F}}$ as
$$ {\mathcal{F}}\Big(-\Delta u+(-\Delta)^su\Big)(\xi)= \Big(|\xi|^2+|\xi|^{2s}\Big) \,
{\mathcal{F}}u(\xi).$$
Hence, we can define a heat kernel~$\mathcal{H}(x,t)$  associated with  the above operator  as
\begin{equation}\label{DBSD-1}
	\mathcal{H}(x,t):=\int_{\mathbb{R}^n}e^{-t(|\xi|^{2}+|\xi|^{2s})+2\pi ix\cdot \xi}\,d\xi
\end{equation}
for~$t>0$ and~$x\in\mathbb{R}^n$. 

In this context, the heat kernel~$\mathcal{ H}(x,t)$  may be viewed as a transition density of the L\'{e}vy process~$X$ and,
relying on probabilistic methods,
some properties about the transition density of~$ X $ 
have been established:
for example, Song and Vondra\v{c}ek~\cite{MR2321989} obtained
upper and lower bounds on the transition density of~$ X $ by comparing the transition densities of the Brownian motion and the $2s$-stable process.
	 
	 On a related note, Bogdan, Grzywny, and Ryznar~\cite{zbMATH06282052}	obtained general bounds on the transition density of a pure-jump  L\'{e}vy process  based on Laplace transform arguments (by adapting this method appropriately, one can find a  sharp upper bound on the transition density of~$X$, but no  lower bound since the L\'{e}vy-Khintchine exponent of the operator~$-\Delta+(-\Delta)^{s}$ is not strictly between~$0$ and~$2$). 

Furthermore,	  
	  Cygan, Grzywny and Trojan~\cite{MR3646773} 
	   provided the asymptotic behaviors of the transition density of a L\'{e}vy process whose 
	L\'{e}vy-Khintchine exponent is strictly and regularly varying between the indexes~$0$ and~$2$
	(from these methods, since in our framework
	the index of the stochastic process at zero is~$2s$ and at infinity
	is~$2$, only  the asymptotic formula for the tail transition density of~$X$ can be obtained).

Additionally, the  L\'{e}vy process~$X$ may be considered as a subordinate Brownian motion, and thus, with the aid of suitable Tauberian Theorems developed by~\cite{bingham1989regular}, the asymptotic behaviors of the Green function of~$X$ near zero and infinity were established by Rao, Song
and Vondra\v{c}ek~\cite{MR2238934}.

In this paper we will revisit these results, framing them in the setting needed for our purposes, and provide a new set of
proofs making use of Fourier analysis techniques (see Theorem~\ref{th properties of h}). These properties will serve as an essential step to investigate some basic features of the  Bessel kernel~$\mathcal{ K}$, defined as
\begin{equation}\label{definition of K}\mathcal{K}(x):=\int_{0}^{+\infty} e^{-t} \,\mathcal{H}(x,t) \,dt.
\end{equation}
In turn, these
properties of~$\mathcal{ K}$ (see Theorem~\ref{th properties of k}) will be pivotal to establish a series of regularity results for the mixed operator which can be described as follows.\medskip

The first regularity result that we present deals with
the H\"{o}lder continuity of weak solutions of problem~\eqref{Maineq}:
 
 \begin{Theorem}[H\"{o}lder continuity]\label{Theorem :holder regularity}
 	Let~$u$ be a weak solution of~\eqref{Maineq}.
 	
 	Then, $u\in C^{0,\mu}(\mathbb{R}^n)$ for some~$\mu\in(0,1)$ and 	\begin{equation*}
 		\|u\|_{C^{0,\mu}(\mathbb{R}^n)}\leqslant C,
 	\end{equation*}
 	for some~$C>0$, depending only on~$n$, $s$, $p$ and~$\|u\|_{H^1(\mathbb{R}^n)}$.
 \end{Theorem}	

We also establish a~$C^{1,\alpha}$-regularity result:

\begin{Theorem}[$C^{1,\alpha}$-regularity]\label{th:regularity}
	Let~$u$ be a bounded weak solution of~\eqref{Maineq}.
	
	Then, $u\in C^{1,\alpha}(\mathbb{R}^n)$ for any~$\alpha\in(0,1)$
and	\begin{equation*}
		\|u\|_{C^{1,\alpha}(\mathbb{R}^n)}\leqslant c\, \left(\|u\|_{L^\infty(\mathbb{R}^n)}+\|u\|^p_{L^\infty(\mathbb{R}^n)}\right),
	\end{equation*}
	for some~$c>0$, depending only on~$n$, $s$ and~$p$.
	
	Moreover, 
	$$ \lim_{|x|\to+\infty} u(x)=0.$$
\end{Theorem}	

{F}rom Theorem~\ref{th:regularity}, we will also deduce the~$C^{2,\alpha}$-regularity result below.

	\begin{Theorem}[$ C^{2,\alpha} $-regularity]\label{th C^2,alpha interior without boundary condition}
		Let~$ u$ be a bounded weak solution of~\eqref{Maineq}.
		
		Then, $u\in C^{2,\alpha}(\mathbb{R}^n)$ for any~$\alpha\in(0,1)$
		and
		\begin{equation*}
			\|u\|_{C^{2,\alpha}(\mathbb{R}^n)}\leqslant C\max\left\{\|u\|^{p^2}_{L^\infty(\mathbb{R}^n)},  \|u\|_{L^\infty(\mathbb{R}^n)}\right\},
		\end{equation*}
for some~$ C>0 $, depending only on~$ n$, $s$, $p$ and~$\alpha$.
	\end{Theorem}

Finally, we  obtain the power-type decay at infinity and the radial symmetry of classical positive solutions to~\eqref{Maineq} by comparison arguments and the method of moving plane, respectively:

\begin{Theorem}[Qualitative properties of positive solutions]\label{th main theorem}
The problem in~\eqref{Maineq} admits a classical, positive, radially symmetric solution~$u$.

Furthermore, $u$ has a power-type decay at infinity, that is, there exist constants~$0<C_1\leqslant C_2$ such that, for all~$|x|\ge1$, 
	\[ \frac{C_1}{|x|^{n+2s}}\leqslant u(x)\leqslant \frac{C_2}{|x|^{n+2s}}.
	\]
\end{Theorem}

\medskip

The rest of this paper is organized as follows.
In Section~\ref{sec:preliminary}, we introduce the functional framework that we work in, and  then obtain the existence of nonnegative weak solutions 
as given by Theorem~\ref{th existence}. 

Section~\ref{sec: C^alpha-regularity} is devoted to 
the proofs of the regularity results in Theorems~\ref{Theorem :holder regularity}, \ref{th:regularity} and~\ref{th C^2,alpha interior without boundary condition}. 

In Section~\ref{sec: Qualitative properties of positive solution}, we establish
Theorem~\ref{th main theorem}.

Several properties of the heat kernel~\eqref{DBSD-1} and the Bessel kernel~\eqref{definition of K} are established in Appendices~\ref{th properties of h1} and~\ref{sce properties of k}.  

	\section{Existence of weak solutions }\setcounter{equation}{0}\label{sec:preliminary}

The aim of this section is to establish Theorem~\ref{th existence}. For this, we will first
introduce the functional setting and provide some basic definitions. The
proof of Theorem~\ref{th existence} will then occupy the forthcoming Section~\ref{sec:fguwewiorg}.

\subsection{Functional framework}
Let~$n\geqslant2$ and~$s\in(0,1)$.

As usual, the norm in~$ L^p(\mathbb{R}^n) $ will be denoted by~$ \|u\|_{L^p(\R^n)} $
and the norm in~$H^1(\R^n)$ will be denoted by~$\|u\|_{H^1(\R^n)}=
\|u\|_{L^2(\R^n)}+\|\nabla u\|_{L^2(\R^n)}$.

We also recall the so-called Gagliardo (semi)norm of~$ u $, defined as 
$$
	[u]_{H^s(\mathbb{R}^n)}:=\left(\iint_{\mathbb{R}^{2n}}\frac{|u(x)-u(y)|^2}{|x-y|^{n+2s}}\, dx\,dy\right)^{\frac{1}{2}}.
$$

We point out that the Gagliardo (semi)norm is controlled
by the~$H^1$-norm, according to the following observation:

\begin{Lemma}\label{lemma:H^1=H^1,s} Let~$s\in(0,1)$.

Then, there exists a constant~$ C>0 $, depending only on~$n$ and~$s$, such that,
for all~$  u\in H^{1}(\mathbb{R}^n)  $,
$$
		[u]^2_{H^s(\mathbb{R}^n)}\leqslant C \Big(\|u\|^2_{L^2(\R^n)} +\|\nabla u\|^2_{L^2(\R^n)}\Big).
$$
\end{Lemma}

\begin{proof}
Utilizing the H\"{o}lder inequality, we see that
	\begin{equation*}
		\begin{split}
			[u]^2_{H^s(\mathbb{R}^n)}&=\iint_{\mathbb{R}^{2n}}\frac{|u(x)-u(y)|^2}{|x-y|^{n+2s}}\, dx\,dy\\
			&=\int_{\mathbb{R}^n}\int_{|z|< 1}\frac{|u(x+z)-u(x)|^2}{|z|^{n+2s}}\, dx\,dz
			+\int_{\mathbb{R}^n}\int_{|z|\geqslant 1}\frac{|u(x+z)-u(x)|^2}{|z|^{n+2s}}\, dx\,dz\\
			&\leqslant\int_{\mathbb{R}^n}\int_{|z|< 1}\frac{|\int_{0}^{1}\nabla u(x+tz) z\,dt|^2}{|z|^{n+2s}}\, dx\,dz
			+2\int_{\mathbb{R}^n}\int_{|z|\geqslant 1}\frac{|u(x+z)|^2+|u(x)|^2}{|z|^{n+2s}}\, dx\,dz\\
			&\leqslant  \|\nabla u\|^2_{L^2(\R^n)}\int_{|z|< 1}\frac{dz}{|z|^{n+2s-2}}
			+4\|u\|^2_{L^2(\R^n)}\int_{|z|\geqslant 1}\frac{dz}{|z|^{n+2s}}\\
			&\leqslant C\left(\|u\|^2_{L^2(\R^n)}+\|\nabla u\|^2_{L^2(\R^n)}\right),
		\end{split}
	\end{equation*}
	as desired.
\end{proof}

In light of Lemma~\ref{lemma:H^1=H^1,s}, given~$s\in(0,1)$, we will equip~$H^1(\R^n)$ with the equivalent norm
\begin{equation*} \|u\|_s:=\Big(
\|u\|^2_{L^2(\R^n)}+\|\nabla u\|^2_{L^2(\R^n)}+[u]^2_{H^s(\mathbb{R}^n)}\Big)^{\frac12}\end{equation*}
and the scalar product
$$ \langle u,v\rangle_s:=\int_{\R^n}\Big(u(x)v(x)+\nabla u(x)\cdot\nabla v(x)\Big)\,dx
+\iint_{\R^{2n}}\frac{(u(x)-u(y))(v(x)-v(y))}{|x-y|^{n+2s}}\,dx\,dy.$$

\begin{Definition}[Weak solution] \label{defweaksol123}
	We say that~$ u\in H^{1}(\mathbb{R}^n) $ is a weak solution of~\eqref{Maineq}
	if, for all~$ v \in H^1(\mathbb{R}^n)$,
	\begin{equation*}
				\int_{\mathbb{R}^n}\Big( u(x) v(x)+	
\nabla u(x)\cdot \nabla v(x)\Big)\, dx+\iint_{\mathbb{R}^{2n}}\frac{(u(x)-u(y)(v(x)-v(y))}{|x-y|^{n+2s}}\, dx\,dy
			=\int_{\mathbb{R}^n}|u(x)|^{p-1} u(x) v(x)\,dx.
	\end{equation*}
\end{Definition}

We consider the functional~$\mathscr{F}:H^{1}(\mathbb{R}^n) \to\R$ defined as
\begin{equation*}\begin{split}
	\mathscr{F}(u):=\;&\frac12\|u\|^2_{s}-\frac{1}{p+1}\int_{\mathbb{R}^n} |u(x)|^{p+1}\, dx
\\	=\;&\frac{1}{2}\int_{\mathbb{R}^n}\Big( u^2(x)+|\nabla u(x)|^{2 }\Big)\,dx+\frac{1}{2}\iint_{\mathbb{R}^{2n}}\frac{|u(x)-u(y)|^{2}}{|x-y|^{n+2s}}\, dx\,dy -\frac{1}{p+1}\int_{\mathbb{R}^n} |u(x)|^{p+1}\, dx.  \end{split}
\end{equation*}
We point out that, in light of Lemma~\ref{lemma:H^1=H^1,s} and the classical Sobolev embedding, the functional~$\mathscr{F}$ is finite for all~$ u\in H^{1}(\mathbb{R}^n) $.

Furthermore, we have that~$ \mathscr{F} \in C^{1}(H^{1}(\mathbb{R}^n)) $ and,
for all~$ \varphi \in H^{1}(\mathbb{R}^n) $,
\begin{eqnarray*}&&
		\left\langle \nabla\mathscr{F}(u),\varphi \right\rangle \\&=&
		\langle u,\varphi\rangle_s-\int_{\mathbb{R}^n}| u(x)|^{p-1} u(x)\varphi(x)\, dx\\&=&
		\int_{\mathbb{R}^n} \Big(u(x) \varphi(x)+ \nabla u(x)\cdot\nabla \varphi(x)\Big)\, dx  + \iint_{\mathbb{R}^{2n}}\frac{(u(x)-u(y))(\varphi(x)-\varphi(y))}{|x-y|^{n+2s}}\, dx\,dy-\int_{\mathbb{R}^n}| u(x)|^{p-1} u(x)\varphi(x)\, dx.
\end{eqnarray*} 
As a consequence, being a weak solution of~\eqref{Maineq} according to Definition~\ref{defweaksol123} is equivalent to being a critical point of the functional~$\mathscr{F}$.

\medskip

Now we introduce an auxiliary problem.
For this, we use the notation~$ u^+:=\max\left\{u,0\right\} $ and~$ u^-:=\max\left\{-u,0\right\} $, so that~$u=u^+-u^-$, and we consider the following problem 
\begin{equation}\label{auxproblem123}
	\left\{
	\begin{array}{ll}
			-\Delta u+(-\Delta )^su+u=( u^+)^p \quad \text{in } \mathbb{R}^n \\
		u\geqslant0 \quad \hbox{in $\mathbb{R}^n$,}\\
		\lim\limits_{|x|\to+\infty}u(x)=0.\\
	\end{array}
	\right.
\end{equation}

The corresponding functional~$\mathscr{F}^+:H^{1}(\mathbb{R}^n)\to\R$ is given by
\begin{equation}\label{fgeiwrtfdnwjg0987654iuytre}\begin{split}
	\mathscr{F}^+(u):=\;&
	\frac12\|u\|^2_s -\frac{1}{p+1}\int_{\mathbb{R}^n}( u^+)^{p+1}(x)\, dx\\
=\;&	
	\frac{1}{2}\int_{\mathbb{R}^n}\Big( u^2(x)+|\nabla u(x)|^{2 }\Big)\,dx+\frac{1}{2}\iint_{\mathbb{R}^{2n}}\frac{|u(x)-u(y)|^{2}}{|x-y|^{n+2s}}\, dx\,dy -\frac{1}{p+1}\int_{\mathbb{R}^n}( u^+)^{p+1}(x)\, dx.\end{split}
\end{equation} 
We point out that~$ \mathscr{F}^+ \in C^{1}(H^{1}(\mathbb{R}^n)) $ and,
for all~$ \varphi \in H^{1}(\mathbb{R}^n) $,
\begin{equation*}\begin{split}
&		\left\langle \nabla\mathscr{F}^+(u),\varphi \right\rangle\\
=\;&\langle u,\varphi\rangle_s
-\int_{\mathbb{R}^n}( u^+)^p(x)\varphi(x)\, dx
\\ =\;&\int_{\mathbb{R}^n} \Big(u(x) \varphi(x)+ \nabla u(x)\cdot\nabla \varphi(x)\Big)\, dx  + \iint_{\mathbb{R}^{2n}}\frac{(u(x)-u(y))(\varphi(x)-\varphi(y))}{|x-y|^{n+2s}}\, dx\,dy-\int_{\mathbb{R}^n}( u^+)^p(x)\varphi(x)\, dx.\end{split}
\end{equation*} 
Accordingly, being a weak solution of~\eqref{auxproblem123} is equivalent
to being a critical point of~$\mathscr{F}^+(u)$.
\medskip

The role of this auxiliary functional is made explicit by the following observation:

\begin{Lemma}\label{lemma:pivotal}
Let~$u$ be a critical point for~$ \mathscr{F}^+$. 

Then,
\begin{equation}\label{pippoclaim1}
{\mbox{$u\geqslant0$ a.e. in~$\mathbb{R}^n$.}}\end{equation}
Moreover, \begin{equation}\label{pippoclaim2}
{\mbox{$u$ is a critical point for~$ \mathscr{F}$.}}\end{equation}
\end{Lemma}

\begin{proof}
We observe that~$u^-\in H^1(\R^n)$ and therefore
\begin{eqnarray*}
	0&=&\left\langle \nabla\mathscr{F}^{+}(u),u^{-}  
	\right\rangle \\
	&=& \int_{\mathbb{R}^n} \Big(u(x) u^-(x)+ \nabla u(x)\cdot\nabla u^-(x)\Big)\, dx  + \iint_{\mathbb{R}^{2n}}\frac{(u(x)-u(y))(u^-(x)-u^-(y))}{|x-y|^{n+2s}}\, dx\,dy\\&&\qquad -\int_{\mathbb{R}^n}( u^+)^p(x)u^-(x)\, dx\\	
	&=& -\int_{\mathbb{R}^n} \Big((u^-)^2(x)+ |\nabla u^-(x)|^2\Big)\, dx  - \iint_{\mathbb{R}^{2n}}\frac{(u^-(x)-u^-(y))^2}{|x-y|^{n+2s}}\, dx\,dy
	\\&&\qquad- \iint_{\mathbb{R}^{2n}}\frac{u^+(x)u^-(y)+u^-(x)u^+(y)}{|x-y|^{n+2s}}\, dx\,dy\\
	&=&-\|u^-\|^2_s
-\iint_{\mathbb{R}^{2n}}\frac{u^+(x)u^-(y)+u^+(y)u^-(x)}{|x-y|^{n+2s}} \, dx\,dy\le- \|u^{-}\|_s^{2}.
\end{eqnarray*}	 
This implies that~$u^-=0$ a.e. in~$\R^n$, thus establishing~\eqref{pippoclaim1}.

As a consequence of~\eqref{pippoclaim1}, recalling also the definitions of the functionals~$\mathscr{F}$ and~$\mathscr{F}^+$,
we obtain~\eqref{pippoclaim2}.
\end{proof}

In view of Lemma~\ref{lemma:pivotal}, we will now focus on critical points for the functional~$\mathscr{F}^+.$ Let
\[  \Gamma:=\left\{\gamma\in C([0,1],H^{1}(\mathbb{R}^n)) \;{\mbox{ s.t. }}\; \gamma(0)=0\; {\mbox{ and }}\; \mathscr{F}^+(\gamma(1))<0\right\} \]
and define
\begin{equation}\label{fhueowt6943gvksholierghrlif09876}
	c:=\inf\limits_{\gamma\in\Gamma}\sup\limits_{t\in[0,1]}\mathscr{F}^+(\gamma(t)).
\end{equation}

We make the following observation:

\begin{Lemma}
We have that~$\Gamma\not=\varnothing$ and
\begin{equation}\label{fhueowt6943gvksholierghrlif09876BIS}c>0.\end{equation} \end{Lemma}

\begin{proof}
Since~$p>1$, we have that~$\mathscr{F^+}$ is unbounded from below and therefore~$\Gamma\not=\varnothing$. 	

Moreover, given~$u\in H^1(\mathbb{R}^n)$ with~$u\not\equiv 0$, by the Sobolev inequality one has that
	\begin{equation*}
			\mathscr{F^+}(u)
			=\frac{1}{2}\|u\|^2_{s}-\frac{1}{p+1}\|u\|_{L^{p+1}({\mathbb{R}}^n)}^{p+1}
			\geqslant \|u\|^2_{s}\left(\frac{1}{2}-\frac{\widetilde c}{p+1}\|u\|_{s}^{p-1}\right),
	\end{equation*}
	where~$\widetilde c>0$ depends on~$n$ and~$p$.
	
Hence, if~$\|u\|_s$ is sufficiently small, say~$\|u\|_s\in\left(0,\left(\frac{p+1}{4\widetilde c}\right)^{\frac{1}{p-1}}\right)$,
we deduce that
$$\mathscr{F^+}(u)\geqslant \frac{1}{4}\|u\|_s^2>0,$$
which gives~\eqref{fhueowt6943gvksholierghrlif09876BIS}.
\end{proof}

We remark that~$0$ is a local minimum of~$\mathscr{F}^+$ rather than a global minimum. Thus the functional~$\mathscr{F}^+$ fulfills the geometric hypothesis of the Mountain Pass Theorem, but
the Palais-Smale condition is not necessarily satisfied.
To overcome this issue, we will make use of the following result, established in~\cite{MR1181725},
which allows us to prove that~$\mathscr{F}^+$ (and therefore~$\mathscr{F}$)
has a positive critical value, thus entailing Theorem~\ref{th existence}.

\begin{Lemma}\label{compact}(\cite[Lemma~2.18]{MR1181725})
	Let~$n\geqslant2$. Assume that~$\left\{u_k\right\}$ is bounded in~$H^{1}(\mathbb{R}^n)$ and that there exists~$R > 0$ such
	that
	\begin{equation*}
		\liminf_{k\to+\infty}\sup\limits_{y\in\mathbb{R}^n}\int_{B_{R}(y)}|u_k(x)|^2\, dx=0.
	\end{equation*}

Then, $u_k\to 0$ in~$L^q(\mathbb{R}^n)$ for all~$q\in(2,2^*)$.
\end{Lemma}

 \subsection{Existence of weak solutions}\label{sec:fguwewiorg}

With the preliminary work done so far, we can now complete the proof of Theorem~\ref{th existence}.

\begin{proof}[Proof of Theorem~\ref{th existence}]
	{We recall the definition of the auxiliary functional~$\mathscr{F}^+$
	in~\eqref{fgeiwrtfdnwjg0987654iuytre}.
	In light of Lemma~\ref{lemma:pivotal}, in order to find a nonnegative weak solution of~\eqref{Maineq}, we focus on finding a nontrivial critical point
	for the auxiliary functional~$\mathscr{F}^+$. 

We claim that
\begin{equation}\label{gdttrryclaim1}
{\mbox{the functional~$\mathscr{F}^+$ has a nontrivial critical point.}}\end{equation}
To prove this, we use the Ekeland's variational principle (see e.g.~\cite{MR982267}) and obtain that
there exists a sequence~$u_k$ such that, as~$k\to+\infty$,
	\begin{equation}\label{ekeland}
		\mathscr{F}^+(u_k)\to c \qquad\text{and}\qquad  \nabla\mathscr{F^+}(u_k)\to 0,
	\end{equation}
	where~$c$ is given in~\eqref{fhueowt6943gvksholierghrlif09876}.
	
We notice that, from~\eqref{ekeland}, for~$k$ large enough,
\begin{equation}\label{kjhgfdqwertyu2345678}
		c+1
		\geqslant\mathscr{F}^+(u_k)-\frac{1}{p+1}\left\langle \nabla\mathscr{F}^+(u_k),u_k\right\rangle
		=\left(\frac{1}{2}-\frac{1}{p+1}\right)\|u_k\|^2_{s}.
	\end{equation}
It follows that~$\left\{u_k\right\}$ is bounded in~$H^1(\mathbb{R}^n)$,
and therefore, up to a subsequence, $\left\{u_k\right\}$
converges to some~$u\in H^1(\mathbb{R}^n)$
weakly in~$H^1(\mathbb{R}^n)$ and strongly in~$L^q_{loc}(\mathbb{R}^n)$
for all~$q\in(1,2^*)$.

As a result, for any~$\varphi\in C^\infty_0(\mathbb{R}^n)$,
\begin{equation}\label{alsoter254854896}
	\begin{split}&0=\lim_{k\to+\infty}
	\left\langle\nabla\mathscr{F}^+(u_k),\varphi\right\rangle=\lim_{k\to+\infty}\left(
	\left\langle u_k,\varphi \right\rangle_{s}-\int_{\mathbb{R}^n}(u_k^+)^p\varphi\, dx\right)\\
	&\qquad=\left\langle u,\varphi \right\rangle_{s}-\int_{\mathbb{R}^n}(u^+)^p\varphi\, dx= \left\langle\nabla\mathscr{F}^+(u),\varphi\right\rangle,
\end{split}
\end{equation}
which entails that~$\nabla\mathscr{F}^+(u)=0$.

Now, for every~$R>0$, we set
$$ \alpha:=\liminf_{k\to+\infty} \sup_{y\in\R^n}\int_{B_R(y)}|u_k(x)|^2\,dx$$
and we claim that
\begin{equation}\label{uneq00}
\alpha\in(0,+\infty).
\end{equation}		
Indeed, clearly~$\alpha\in[0,+\infty]$. Moreover, since~$u_k$ converges strongly 
in~$L^2_{loc}(\mathbb{R}^n)$ to some~$u\in H^1(\R^n)$, we have that~$\alpha\in[0,+\infty)$.

Hence, it remains to show that~$\alpha>0$. For this, we argue by contradiction and suppose
that~$\alpha=0$. Then,  
we deduce from Lemma~\ref{compact} that~$u_k\to 0~$ in~$ L^{p+1}(\mathbb{R}^n)$. As a consequence of this, and recalling~\eqref{ekeland}, we have that
	\begin{equation*}
		\begin{split}
			&c=
			\lim_{k\to+\infty}\left(\mathscr{F}^+(u_k)-\frac{1}{2}\left\langle\nabla\mathscr{F}^+(u_k),u_k\right\rangle\right)
			=\left(\frac{1}{2}-\frac{1}{p+1}\right)\lim_{k\to+\infty}
			\int_{\mathbb{R}^n}(u_k^+)^{p+1}(x)\,dx\\
			&\qquad=\left(\frac{1}{2}-\frac{1}{p+1}\right)\lim_{k\to+\infty}\|u_k^+\|^{p+1}_{L^{p+1}(\mathbb{R}^n)}=0,
			\end{split}
	\end{equation*}
	in contradiction with~\eqref{fhueowt6943gvksholierghrlif09876BIS}.
	This establishes~\eqref{uneq00}.
	
It now follows from~\eqref{uneq00} that
we have that there exists a sequence~$y_k\in\mathbb{R}^n$ such that, for all~$k\in\N$ (and possibly
taking a subsequence of~$u_k$),
	\begin{equation}\label{uneq0}
		\int_{B_{R}(y_k)}|u_k(x)|^2\, dx>\frac{\alpha}2.
	\end{equation}

In light of~\eqref{uneq0}, we define
\begin{equation}\label{kjhgfdmnbvcxz09876543}
\bar{u}_k(x):=u_k(x+y_k)\end{equation}
and we see that~$\|\bar{u}_k\|_{s}=\|{u}_k\|_{s}$. Therefore,
 \begin{equation*}
 \mathscr{F^+}(\bar{u}_k)=\frac{1}{2}\|u_k\|_{s}^2 -\frac{1}{p+1}\int_{\mathbb{R}^n}( u_k^+)^{p+1}(x)\, dx= \mathscr{F^+}({u}_k)
 \end{equation*}
and, for all~$\varphi\in H^1(\R^n)$,
\begin{eqnarray*}
&&\left\langle  \nabla\mathscr{F}^+(\bar{u}_k),\varphi \right\rangle=
\left\langle \bar{u}_k,\varphi \right\rangle_s-\int_{\R^n} (\bar{u}_k^+)^{p}(x)\varphi(x)\,dx\\
&&\qquad=\left\langle {u}_k,\varphi(\cdot-y_k) \right\rangle_s-\int_{\R^n} ({u}_k^+)^{p}(x)\varphi(x-y_k)\,dx=\left\langle  \nabla\mathscr{F}^+({u}_k),\varphi(\cdot-y_k) \right\rangle,
\end{eqnarray*}
and thus, in view of~\eqref{ekeland},
we have that, as~$k\to+\infty$,
\begin{equation}\label{kjhgfdmnbvcxz0987654300}
\mathscr{F}^+(\bar{u}_k)\to c\qquad{\mbox{and}}\qquad \nabla\mathscr{F^+}(\bar{u}_k)\to 0.\end{equation}

Now, using~\eqref{kjhgfdqwertyu2345678} with~$u_k$ replaced by~$\bar{u}_k$, we conclude that
there exists~$\bar{u}\in H^1(\mathbb{R}^n)$ such that
\begin{equation}\label{kjhgfdmnbvcxz098765430}
{\mbox{$\bar{u}_k$ converges to~$\bar{u}$
weakly in~$ H^1(\mathbb{R}^n)$ and strongly
in~$ L^{q}_{loc}(\mathbb{R}^n)$ for all~$q\in(1,2^*)$.}}\end{equation}

Moreover, by~\eqref{alsoter254854896} with~$u_k$ replaced by~$\bar{u}_k$,
we deduce that~$\nabla \mathscr{F}^+(\bar{u})=0 $, and hence~$\bar{u}$
is a critical point for~$ \mathscr{F}^+$. Accordingly, to complete the proof
of the claim in~\eqref{gdttrryclaim1}, it only remains to check that
\begin{equation}\label{gdttrryclaim10}
\bar{u}\not\equiv0.
\end{equation} 
To this end, we recall~\eqref{uneq0} and change variable~$z:=x-y_k$ to get that
\begin{eqnarray*}
\frac\alpha2<\int_{B_{R}(y_k)}|u_k(x)|^2\, dx=\int_{B_{R}(y_k)}|\bar{u}_k(x-y_k)|^2\, dx
=\int_{B_{R}}|\bar{u}_k(z)|^2\, dz
\end{eqnarray*}
Since~$ \bar{u}_k$ converges to~$\bar{u}$
strongly in~$L^2_{loc}(\mathbb{R}^n) $, this implies that
$$ \frac\alpha2\le \int_{B_{R}}|\bar{u}(z)|^2\, dz,$$
which gives~\eqref{gdttrryclaim10}, thus completing the proof of~\eqref{gdttrryclaim1}.}
\end{proof}

In particular, we can show that
the nontrivial critical point for~$\mathscr{F}^+$ found in
the proof of Theorem~\ref{th existence} is
at the critical level~$c$
	given by~\eqref{fhueowt6943gvksholierghrlif09876}:

\begin{Corollary}\label{aggdopomeglio}
Let~$\bar u$ be the nontrivial critical point for~$\mathscr{F}^+$,
as given by~\eqref{gdttrryclaim1}, and let~$c$ be as in~\eqref{fhueowt6943gvksholierghrlif09876}.

Then,
\begin{equation}\label{gdttrryclaim2}
\mathscr{F}^+(\bar{u})=c.
\end{equation}
\end{Corollary}

\begin{proof}
{We first check that
\begin{equation}\label{10gdttrryclaim1}
\mathscr{F^+}(\bar{u})\geqslant c.\end{equation}
Indeed, we observe that
$$0=\left\langle\nabla \mathscr{F}^+(\bar{u}),\bar{u}\right\rangle
=\|\bar{u}\|^2_{s}-\|\bar{u}^+\|^{p+1}_{L^{p+1}(\mathbb{R}^n)} ,$$ 
and therefore, in light of~\eqref{gdttrryclaim10},
we deduce that~$\bar{u}^+\not\equiv 0$. 

Moreover, for every~$t>0$,
\begin{eqnarray*}
&&\mathscr{F^+}(t\bar{u})
=\frac{t^2}2\|\bar u\|^2_s -\frac{t^{p+1}}{p+1}\int_{\mathbb{R}^n}( \bar{u}^+)^{p+1}(x)\, dx,
\end{eqnarray*}
and thus, since~$p+1>2$, we have that~$\mathscr{F^+}(t\bar{u})\to-\infty $ as~$t\to+\infty$. Hence, there exists some~$T>0$ such that~$\mathscr{F^+}(T\bar{u})<0 $. 

As a consequence, we consider the path~$\gamma_{\bar{u}}(t):=tT\bar{u}$ for all~$t\in[0,1]$ and notice that~$\gamma_{\bar{u}}\in\Gamma$.  Thus,
 \begin{equation}\label{c<}
 	c\leqslant \sup_{t\in[0,1]}\mathscr{F^+}(\gamma_{\bar{u}}(t))=\sup_{t\in[0,1]}\mathscr{F^+}(tT\bar{u}).
 \end{equation}
In addition, let us define~$h_u(t):=\mathscr{F^+}(t\bar u)$ for all~$t>0$. In this way, we compute 
 \[ h_{\bar{u}}'(t)=\left\langle\nabla\mathscr{F^+}(t\bar{u}),\bar{u}\right\rangle= t\|\bar{u}\|^2_{s}-t^p\|\bar{u}^+\|^{p+1}_{L^{p+1}(\mathbb{R}^n)} \]
 and we see that~$h_{\bar{u}}$ has a unique maximum~$t_{\bar{u}}>0$. 
 
Thus, we have that~$\left\langle\nabla \mathscr{F^+}(t_{\bar{u}}\bar{u}),t_{\bar{u}}\bar{u}\right\rangle=0$, which implies that~$t_{\bar{u}}=1$. Accordingly, we have that~$ \mathscr{F^+}(\bar{u})\geqslant \mathscr{F^+}(t\bar{u})$ for all~$t>0$. By combining this information with~\eqref{c<}, we obtain~\eqref{10gdttrryclaim1}.
 
 We now check that
\begin{equation}\label{11gdttrryclaim1}
\mathscr{F^+}(\bar{u})\leqslant c.\end{equation}
For this, we recall the definition of the sequence~$\bar u_k$ in~\eqref{kjhgfdmnbvcxz09876543} and we observe that, for every~$r>0$,
\begin{equation*}
\mathscr{F}^+(\bar{u}_k)-\frac{1}{2}\left\langle\nabla\mathscr{F}^+(\bar{u}_k),\bar{u}_k\right\rangle=\left(\frac{1}{2}-\frac{1}{p+1}\right)\int_{\mathbb{R}^n}(\bar{u}_k^+)^{p+1}(x)\, dx\geqslant\left(\frac{1}{2}-\frac{1}{p+1}\right)\int_{B_{r}}(\bar{u}_k^+)^{p+1}(x)\, dx.
\end{equation*}
Taking the limit as~$k\to+\infty$ and recalling~\eqref{kjhgfdmnbvcxz0987654300} and~\eqref{kjhgfdmnbvcxz098765430}, we get that
\begin{equation*}
	c\geqslant \left(\frac{1}{2}-\frac{1}{p+1}\right)\int_{B_{r}}(\bar{u}^+)^{p+1}(x)\, dx.
\end{equation*}
Since this holds true for all~$r>0$, we conclude that
\begin{equation*}
	c\geqslant \left(\frac{1}{2}-\frac{1}{p+1}\right)\int_{\R^n}(\bar{u}^+)^{p+1}(x)\, dx
	=\mathscr{F}^+(\bar{u})-\frac{1}{2}\left\langle\nabla\mathscr{F}^+(\bar{u}),\bar{u}\right\rangle=\mathscr{F}^+(\bar{u}),
\end{equation*}
which gives~\eqref{11gdttrryclaim1}.

{F}rom~\eqref{10gdttrryclaim1} and~\eqref{11gdttrryclaim1}, we obtain~\eqref{gdttrryclaim2}, as desired.}
\end{proof}

\section{Regularity of weak solutions}\label{sec: C^alpha-regularity}

In this section we focus on the H\"{o}lder continuity, $C^{1,\alpha}$ and $C^{2,\alpha}$-regularity results
for weak solutions of~\eqref{Maineq} amd prove Theorems~\ref{Theorem :holder regularity}, \ref{th:regularity} and~\ref{th C^2,alpha interior without boundary condition}.
These regularity  estimates will be the basis for the qualitative analysis we carry out in the next section.

To this end, we start with some basic properties of the heat kernel~\eqref{DBSD-1} and the Bessel kernel~\eqref{definition of K} which are essential
for proving Theorems~\ref{Theorem :holder regularity} and~\ref{th:regularity}, as summarized by the following two results.

\begin{Theorem}[Properties of the heat kernel]\label{th properties of h}
	Let~$n\geqslant 1$ and~$s\in(0,1)$. Let~$\mathcal{ H}$ be as defined in~\eqref{DBSD-1}. 

	Then,
	\begin{itemize}
		\item $\mathcal{H}$ is nonnegative, radially symmetric and nonincreasing with respect to~$r=|x|$.
		\item There exist positive constants~$C_1$ and~$C_2$ such that
		\begin{eqnarray*}
&&			\mathcal{H}(x,t)\leqslant C_1 \left\{\frac{t}{|x|^{n+2s}}\vee \frac{t^s}{|x|^{n+2s}}\right\}\wedge \left\{ t^{-\frac{n}{2s}}\wedge  t^{-\frac{n}{2}}\right\}
	\\		\mbox{and } \quad &&		\mathcal{H}(x,t)\geqslant C_2\begin{cases}
				\frac{t}{|x|^{n+2s}} \qquad&\mbox{if }\; 1<t<|x|^{2s} ,\\
				e^{\frac{\pi|x|^2}{t}}t^{-\frac{n}{2}} \qquad&\mbox{if }\; |x|^2<t<|x|^{2s}<1,
			\end{cases}
		\end{eqnarray*}
		where~$a\wedge b:=\min\left\{a,b\right\}$ and~$a\vee b:=\max\left\{a,b\right\}$.
	\end{itemize}
\end{Theorem}

\begin{Theorem}[Properties of the Bessel kernel]\label{th properties of k}
	Let~$n\geqslant 1$ and~$s\in(0,1)$. Let~$\mathcal{ K}$ be as defined in~\eqref{definition of K}.
	
	Then,
	\begin{itemize}
		\item[(a)] $\mathcal{K}$ is positive, radially symmetric, smooth in~$\mathbb{R}^n\setminus\left\{0\right\}$ and nonincreasing  with respect to~$r=|x|$.
		\item[(b)] There exist positive constants~$C_3$ and~$C_4$ such that, if~$|x|\geqslant 1$,
		$$
			\frac{C_3}{|x|^{n+2s}}\leqslant\mathcal{K}(x)\leqslant \frac{1}{C_3|x|^{n+2s}}
		$$
	and, if~$|x|\leqslant 1$,
		\begin{equation*}
			\frac{C_4}{|x|^{n-2}}\leqslant\mathcal{K}(x)\leqslant\frac1{C_4}\begin{cases}
				{|x|^{2-n}} \qquad &\text{ if }{n\geqslant 3},\\
				1+|\ln|x|| &\text{ if } n=2,\\
				1&\text{ if } n=1.
			\end{cases} 
		\end{equation*}
		\item[(c)] There exists a positive constant~$C_5$ such that, if~$|x|\ge1$,
		$$
			|\nabla\mathcal{K}(x)|\leqslant \frac{C_5}{|x|^{n+2s+1}}\qquad \text{and} \qquad |D^2\mathcal{ K}(x)|\leqslant \frac{C_5}{|x|^{n+2s+2}}.
		$$
				\item[(d)] If~$n\geqslant 3$, then~$\mathcal{K}\in L^q(\mathbb{R}^n)$ for all~$q\in [1,\frac{n}{n-2})$. If~$n=1$, $2$, then~$\mathcal{K}\in L^q(\mathbb{R}^n)$ for all~$q\in[1,+\infty)$.
		\item[(e)] If~$q\in[1,+\infty)$ and~$f\in L^q(\mathbb{R}^n)$, then the function~$u:=\mathcal{K}\ast f$ is a solution of $$
		- \Delta u +  (-\Delta)^s u+u = f \quad \text{ in }  \mathbb{R}^n.$$
	\end{itemize}
\end{Theorem}
 
 For the reader’s convenience,  we provide the detailed proofs of Theorems~\ref{th properties of h} and~\ref{th properties of k} by using Fourier analysis techniques in Appendices~\ref{th properties of h1} and~\ref{sce properties of k}.

\subsection{H\"{o}lder continuity of weak solutions}\label{sec:C^{1,alpha}-regularity of weak solution1}

We devote this section to the proof of Theorem~\ref{Theorem :holder regularity}, which will achieved by combining
$L^p$-theory and Theorem~\ref{th properties of k}. 

\subsubsection{$W^{2,p}$-regularity of weak solutions}

We shall explore the~$W^{2,p}$-regularity theory for weak solutions of linear equations,
which is a pivotal step towards the proof of the H\"{o}lder continuity for the nonlinear equation, which will be obtained combining
the $L^p$-theory and a localization trick.

\begin{Lemma}\label{lemma embedding of X^p}
Let~$p\ge1$ and~$f\in L^{p}(\mathbb{R}^n)$. 
	Let~$u$ be a solution of
	$$	-\Delta u+(-\Delta )^{s}u+u=f\qquad \text{in }\mathbb{R}^n.
	$$ 
	
	Then, $ u\in W^{2,p}(\mathbb{R}^n)$ and
	\[ \|u\|_{W^{2,p}(\mathbb{R}^n)}\leqslant C\|f\|_{L^p(\mathbb{R}^n)}, \]
for some constant~$C>0$ depending on~$n$, $s$ and~$p$.
	\end{Lemma}
	
\begin{proof}
The gist of the proof relies on checking that~$u$ satisfies an equation of the type~$-\Delta u+u=g $, for some function~$g$, in order
to apply the classical Calder\'on–Zygmund regularity theory to this equation. For this, the core of the argument will be to check that~$g\in L^p(\R^n)$ (and that~$\|g\|_{L^p(\R^n)}\le C \|f\|_{L^p(\R^n)}$).

The technical details of the proof go as follows. 
{F}rom~\cite[Theorem~3, page~135]{MR290095}, we know that
we can identify the Sobolev space~$W^{2,p}(\mathbb{R}^n)$ with the space
	\begin{equation}\label{w 2,p}
		\mathcal{W}^{2,p}:=\left\{u\in L^p(\mathbb{R}^n)\;{\mbox{ s.t. }}\;\mathcal{F}^{-1}\Big((1+|\xi|^2)\hat{u}\Big)\in L^p(\mathbb{R}^n)\right\}.
	\end{equation}  
In light of this, it suffices to show that~$ u\in \mathcal{W}^{2,p} $.

	We first claim that
	\begin{equation}\label{asdfghj1234rtyeu8gt5490} u\in	L^{p}(\mathbb{R}^n).\end{equation}
To this end, we point out that, in the distributional sense,
	\[ u=\mathcal{F}^{-1}\left(\frac{1}{1+|\xi|^2+|\xi|^{2s}}\right)\ast f=\mathcal{K}\ast f ,\]
	where the kernel~$\mathcal{K}$ is given in~\eqref{definition of K}. Since~$\mathcal{ K}\in L^1(\mathbb{R}^n)$
	(thanks to~(b) of Theorem~\ref{th properties of k}), the Young's convolution inequality gives that
	\[ \|u\|_p\leqslant C_{n,s} \|f\|_p \]
	for some constant~$C_{n,s}>0$. This estaslishes~\eqref{asdfghj1234rtyeu8gt5490}.
	
	Moreover, we remark that~$u$ satisfies
		\begin{eqnarray*}&& \Big(1+|\xi|^2\Big)\hat{u}=\frac{1+|\xi|^2}{1+|\xi|^2+|\xi|^{2s}}\hat{f}=\hat{g}, \\
		{\mbox{where }}\quad && g:=\left(\delta_0+\mathcal{ K}-\mathcal{F}^{-1}\left(\frac{1+|\xi|^{2s}}{1+|\xi|^2+|\xi|^{2s}}\right)\right)\ast f.
		\end{eqnarray*}
		Here above and in what follows~$\delta_0$ denotes the Dirac's delta at~$0$.  
	
Now we define 
\begin{eqnarray*}
		\phi(\xi)&:=&\frac{1+|\xi|^{2s}}{1+|\xi|^2+|\xi|^{2s}}=\frac{1}{1+\frac{|\xi|^2}{1+|\xi|^{2s}}}
\\
{\mbox{and }} \quad 
		\mathcal{ I}(x)&:=&
		\mathcal{F}^{-1}\phi(x)
		=\int_{0}^{+\infty} e^{-t} \int_{\mathbb{R}^n}e^{-\frac{t|\xi|^2}{1+|\xi|^{2s}}}e^{2\pi ix\cdot \xi}\,d\xi \,dt.
	\end{eqnarray*}
	
We claim that
\begin{equation}\label{dow43ytoeghewoiuewty8043}\mathcal{ I} \in L^1(\mathbb{R}^n) .\end{equation}

Once the claim in~\eqref{dow43ytoeghewoiuewty8043} will be established, the proof of Lemma~\ref{lemma embedding of X^p}
can be completed as follows. One notices that~$g=(\delta_0+\mathcal{ K}-\mathcal{I})\ast f=f+\mathcal{K}\ast f-\mathcal{I}\ast f$, and therefore
$$
	\|g\|_{L^p(\R^n)}=\|f+\mathcal{ K}\ast f-\mathcal{ I}\ast f\|_{L^p(\R^n)}\leqslant \left(1+\|\mathcal{ K}\|_{L^1(\R^n)}+\|\mathcal{ I}\|_{L^1(\R^n)}\right)\|f\|_{L^p(\R^n)}.
$$
{F}rom this, one concludes that~$u\in \mathcal{W}^{2,p}$ and that~$\|u\|_{W^{2,p}(\mathbb{R}^n)}\leqslant C\|f\|_{L^p(\mathbb{R}^n)}$, as
desired.

Hence, to complete the proof of Lemma~\ref{lemma embedding of X^p}, we now focus on the proof of~\eqref{dow43ytoeghewoiuewty8043}.
For this,  we define, for every~$\kappa$, $t_1$, $t_2>0$ and~$x\in\mathbb{R}^n$,
$$
		\mathcal{J}(x,t_1,\kappa,t_2):=\int_{\mathbb{R}^n}e^{-\frac{t_1|\xi|^2}{\kappa+t_2|\xi|^{2s}}}e^{2\pi ix\cdot \xi}\,d\xi
$$
and we observe that
	\begin{equation}\label{scaling  j}
			\mathcal{J}(x,t,1,1)
			=t^{-\frac{n}{2}}\mathcal{J}\left(t^{-\frac{1}{2}}x,1,1,t^{-s}\right)
			=t^{-\frac{n}{2-2s}}\mathcal{J}\left(t^{-\frac{1}{2-2s}}x,1,t^{\frac{s}{1-s}},1\right).
	\end{equation}
	Also, for any~$t>0$, we notice that 
	\begin{equation}
		\begin{split}\label{split t}
			|\mathcal{J}(x,t,1,1)| 
			&\leqslant C_{n,s}\left(t^{-\frac{n}{2-2s}}\vee t^{-\frac{n}{2}}\right).
		\end{split}
	\end{equation}
	
We split~$\mathcal{I}$ as~$\mathcal{I}(x)=\mathcal{I}_1(x)+\mathcal{I}_2(x)$, where
$$			\mathcal{I}_1(x):=\int_{1}^{+\infty} e^{-t} \mathcal{J}(x,t,1,1) \,dt
\qquad{\mbox{and}}\qquad
\mathcal{I}_2(x):=
\int_{0}^{1} e^{-t} \mathcal{J}(x,t,1,1) \,dt.$$
We will now prove in {Step~1} that~$ \mathcal{I}_1\in L^1(\mathbb{R}^n)$
and in {Step~2} that~$ \mathcal{I}_2 \in L^1(\mathbb{R}^n)$. Thus, {Step~1} and {Step~2} will give the desired
claim in~\eqref{dow43ytoeghewoiuewty8043}.
	\smallskip
	
	{Step~1. } We prove that~$ \mathcal{I}_1 \in L^1(\mathbb{R}^n)$ by exploiting  the rescaling property~\eqref{scaling  j}. To this end,
	we observe that~$t^{-s}\in(0,1)$ for every~$t\in(1,+\infty)$. Thus, we pick~$\eta\in(0,1)$ and use the Fourier Inversion Theorem
	 for the radial function~$\mathcal{J}(x,1,1,\eta)$ (see e.g.~\cite[Chapter~\uppercase\expandafter{\romannumeral 2}]{MR31582}). In this way,
	 we find that
	\begin{equation}
		\begin{split}\label{sdfsdf}
			\mathcal{J}(x,1,1,\eta)
			&=\frac{(2\pi)^{\frac{n}{2}}}{|x|^{\frac{n}{2}-1}}\int_{0}^{+\infty} e^{-\frac{r^2}{1+\eta r^{2s}}} r^{\frac{n}{2}} J_{\frac{n}{2}-1}(|x|r)\, dr\\
			&=\frac{(2\pi)^{\frac{n}{2}}}{|x|^n}\int_{0}^{+\infty} e^{-\frac{t^2}{|x|^2+\eta t^{2s}|x|^{2-2s}}} t^{\frac{n}{2}} J_{\frac{n}{2}-1}(t)\, dt\\
			&=\frac{(2\pi)^{\frac{n}{2}}}{|x|^n}\int_{0}^{+\infty}e^{-\frac{t^2}{|x|^2+\eta t^{2s}|x|^{2-2s}}}\left(\frac{2t^{1+\frac{n}{2}}}{|x|^2+\eta t^{2s}|x|^{2-2s}}-\frac{2s\eta t^{1+2s+\frac{n}{2}}|x|^{2-2s}}{\left(|x|^2+\eta t^{2s}|x|^{2-2s}\right)^2}\right)J_{\frac{n}{2}}(t)\, dt,
		\end{split}
	\end{equation}
	where~$J_{v}$ denotes the Bessel function of first kind of order~$ v$.
	
We claim that
	\begin{equation}\label{claim}
		\lim\limits_{|x|\to+\infty}\sup_{\eta\in(0,1)}|x|^{n+2-2s}\mathcal{J}(x,1,1,\eta)=0.
	\end{equation}
Indeed, from~\eqref{sdfsdf} it follows that
$$
		|x|^{n+2-2s}\mathcal{J}(x,1,1,\eta)
		={(2\pi)^{\frac{n}{2}}}\text{ Re }\int_{0}^{+\infty}e^{-\frac{t^2}{|x|^2+\eta t^{2s}|x|^{2-2s}}}\left(\frac{2t^{1+\frac{n}{2}}}{|x|^{2s}+\eta t^{2s}}-\frac{2s\eta t^{1+2s+\frac{n}{2}}}{\left(|x|^{2s}+\eta t^{2s}\right)^2}\right)H^{1}_{\frac{n}{2}}(t)\, dt
$$		where~$ H^{(1)}_{\frac{n}{2}}(z)$ is the Bessel function of the third kind and Re~$A$ denotes the real part of~$A$.
		
	 We  consider the straight line
	 $$L_1:=\left\{z\in\mathbb{C}:\arg z=\frac{\pi}{6}\right\}$$ and, 
	using the Residue Theorem, we find that
	$$
			|x|^{n+2-2s}\mathcal{J}(x,1,1,\eta)
			=2{(2\pi)^{\frac{n}{2}}}\text{ Re }\int_{L_1} e^{-\frac{z^2}{|x|^2+\eta z^{2s}|x|^{2-2s}}}\left(\frac{z^{1+\frac{n}{2}}}{|x|^{2s}+\eta z^{2s}}-\frac{s\eta z^{1+2s+\frac{n}{2}}}{\left(|x|^{2s}+\eta z^{2s}\right)^2}\right)H^{1}_{\frac{n}{2}}(z)\, dz.
$$
 Furthermore, we observe that, for any~$\eta\in(0,1)$ and~$|x|>1$,  
	\begin{equation}
		\begin{split}\label{H11}
			&\left|\int_{L_1} e^{-\frac{z^2}{|x|^2+\eta z^{2s}|x|^{2-2s}}}\left(\frac{z^{1+\frac{n}{2}}}{|x|^{2s}+\eta z^{2s}}-\frac{s\eta z^{1+2s+\frac{n}{2}}}{\left(|x|^{2s}+\eta z^{2s}\right)^2}\right)H^{1}_{\frac{n}{2}}(z)\, dz\right|\\
			=&\left|\int_{0}^{+\infty}e^{-\frac{r^2e^{i\frac{\pi}{3}}}{|x|^2+\eta r^{2s}e^{i\frac{s\pi}{3}}|x|^{2-2s}}}\left(\frac{(re^{i\frac{\pi}{6}})^{1+\frac{n}{2}}}{|x|^{2s}+\eta (re^{i\frac{\pi}{6}})^{2s}}-\frac{s\eta (re^{i\frac{\pi}{6}})^{1+2s+\frac{n}{2}}}{\left(|x|^{2s}+\eta (re^{i\frac{\pi}{6}})^{2s}\right)^2}\right) H^{(1)}_{\frac{n}{2}}\left(re^{i\frac{\pi}{6}}\right)e^{i\frac{\pi}{6}}\, dr\right|\\
			\leqslant&\int_{0}^{+\infty}\left(\frac{r^{1+\frac{n}{2}}(|x|^{2s}+2r^{2s})}{|x|^{4s}}+\frac{sr^{1+2s+\frac{n}{2}}(|x|^{4s}+4r^{4s})}{|x|^{8s}}\right) \left|H^{(1)}_{\frac{n}{2}}\left(re^{i\frac{\pi}{6}}\right)\right|\, dr.\\
		\end{split}
	\end{equation}
	
	Employing~\eqref{estimate of H}, we  see that
	\begin{equation}\label{H1}
		\left|H^{(1)}_{\frac{n}{2}}\left(re^{i\frac{\pi}{6}}\right)\right|=\left|-ie^{-i\frac{n\pi}{4}}\int_{-\infty}^{+\infty}e^{ire^{i\frac{\pi}{6}}\frac{e^t+e^{-t}}{2}} e^{-\frac{n}{2}t}\, dt\right|\leqslant 2\int_{0}^{+\infty}e^{-\frac{r}{4}e^t}e^{\frac{nt}{2}}\,dt.
	\end{equation}
	By combining~\eqref{H11} with~\eqref{H1}, one obtains that, if~$\eta\in(0,1)$ and~$|x|>1$,
	\begin{equation*}
		\left|\int_{L_1} e^{-\frac{z^2}{|x|^2+\eta z^{2s}|x|^{2-2s}}}\left(\frac{z^{1+\frac{n}{2}}}{|x|^{2s}+\eta z^{2s}}-\frac{s\eta z^{1+2s+\frac{n}{2}}}{\left(|x|^{2s}+\eta z^{2s}\right)^2}\right)H^{(1)}_{\frac{n}{2}}(z)\, dz\right|
		\leqslant 
		2c_{n,s}\left(|x|^{-2s}+|x|^{-4s}\right)
	\end{equation*}
	for some constant~$c_{n,s}$. As a result, for any~$\epsilon>0$, there exists~$M>0$ independent of~$\eta$ such that, for every~$|x|>M$ and~$\eta\in(0,1)$,
	\[ \left|\int_{L_1} e^{-\frac{z^2}{|x|^2+\eta z^{2s}|x|^{2-2s}}}\left(\frac{z^{1+\frac{n}{2}}}{|x|^{2s}+\eta z^{2s}}-\frac{s\eta z^{1+2s+\frac{n}{2}}}{\left(|x|^{2s}+\eta z^{2s}\right)^2}\right)H^{(1)}_{\frac{n}{2}}(z)\, dz\right|<\frac{\epsilon}{2(2\pi)^{\frac{n}{2}}} \]
	which establishes~\eqref{claim}, as desired.
	
	By combining~\eqref{scaling  j} with ~\eqref{claim}, we know that  there exists~$M>0$ depending on~$n$ and~$s$ such that,
	when~$ t>1 $ and~$|x|>Mt^{\frac{1}{2}}$,	\begin{equation*}
		0\leqslant	\left|\mathcal{J}(x,t,1,1)\right|<\frac{t^{1-s}}{|x|^{n+2-2s}} .
	\end{equation*}
	Owing to this and~\eqref{split t}, we conclude that, for every~$|x|>M$,
	\begin{equation*}
		\mathcal{I}_1(x)=\int_{1}^{+\infty} e^{-t} \mathcal{J}(x,t,1,1)\,dt\leqslant \int_{1}^{\left|\frac{x}{M}\right|^{2}}e^{-t}\frac{t^{1-s}}{|x|^{n+2-2s}}\, dt+C_{n,s}\int_{\left|\frac{x}{M}\right|^{2}}^{+\infty}e^{-t}t^{-\frac{n}{2}}\, dt\leqslant c_1 \frac{1}{|x|^{n+2-2s}},
	\end{equation*}
	and, for every~$ |x|\leqslant M$,
	\begin{equation*}
		\mathcal{I}_1(x)=\int_{1}^{+\infty} e^{-t} \mathcal{J}(x,t,1,1)\,dt\leqslant C_{n,s} \int_{\left|\frac{x}{M}\right|^{2}}^{+\infty}e^{-t}t^{-\frac{n}{2}}\, dt\leqslant c_2\frac{1}{|x|^{n-2+2s}}.
	\end{equation*}
{F}rom the last two formulas we obtain that~$\mathcal{ I}_1\in L^1(\mathbb{R}^n)$, as desired.
	\smallskip
	
{Step~2.  } We now prove that~$ \mathcal{I}_2\in L^1(\mathbb{R}^n) $. This will be a byproduct of the rescaling property~\eqref{scaling  j}. The full argument goes as follows.
	Since~$t\in(0,1)$, one has that~$t^{\frac{s}{1-s}}\in(0,1)$. We thus 
	use again the Fourier Inversion Theorem for the radial function~$\mathcal{J}(x,1,\eta,1)$, finding that
	\begin{equation*}
		\begin{split}
			\mathcal{J}(x,1,\eta,1)&=\int_{\mathbb{R}^n}e^{-\frac{|\xi|^2}{\eta+|\xi|^{2s}}}e^{2\pi ix\cdot \xi}\,d\xi\\
			&=\frac{(2\pi)^{\frac{n}{2}}}{|x|^n}\int_{0}^{+\infty}e^{-\frac{t^2}{\eta|x|^2+ t^{2s}|x|^{2-2s}}}\left(\frac{2t^{1+\frac{n}{2}}}{\eta|x|^2+ t^{2s}|x|^{2-2s}}-\frac{2s t^{1+2s+\frac{n}{2}}|x|^{2-2s}}{\left(\eta|x|^2+ t^{2s}|x|^{2-2s}\right)^2}\right)J_{\frac{n}{2}}(t)\, dt.
		\end{split}
	\end{equation*}
We now claim that 
	\begin{equation}\label{claim 2}
		\lim\limits_{|x|\to+\infty}\sup_{\eta\in(0,1)}|x|^{n+1-s}\mathcal{J}(x,1,\eta,1)=0.
	\end{equation}
	Indeed, using a similar argument as in the proof of {Step~1}, for any~$\eta\in(0,1)$ and~$|x|>1$, one has
	\begin{equation*}
		\begin{split}
			\left||x|^{n+1-s}\mathcal{J}(x,1,\eta,1)\right|&\leqslant2(2\pi)^{\frac{n}{2}}\left|\int_{L_1} e^{-\frac{z^2}{\eta|x|^2+ z^{2s}|x|^{2-2s}}}\left(\frac{z^{1+\frac{n}{2}}}{\eta|x|^{1+s}+ z^{2s}|x|^{1-s}}-\frac{s z^{1+2s+\frac{n}{2}}}{\left(\eta|x|^{\frac{1+3s}{2}}+ z^{2s}|x|^{\frac{1-s}{2}}\right)^2}\right)H^{1}_{\frac{n}{2}}(z)\, dz\right|\\
			&\leqslant 2(2\pi)^{\frac{n}{2}}\int_{0}^{+\infty}\left(\frac{6r^{1+\frac{n}{2}-2s}}{|x|^{1-s}}+\frac{17sr^{1-2s+\frac{n}{2}}}{|x|^{1-s}}\right) \left|H^{(1)}_{\frac{n}{2}}\left(re^{i\frac{\pi}{6}}\right)\right|\, dr. 
		\end{split}
	\end{equation*}
	
	Owing to~\eqref{H1}, for any~$\eta\in(0,1)$ and~$|x|>1$, we infer that
$$		\left|\int_{L_1} e^{-\frac{z^2}{\eta|x|^2+ z^{2s}|x|^{2-2s}}}\left(\frac{z^{1+\frac{n}{2}}}{\eta|x|^{2s}+ z^{2s}|x|^{1-s}}-\frac{s z^{1+2s+\frac{n}{2}}}{\left(\eta|x|^{\frac{1+3s}{2}}+ z^{2s}|x|^{\frac{1-s}{2}}\right)^2}\right)H^{1}_{\frac{n}{2}}(z)\, dz\right|
		\leqslant 
		c_{n,s}|x|^{s-1}.
$$
	Hence, for any~$\epsilon>0$, there exists~$M>0$ independent of~$\eta$ such that, for every~$|x|>M$ and~$\eta\in(0,1)$,
	\[ \left|\int_{L_1} e^{-\frac{z^2}{\eta|x|^2+ z^{2s}|x|^{2-2s}}}\left(\frac{z^{1+\frac{n}{2}}}{\eta|x|^{2s}+ z^{2s}|x|^{1-s}}-\frac{s z^{1+2s+\frac{n}{2}}}{\left(\eta|x|^{\frac{1+3s}{2}}+ z^{2s}|x|^{\frac{1-s}{2}}\right)^2}\right)H^{1}_{\frac{n}{2}}(z)\, dz\right|<\frac{\epsilon}{2(2\pi)^{\frac{n}{2}}}, \]
	which establishes~\eqref{claim 2}.
	
	{F}rom~\eqref{scaling  j} and~\eqref{claim 2},  we can find~$M>0$, depending
	 only on~$n$ and~$s$, such that, if~$|x|>Mt^{\frac{1}{2-2s}}$
	and~$ t\in(0,1)$,
	\begin{equation*}
		0\leqslant	\left|\mathcal{J}(x,t,1,1)\right|<\frac{t^{\frac{1}{2}}}{|x|^{n+1-s}}.
	\end{equation*}
	Owing to this and~\eqref{split t}, we find that, for every~$|x|>M$,
\begin{equation*}
	\mathcal{I}_2(x)=\int_{0}^{1} e^{-t} \mathcal{J}(x,t,1,1)\,dt\leqslant \int_{0}^{\left|\frac{x}{M}\right|^{2-2s}}e^{-t}\frac{t^{\frac{1}{2}}}{|x|^{n+1-s}}\, dt	\leqslant c_1 \frac{1}{|x|^{n+1-s}},
\end{equation*}
and, for every~$|x|\leqslant M$,
\begin{equation*}
	\mathcal{I}_2(x)=\int_{0}^{1} e^{-t} \mathcal{J}(x,t,1,1)\,dt\leqslant\int_{0}^{\left|\frac{x}{M}\right|^{2-2s}}e^{-t}\frac{t^{\frac{1}{2}}}{|x|^{n+1-s}}\, dt+ C_{n,s} \int_{\left|\frac{x}{M}\right|^{2-2s}}^{1}e^{-t}t^{-\frac{n}{2-2s}}\, dt\leqslant c_2\frac{1}{|x|^{n-1+s}}.
\end{equation*}
Combining the last two formulas, one deduces that~$\mathcal{I}_2\in L^1(\mathbb{R}^n)$ as well.
\end{proof}

\subsubsection{$C^{0,\mu}$-regularity of weak solutions}\label{subsec:C^{1,alpha}-regularity of weak solution1}

We dedicate this part to show the $C^{0,\mu}$-regularity of weak solutions based on  Lemma~\ref{lemma embedding of X^p} and the usual iteration technique.

For this, we first make the following observation:

\begin{Lemma}\label{STAHNMD}
For all~$t\in\left(-\infty,\frac{n}2\right)$, let
$$ \psi(t):=\frac{nt}{n-2t}.$$

Then, for every~$t_0\in\left(\frac{n}2-\frac{n}{2p},\frac{n}2\right)$
there exists~$j_0=j_0(t_0)\in\N$ with~$j_0\ge1$ such that
$$ \stackrel{{j_0 {\mbox{ times}}}}{\psi\circ\dots\circ\psi}(t_0)
\ge\frac{n}2>\stackrel{{j {\mbox{ times}}}}{\psi\circ\dots\circ\psi}(t_0)\qquad{\mbox{for all~$j<j_0$}}.$$
\end{Lemma}

\begin{proof}
Denote by
$$ \psi_j(t_0):= \stackrel{{j {\mbox{ times}}}}{\psi\circ\dots\circ\psi}(t_0)$$
and suppose by contradiction that, for all~$j\ge1$,
$$ \psi_j(t_0)<\frac{n}2.$$
 
Also, notice that~$\psi(t)\ge t$, and therefore, for all~$j\ge1$,
we have that~$  \psi_{j+1}(t_0)\ge \psi_j(t_0)$.
Hence, the following limit exists
$$ \ell:=\lim_{j\to+\infty}\psi_j(t_0)$$
and moreover, by construction, $\ell\in\left[t_0,\frac{n}2\right]$.

In addition,
$$ \psi_{j+1}(t_0)=\frac{n\,\psi_j(t_0)}{n-2\psi_j(t_0)}$$
and thus, taking the limit in~$j$,
$$\ell=\frac{n\ell}{n-2\ell}.$$
Solving for~$\ell$, we find that~$\ell=0$, which gives the desired contradiction.
\end{proof}

\begin{proof}[Proof of Theorem~\ref{Theorem :holder regularity}]
We consider a sequence of cut off functions~$ \phi_j\in C^\infty_0(\mathbb{R}^n)  $ for any~$j\geqslant 1$ satisfying 
\begin{equation}\label{nbvrer438tphiuno}
		\phi_j \equiv 1 \text{  in } B_{1/2^{2j-2}},\qquad
		\text{supp}(\phi_j)\subset B_{1/2^{2j-3}} \qquad{\mbox{and}}\qquad
		0\leqslant \phi_j\leqslant 1 \text{  in } \mathbb{R}^n.
	\end{equation}
Let~$u_j$ solve
\begin{equation}\label{equazioneujei}
	-\Delta u_j+(-\Delta )^{s}u_j+u_j=\phi_j u^p\qquad \text{in }\mathbb{R}^n. 
\end{equation}
Then 
\begin{equation*}
	-\Delta (u-u_j)+(-\Delta )^{s}(u-u_j)+(u-u_j)=(1-\phi_j) u^p\qquad \text{in }\mathbb{R}^n. 
\end{equation*}
Moreover,
\begin{equation*}
	u-u_j=\mathcal{ K}\ast \Big((1-\phi_j) u^p\Big)
\end{equation*}
where~$\mathcal{ K}$ is given by~\eqref{definition of K}.  

Let now~$q_0$ be either equal to the critical exponent~$\frac{2n}{n-2}$ if~$n>2$ or any real number in~$(p,+\infty)$ if~$n=2$. Let also~$\theta:=q_0/(q_0-p)>1$.

Moreover, we set~$\gamma_1:=q_{0}/p$. We observe that~$\gamma_1>
\frac{n}2-\frac{n}{2p}$, thanks to~\eqref{duiwegfufvckuqfhoq496712123hew}.
Hence, in light of Lemma~\ref{STAHNMD}, we can define
$$ j_\star:=\begin{dcases} 0 & {\mbox{ if }}\gamma_1\ge\frac{n}2,\\
j_0(\gamma_1) &{\mbox{ if }}\gamma_1<\frac{n}2,
\end{dcases}$$
and, for every~$j\in\{0,\dots,j_\star\}$, 
$$\gamma_{j+1}:=\stackrel{{j {\mbox{ times}}}}{\psi\circ\dots\circ\psi}(\gamma_1).$$
In this way, we have that
\begin{equation}\label{inviewof3i2r6gfkjt987654}
\gamma_{j_\star+1}\ge\frac{n}2.
\end{equation}
Also, for every~$j\in\{0,\dots,j_\star\}$, we set~$q_j:=p\gamma_{j+1}$.

In this setting, we claim that, for every~$j\in\{1,\dots,j_\star+1\}$,
\begin{equation}\label{baseinduz904}
\| u_j\|_{W^{2,\gamma_j}(\R^n)}\le C_j\,\Big( \|u\|^{p^{j}}_{H^1(\mathbb{R}^n)}+\|u\|^{p}_{H^1(\mathbb{R}^n)}\Big).\end{equation}
The proof of this claim is by induction over~$j$.

We first check that~\eqref{baseinduz904} holds true when~$j=1$.
For this sake, we notice that~$\gamma_1=q_0/p$, and therefore,
since~$u\in L^{q_0}(\mathbb{R}^n)$, one has that~$\phi_1 u^p\in L^{\gamma_1}(\mathbb{R}^n)$.
{F}rom this and Lemma~\ref{lemma embedding of X^p},
applied to the equation for~$u_1$ in~\eqref{equazioneujei},
it then follows that~$ u_1\in W^{2,\gamma_1}(\mathbb{R}^n)$ and 
\begin{equation}\label{yq738rtgfefgwuty}
\|u_1\|_{W^{2,\gamma_1}(\mathbb{R}^n)} 
\le C \|\phi_1 u^p\|_{L^{\gamma_1}(\mathbb{R}^n)}
= C \|\phi_1 u^p\|_{L^{\gamma_1}(B_2)}
\leqslant C
\|u\|^{p}_{L^{q_0}(B_2)}
\le C \|u\|^{p}_{H^1(\mathbb{R}^n)}
,\end{equation}
up to renaming~$C>0$.
The last inequality in~\eqref{yq738rtgfefgwuty} uses 
the Sobolev Embedding (see also~\cite[Corollary 9.11]{MR2759829} when~$n=2$).
This establishes~\eqref{baseinduz904} when~$j=1$.

Now we suppose that~\eqref{baseinduz904} holds true for all the indexes in~$\{1,\dots,j\}$ with~$j\le j_\star$ and we prove it for the index~$j+1$.
For this, we observe that, since~$\gamma_j<n/2$ by construction,
we have that~$q_j=\frac{n\gamma_j}{n-2\gamma_j}$ and~$\gamma_{j+1}=q_j/p$.

Moreover, owing to the H\"{o}lder inequality (used here with exponents~$q_0/p$ and~$\theta$)
and the smoothness of~$\mathcal{ K}$ away from the origin and its decay
(recall~(a) and~(b) of Theorem~\ref{th properties of k}), we have that,
for every~$x\in B_{1/2^{2j-1}}$,
\begin{equation*}
|u(x)-u_j(x)|=\Big|\mathcal{ K}\ast \Big((1-\phi_j) u^p\Big)(x)\Big|\leqslant\|\mathcal{ K}\|_{L^\theta(\mathbb{R}^n\setminus B_{1/2^{2j-2}})}\|u\|^p_{{L^{q_0}(\mathbb{R}^n)}}\leqslant C\|u\|^p_{L^{q_0}(\mathbb{R}^n)}.
\end{equation*}
This and the triangle inequality give that
\begin{equation*}
\|u\|_{L^{q_j}( B_{1/2^{2j-1}})}\le C \|u\|^p_{L^{q_0}(\mathbb{R}^n)}+
\|u_j\|_{L^{q_j}( B_{1/2^{2j-1}})}.
\end{equation*}
{F}rom this and the Sobolev Embedding, we deduce that
\begin{equation*}
\|u\|_{L^{q_j}( B_{1/2^{2j-1}})}\le C\left( \|u\|^p_{L^{q_0}(\mathbb{R}^n)}+
\|u_j\|_{W^{2,\gamma_j}( \R^n)}\right),
\end{equation*}
up to renaming~$C>0$, depending on $n$, $s$, $p$ and~$j$.
Hence, using the inductive assumption,
\begin{equation*}
\|u\|_{L^{q_j}( B_{1/2^{2j-1}})}\le C\left( \|u\|^p_{L^{q_0}(\mathbb{R}^n)}+
\|u\|^{p^{j}}_{H^1(\mathbb{R}^n)}+\|u\|^{p}_{H^1(\mathbb{R}^n)}
\right).
\end{equation*}
Using again the Sobolev Embedding and renaming~$C$ once more,
\begin{equation*}
\|u\|_{L^{q_j}( B_{1/2^{2j-1}})}\le C\left( 
\|u\|^{p^{j}}_{H^1(\mathbb{R}^n)}+\|u\|^{p}_{H^1(\mathbb{R}^n)}
\right).
\end{equation*}

Hence, we have that~$\phi_{j+1} u^p\in L^{\gamma_{j+1}}(\mathbb{R}^n) $
and
\begin{eqnarray*}&& \|\phi_{j+1}u^p\|_{L^{\gamma_{j+1}}(\mathbb{R}^n)}\leqslant C \|u\|^p_{L^{q_{j}}(B_{1/(2^{2j-1})})} \leqslant C\left(\|u\|^{p^j}_{H^1(\mathbb{R}^n)}+\|u\|^{p}_{H^1(\mathbb{R}^n)}\right)^p
\\&&\qquad\qquad
\le C\left(\|u\|^{p^{j+1}}_{H^1(\mathbb{R}^n)}+\|u\|^{p^2}_{H^1(\mathbb{R}^n)}\right)
\le C\left(\|u\|^{p^{j+1}}_{H^1(\mathbb{R}^n)}+\|u\|^{p}_{H^1(\mathbb{R}^n)}\right).\end{eqnarray*}
We can therefore use
Lemma~\ref{lemma embedding of X^p}
for the equation~\eqref{equazioneujei} for~$u_{j+1}$.
In this way, we obtain
that~$ u_{j+1}\in W^{2,\gamma_{j+1}}(\mathbb{R}^n)$ and
\begin{equation*}
 \|u_{j+1}\|_{W^{2,\gamma_{j+1}}(\mathbb{R}^n)}\leqslant C\|\phi_{j+1}u^p\|_{L^{\gamma_{j+1}}(\mathbb{R}^n)}\leqslant  C\left(\|u\|^{p^{j+1}}_{H^1(\mathbb{R}^n)}+\|u\|^{p}_{H^1(\mathbb{R}^n)}\right),
\end{equation*}
for some constant~$C>0$, depending on $n$, $s$, $p$ and~$j$.
This completes the proof of the claim in~\eqref{baseinduz904}. 

Using~\eqref{baseinduz904} with~$j_\star+1$, we find that
\begin{equation}\label{qwww}
\| u_{j_\star+1}\|_{W^{2,\gamma_{j_\star+1}}(\R^n)}\le C\,
\Big( \|u\|^{p^{{j_\star+1}}}_{H^1(\mathbb{R}^n)}+\|u\|^{p}_{H^1(\mathbb{R}^n)}\Big).\end{equation}

Now, in view of~\eqref{inviewof3i2r6gfkjt987654}, we distinguish two cases,
either~$\gamma_{j_\star+1}> {n}/{2}$
or~$\gamma_{j_\star+1}= {n}/{2}$.

If~$\gamma_{j_\star+1}> {n}/{2}$
we use the Sobolev Embedding Theorem, by choosing
\begin{equation*}
	\mu\in\left(0,\,\min\left\{1, 2-\frac{n}{\gamma_{j_0+1}}\right\}\right),
\end{equation*}
and we see that~$u_{j_\star+1}\in C^{0,\mu}(\mathbb{R}^n) $. In particular, from~\eqref{qwww}, it follows that
\begin{equation}\label{vvvdfv}
	\|u_{j_\star+1}\|_{C^{0,\mu}(\mathbb{R}^n)}\leqslant C\left(\|u\|^{p^{j_\star+1}}_{H^1(\mathbb{R}^n)}+\|u\|^{p}_{H^1(\mathbb{R}^n)}\right).
\end{equation}
Moreover, using the smoothness of~$\mathcal{ K}$ away from
the origin, when~$|x|<{1}/{2^{2j_\star+1}}$ one has that
\begin{equation*}
	\begin{split}
		&|\nabla(u-u_{j_\star+1})|\leqslant \int_{\mathbb{R}^{n}}|\nabla\mathcal{ K}(x-y)||(1-\phi_{{j_\star+1}}(y))u^p(y)|\, dy\\
		&\qquad\leqslant \|\mathcal{ \nabla K}\|_{L^\theta(\R^n\setminus
		B_{{1}/({2^{2j_\star+1}})})}\|u\|^p_{L^{q_0}(\mathbb{R}^n)}\leqslant C\|u\|^p_{L^{q_0}(\mathbb{R}^n)}.
	\end{split}
\end{equation*}
By combining this and~\eqref{vvvdfv}, we deduce that, for any~$x$, $y\in B_{\frac{1}{2^{2j_\star+1}}}$ with~$x\neq y$,
\begin{equation*}
	\begin{split}
		\frac{|u(x)-u(y)|}{|x-y|^\mu}&=\frac{|(u-u_{j_\star+1})(x)-(u-u_{j_\star+1})(y)+u_{j_\star+1}(x)-u_{j_\star+1}(y)|}{|x-y|^\mu}\\
		&\leqslant \sup\limits_{\xi\in B_{\frac{1}{2^{2j_\star+1}}}}|\nabla (u-u_{j_\star+1})(\xi)||x-y|^{1-\mu}+\|u_{j_\star+1}\|_{C^{0,\mu}(\mathbb{R}^n)}\\
		&\leqslant C \left(\|u\|^{p^{j_\star+1}}_{H^1(\mathbb{R}^n)}+\|u\|^{p}_{H^1(\mathbb{R}^n)}\right)
	\end{split}
\end{equation*}
and  
$$	|u(x)|\leqslant |u(x)-u_{j_\star+1}(x)|+|u_{j_\star+1}(x)|\leqslant C
\left(\|u\|^{p^{j_\star+1}}_{H^1(\mathbb{R}^n)}+\|u\|^{p}_{H^1(\mathbb{R}^n)}\right).
$$
The ball~$B_{\frac{1}{2^{2j_\star+1}}}$ is centred at the origin, but we may arbitrarily move it around~$\mathbb{R}^n$. Covering~$\mathbb{R}^n$ with these balls, we conclude that~$u\in C^{0,\mu}(\mathbb{R}^n)$ for some~$\mu\in(0,1)$.

This completes the desired result when~$\gamma_{j_\star+1}> {n}/{2}$
and we now focus on the case~$\gamma_{j_\star+1}={n}/{2}$.

Recalling~\eqref{w 2,p} and using~\cite[Theorem~3, page~135]{MR290095}, we see that~$u_{j_\star+1}\in \mathcal{W}^{2,\frac{n}{2}} $. Furthermore,
owing to~\cite[formula~(40), page~135]{MR290095}, we know that~$ u_{j_\star+1}\in  \mathcal{W}^{1+\frac{4}{5},\frac{n}{2}}$,
where
$$
\mathcal{W}^{1+\frac{4}{5},\frac{n}{2}}:=\left\{u\in L^{\frac{n}{2}}(\mathbb{R}^n)\;{\mbox{ s.t. }}\;\mathcal{F}^{-1}\left( (1+|\xi|^2)^{\frac{1+4/5}{2}}\hat{u}\right)\in L^{\frac{n}{2}}(\mathbb{R}^n)\right\}.
$$
Thus, from~\cite[Theorem~3.2]{MR3002595}, it follows that~$u_{j_\star+1}\in L^{5n}(\mathbb{R}^n)$ and
\begin{equation}\label{5n}
	\|u_{j_\star+1}\|_{L^{5n}(\mathbb{R}^n)}\leqslant C	\|u_{j_\star+1}\|_{W^{2,\frac{n}{2}}(\mathbb{R}^n)},
\end{equation}
for some constant~$C>0$. 

Using the decay properties of~$\mathcal{ K}$ as given by Theorem~\ref{th properties of k}, we obtain that,
for any~$x\in B_{1/(2^{2j_\star+1})}$,
\begin{equation}\label{u-u_j_0+1}
	|u-u_{j_\star+1}|=\left|\mathcal{ K}\ast \Big((1-\phi_{j_0+1}) u^p\Big)\right|\leqslant\|\mathcal{ K}\|_{L^{\theta}(\R^n\setminus B_{1/(2^{2j_\star+1})})}\|u\|^p_{{L^{q_{0}}(\mathbb{R}^n)}}\leqslant C\|u\|^p_{L^{q_{0}}(\mathbb{R}^n)}.
\end{equation}
By combining~\eqref{qwww}, \eqref{5n} and~\eqref{u-u_j_0+1}, we have that
$$	\|u\|_{L^{5n}(B_{1/(2^{2j_\star+1})})} \leqslant C \left(\|u\|^{{p}^{j_\star+1}}_{H^1(\mathbb{R}^n)}+\|u\|^{p}_{H^1(\mathbb{R}^n)}\right).
$$
This gives that~$\phi_{j_\star+2} u^p\in L^{5n/p}(\mathbb{R}^n) $
and therefore,
by Lemma~\ref{lemma embedding of X^p}, we deduce that~$ u_{j_\star+2}\in W^{2,\frac{5n}{p}}(\mathbb{R}^n)$. Since~$p<\frac{n+2}{n-2}$, we observe that~$\frac{5n}{p}>\frac{n}{2}$. Hence, using the Sobolev Embedding Theorem, we have that~$u_{j_\star+2}\in C^{0,\mu}(\mathbb{R}^n) $, 
with
$$
	\mu\in\left(0\,\min\left\{1, 2-\frac{p}5\right\}\right),
$$
and 
\begin{equation}\label{vvvdfvgv}
	\|u_{j_\star+2}\|_{C^{0,\mu}(\mathbb{R}^n)}\leqslant C
	\|u_{j_\star+2}\|_{W^{2,5n/p}(\mathbb{R}^n)} \leqslant C
	\left(\|u\|^{p^{j_\star+2}}_{H^1(\mathbb{R}^n)}+\|u\|^{p}_{H^1(\mathbb{R}^n)}\right).
\end{equation}

Moreover, by the smoothness of~$\mathcal{ K}$,  for every~$|x|<{1}/{2^{2j_\star+3}}$, we have that
\begin{equation*}
		|\nabla(u-u_{j_\star+2})|\leqslant \|\mathcal{ \nabla K}\|_{L^\theta(
		\R^n\setminus B_{{1}/({2^{2j_\star+3}})})}\|u\|^p_{L^{q_0}(\mathbb{R}^n)}\leqslant C\|u\|^p_{L^{q_0}(\mathbb{R}^n)}.
\end{equation*}
By combining this and~\eqref{vvvdfvgv}, we have that,
for any~$x$, $y\in B_{\frac{1}{2^{2j_\star+3}}}$ with~$x\neq y$,
$$		\frac{|u(x)-u(y)|}{|x-y|^\mu}
		\leqslant C\left(\|u\|^{p^{j_\star+2}}_{H^1(\mathbb{R}^n)}+\|u\|^{p}_{H^1(\mathbb{R}^n)}\right)
$$
and  
$$	|u(x)|\leqslant |u(x)-u_{j_\star+2}(x)|+|u_{j_\star+2}(x)|\leqslant C
\left(\|u\|^{p^{j_\star+1}}_{H^1(\mathbb{R}^n)}+\|u\|^{p}_{H^1(\mathbb{R}^n)}\right).
$$
	  The ball~$B_{\frac{1}{2^{2j_0+3}}}$ is centred at the origin, but we may arbitrarily move it around~$\mathbb{R}^n$. Covering~$\mathbb{R}^n$ with these balls, we  conclude that~$u\in C^{0,\mu}(\mathbb{R}^n)$ for some~$\mu\in(0,1)$.
	  This completes the proof of Theorem~\ref{Theorem :holder regularity}.
\end{proof}

\subsection{$C^{1,\alpha}$-regularity of weak solutions}\label{sec:C^{1,alpha}-regularity of weak solution2}
We aim here to establish the 
$ C^{1,\alpha} $-regularity of weak solutions of~\eqref{Maineq} based on the $L^p$-theory and the smoothness of the kernel~$\mathcal{ K}$ defined by~\eqref{definition of K}.

\begin{proof}[Proof of Theorem~\ref{th:regularity}]
We consider the cut off function~$ \phi_1\in C^\infty_0(\mathbb{R}^n) $ satisfying 
the properties in~\eqref{nbvrer438tphiuno} with~$j=1$.
	Let~$u_1$ solve
	\begin{equation*}
		-\Delta u_1+(-\Delta )^{s}u_1+u_1=\phi_1 u^p\qquad \text{in }\mathbb{R}^n. 
	\end{equation*}
	Then 
	\begin{equation*}
		-\Delta (u-u_1)+(-\Delta )^{s}(u-u_1)+(u-u_1)=(1-\phi_1) u^p\qquad \text{in }\mathbb{R}^n. 
	\end{equation*}
	Moreover, we deduce that 
	$$	u-u_1=\mathcal{ K}\ast \Big((1-\phi_1) u^p\Big).
	$$
	
	We recall that~$u\in L^\infty(\mathbb{R}^n)$, thanks to
Theorem~\ref{Theorem :holder regularity}.

Now, owing to the H\"{o}lder inequality and the smoothness of~$\mathcal{ K}$ away from the origin and its decay
(recall Theorem~\ref{th properties of k}), we have that, for every~$x\in B_{{1}/{2}}$,
	\begin{equation*}
		|u(x)-u_1(x)|=\Big|\mathcal{ K}\ast \Big((1-\phi_1) u^p\Big)(x)\Big|
		\leqslant\|\mathcal{ K}\|_{L^1(\R^n\setminus B_{1/2})}\|u\|^p_{L^\infty(\mathbb{R}^n)}\leqslant c_{n,s,p}\|u\|^p_{L^\infty(\mathbb{R}^n)},
	\end{equation*}
	\begin{equation*}
			|\nabla(u-u_1)(x)|\leqslant \int_{\mathbb{R}^{n}}|\nabla\mathcal{ K}(x-y)||(1-\phi_1(y))u^p(y)|\, dy
			\leqslant \|\mathcal{ \nabla K}\|_{L^1(\R^n\setminus B_{1/2})}\|u\|^p_{L^\infty(\mathbb{R}^n)}\leqslant c_{n,s,p}\|u\|^p_{L^\infty(\mathbb{R}^n)}
	\end{equation*}
and 
	\begin{equation}\label{u-u_1 D^2}
		|D^2(u-u_1)(x)|\leqslant \int_{\mathbb{R}^{n}}|{D^2 K}(x-y)||(1-\phi_1(y))u^p(y)|\, dy\leqslant \|{ D^2 K}\|_{L^1(\R^n\setminus B_{1/2})}\|u\|^p_{L^\infty(\mathbb{R}^n)}\leqslant c_{n,s}\|u\|^p_{L^\infty(\mathbb{R}^n)}.
	\end{equation}
 In the light of these estimates, we can focus on the regularity of~$u_1$.
 For this, we observe that, for any~$q\geqslant 1$,
	\begin{equation}\label{q}
		\|\phi_1 u^p\|_{L^q(\R^n)}
		=\left(\int_{\mathbb{R}^n}|\phi_1(x) u^p(x)|^q\, dx\right)^{1/q}
		\leqslant \left(\frac{\pi^{\frac{n}{2}}}{\frac{n}{2}\Gamma(\frac{n}{2})}\right)^{1/q}\|u\|^p_{L^\infty(\mathbb{R}^n)}
		\leqslant \left(1+\frac{\pi^{\frac{n}{2}}}{\frac{n}{2}\Gamma(\frac{n}{2})}\right)\|u\|^p_{L^\infty(\mathbb{R}^n)}.
	\end{equation}
	That is, $\phi_1 u^p\in L^q(\mathbb{R}^n)$ for any~$q\geqslant 1$. {F}rom Lemma~\ref{lemma embedding of X^p}, it then follows that~$ u_1\in W^{2,q}(\mathbb{R}^n)$ for any~$q\geqslant 1$, and in particular one can take any~$q>n$. Accordingly, 
from the Sobolev Embedding, we deduce that~$ u_1\in C^{1,1-\frac{n}{q}}(\mathbb{R}^n) $.

Thus, denoting by~$\alpha:=1-\frac{n}{q}$, from Lemma~\ref{lemma embedding of X^p} and~\eqref{q},   
	it follows that   
	\begin{equation*}
		\|u_1\|_{C^{1,\alpha}(\mathbb{R}^n)}\leqslant C_{n} \|u_1\|_{W^{2,q}(\mathbb{R}^n)}\leqslant C_{n,s} \|\phi_1 u^p\|_{L^q(\R^n)}  \leqslant C_{n,s}\|u\|^p_{L^\infty(\mathbb{R}^n)}
	\end{equation*}
	for some constant~$C_{n,s}$ independent of~$\alpha$.
	
	{F}rom this and~\eqref{u-u_1 D^2}, we deduce that for any~$x$, $y\in B_{1/2}$ with~$x\neq y$,
	\begin{equation*}
		\begin{split}
			\frac{|Du(x)-Du(y)|}{|x-y|^\alpha}&=\frac{|D(u-u_1)(x)-D(u-u_1)(y)+Du_1(x)-Du_1(y)|}{|x-y|^\alpha}\\
			&\leqslant \sup\limits_{\xi\in B_{1/4}}|D^2(u-u_1)(\xi)||x-y|^{1-\alpha}+\|u_1\|_{C^{1,\alpha}(\mathbb{R}^n)}\\
			&\leqslant c_{n,s,p} \|u\|^p_{L^\infty(\mathbb{R}^n)}.
		\end{split}
	\end{equation*}
Therefore, we have that~$u\in C^{1,\alpha}(\overline{B_{1/2}}) $
for any~$\alpha\in(0,1)$.

We point out that the ball~$B_{1/2}$ is centred at the origin, but we may arbitrarily move it around~$\mathbb{R}^n$. Covering~$\mathbb{R}^n$ with these balls, we conclude that~$u\in C^{1,\alpha}(\mathbb{R}^n)$ for any~$\alpha\in(0,1)$. This completes the proof of Theorem~\ref{th:regularity}.
\end{proof}

\subsection{$C^{2,\alpha}$-regularity of weak solutions}\label{sec:C2}

The goal of this section is to establish the $C^{2,\alpha}$-regularity result for
solutions of problem~\eqref{Maineq},
as stated in Theorem~\ref{th C^2,alpha interior without boundary condition}.
For this, we 
combine a suitable truncation argument for
the solution~$ u $ with the $ C^{1,\alpha} $-regularity argument.

We point out that Theorem~\ref{th C^2,alpha interior without boundary condition} can be obtained by appropriately  modifying~\cite[Theorem~1.6]{SVWZ23}. For the convenience of the reader, we sketch the proof in the following subsections.

To begin with, we introduce some notations.
For~$\alpha\in(0,1)$, 
$k\in{\mathbb{N}}$, $x_0\in B_{3/4}$ and~$R\in\left(0,\frac1{20}\right)$,
we denote the interior norms as follows:
	\begin{eqnarray*}	[u]_{\alpha;B_{R}(x_0)}&:=&\sup\limits_{x,y\in B_{R}(x_0)}\frac{|u(x)-u(y)|}{|x-y|^\alpha},\\
		|u|^\prime_{k;B_{R}(x_0)}&:=&\sum\limits_{j=0}\limits^{k}R^j\|D^j u\|_{L^\infty(B_{R}(x_0))}\\
{\mbox{and }}\quad
|u|^\prime_{k,\alpha;B_{R}(x_0)}&:=&|u|^\prime_{k;B_{R}(x_0)}+R^{k+\alpha}[D^ku]_{\alpha;B_{R}(x_0)}.
	\end{eqnarray*}

\subsubsection{{A mollifier technique and a truncation argument}}\label{sub A priori C{2,alpha}-estimate}

Let~$u\in H^1(\mathbb{R}^n)$ solve~\eqref{Maineq},
and let~$\eta_\epsilon$ be a standard mollifier. For every~$x\in\mathbb{R}^n$,
$R\in\left(0,\frac1{20}\right)$ and~$\epsilon\in(0,R)$, we denote by
\begin{equation*}
	u_\epsilon(x):=(\eta_\epsilon\ast u)(x)= \int_{|y|\leqslant \epsilon} \eta_\epsilon(y)u(x-y)\, dy.
\end{equation*}
Also, we set\begin{equation*}
 g_\epsilon:=\eta_\epsilon \ast (-u+u^p).
\end{equation*}
Then, we have that
\begin{equation}\label{u epsilon}
	-\Delta u_\epsilon+(-\Delta)^s u_\epsilon= g_\epsilon \qquad \text{  in } B_1.
\end{equation}

Moreover, the following regularity estimates on~$u_\epsilon$ and~$g_\epsilon$ follow as a direct consequence of their definitions:
 
\begin{Lemma}\label{pro u epsilon1} Let~$u\in L^\infty(\R^n)$.
Then, $u_\epsilon\in L^\infty(\mathbb{R}^n)$ and	$$
		\|u_\epsilon\|_{L^\infty(\mathbb{R}^n)}\leqslant \|u\|_{L^\infty(\mathbb{R}^n)}.
	$$
	
If in addition~$u\in C^1(\mathbb{R}^n)$, then, for every~$y\in\R^n$,
	\begin{equation*}
		\|g_\epsilon\|_{C^{\alpha}(\overline{B_{1}(y) })}\leqslant c_p\left(\|u\|_{C^1(\mathbb{R}^n)}+\|u\|_{C^1(\mathbb{R}^n)}^p\right).	
	\end{equation*}
\end{Lemma}

	We now use a cut off argument for~$u_\epsilon$ to get the a $C^{2,\alpha}$-estimate for~$u_\epsilon$. 
	Consider a cut off function~$ \phi\in C^\infty_0(\mathbb{R}^n)$ satisfying 
	\begin{equation*}
			\phi \equiv 1 \text{  in } B_{{3}/{2}},\qquad
			\text{supp}(\phi)\subset B_{2}\qquad{\mbox{and}}\qquad
			0\leqslant \phi\leqslant 1 \text{  in } \mathbb{R}^n
	\end{equation*}
	and let 
$$
		\phi^R(x):=\phi\left(\frac{x-x_0}{R}\right).
$$
	We point out that 
	\[ B_{4R}(x_0)\subset B_1 \quad{\mbox{ and }}\quad \text{supp}(\phi^R)\subset B_{2R}(x_0). \]
	
	With this notation, one obtains the following result:
	
	\begin{Lemma}\label{lemma v epsilon}(\cite[Lemma~5.4]{SVWZ23})
		Let~$\alpha\in(0,1)$ and~$ g\in C^\alpha_{\rm loc}(B_1) $.
		Let~$ u\in C^{2,\alpha}(\overline{B_1})\cap L^\infty(\mathbb{R}^n)$ be a solution of 
		\begin{equation*}
			-\Delta u+(-\Delta)^s u= g \qquad \text{  in } B_1.
		\end{equation*}
		
		Then, there exists~$ \psi\in C^{\alpha}( B_{R}(x_0))$ such that~$ v:=\phi^R u $ satisfies 
		\begin{equation*}
			-\Delta v+(-\Delta)^s v= \psi \qquad \text{  in } B_1.
		\end{equation*}
	
	In particular, 
		\begin{equation}\label{estimate psi 2}
			R^2|\psi|^\prime_{0,\alpha;B_{R}(x_0)}\leqslant  C_{n,s}\left(R^2| g|^\prime_{0,\alpha;B_{R}(x_0)}+\|u\|_{L^\infty(\mathbb{R}^n)}\right),
		\end{equation}
		for some positive constant~$C_{n,s}$.
		\end{Lemma}
		
As a consequence of  Lemma~\refeq{lemma v epsilon},	setting~$v_\epsilon:=\phi^R u_\epsilon$,
we have that there exists~$ \psi_\epsilon\in C^{\alpha}(B_{R}(x_0)) $ such that~$ v_\epsilon $ satisfies  
	\begin{equation*}
		-\Delta v_\epsilon+(-\Delta)^s v_\epsilon= \psi_\epsilon \qquad \text{  in } B_1.
	\end{equation*}
	In particular, employing~\eqref{estimate psi 2}, one finds that
	\begin{equation}
		\begin{split}\label{eq:psi epsilon}
			R^2|\psi_\epsilon|^\prime_{0,\alpha;B_{R}(x_0)}&\leqslant  C_{n,s} \left(R^2| g_\epsilon|^\prime_{0,\alpha;B_{R}(x_0)}+ \|u_\epsilon\|_{L^\infty(\mathbb{R}^n)}\right).
		\end{split}
	\end{equation}

We also observe that, since~$v_\epsilon\in C^\infty_0( B_{2R}(x_0))$, for all~$\delta>0$,
		there exists~$ C_{\delta}>0  $ such that  
		\begin{equation*}
			\begin{split}
			&	R^2|(-\Delta)^s v_\epsilon|^\prime_{0,\alpha;B_{R}(x_0)}= R^2\|(-\Delta)^s v_\epsilon\|_{L^\infty(B_{R}(x_0))} +R^{2+\alpha}[(-\Delta)^s v_\epsilon]_{\alpha; B_{R}(x_0)}\\
				&\qquad\leqslant C_{n,s}
				|u_\epsilon|^\prime_{2,\alpha_0;B_{2R}(x_0)}
				\leqslant\delta |u_\epsilon|^\prime_{2,\alpha;B_{2R}(x_0)} +C_\delta \|u_\epsilon\|_{L^\infty(B_{2R}(x_0))},
			\end{split}
		\end{equation*}
		where 
		\begin{equation*}
			\alpha_0:=\begin{cases}
				0\qquad & \alpha<2-2s,\\
				\alpha-(1-s) &\alpha \geqslant 2-2s.
			\end{cases}
		\end{equation*}
		Thus, combining this with~\cite[Theorem~4.6]{GTbook}, one deduces that for all~$\delta>0$
	there exists~$ C_{\delta}>0  $ such that  
	\begin{equation*}
		\begin{split}
			|v_\epsilon|^\prime_{2,\alpha;B_{R/2}(x_0)}&\leqslant C\left(\|v_\epsilon\|_{L^\infty(B_{R}(x_0))}+R^2\left(|\psi_\epsilon|^\prime_{0,\alpha;B_{R}(x_0)}+|(-\Delta)^s v_\epsilon|^\prime_{0,\alpha;B_{R}(x_0)}\right)\right)\\
			&\leqslant C\bigg(\|v_\epsilon\|_{L^\infty(B_{R}(x_0))}+R^2|\psi_\epsilon|^\prime_{0,\alpha;B_{R}(x_0)}
			+\delta |u_\epsilon|^\prime_{2,\alpha;B_{2R}(x_0)}+C_\delta \|u_\epsilon\|_{L^\infty(B_{2R}(x_0))}\bigg). 
		\end{split}
	\end{equation*}
	
	Therefore, for every~$x_0\in B_{{3}/{4}}$, recalling the definition of~$v_\epsilon$ and exploting~\eqref{eq:psi epsilon} and Lemma~\ref{pro u epsilon1}, we conclude that for all~$ \delta>0$, there exists~$C_\delta$ such that 
	\begin{equation}\label{final equation 2}
		\begin{split}
			&|u_\epsilon|^\prime_{2,\alpha;B_{R/2}(x_0)}=|v_\epsilon|^\prime_{2,\alpha;B_{R/2}(x_0)}\\ 
			& \qquad\leqslant C \left(R^2|\psi_\epsilon|^\prime_{0,\alpha;B_{R}(x_0)}+\delta |u_\epsilon|^\prime_{2,\alpha;B_{2R}(x_0)}+C_\delta \|u\|_{L^\infty(B_{2R}(x_0))}\right)\\
			&\qquad\le C
			 \left(R^2| g_\epsilon|^\prime_{0,\alpha;B_{R}(x_0)}+ \|u_\epsilon\|_{L^\infty(\mathbb{R}^n)}
			 +\delta |u_\epsilon|^\prime_{2,\alpha;B_{2R}(x_0)}+C_\delta \|u\|_{L^\infty(B_{2R}(x_0))}
			 \right)
			\\
			&\qquad\leqslant C \left(\|u\|_{C^1(\mathbb{R}^n)}+\|u\|_{C^1(\mathbb{R}^n)}^p+\|u\|_{L^\infty(\mathbb{R}^n)}+\delta |u_\epsilon|^\prime_{2,\alpha;B_{2R}(x_0)}+C_\delta \|u\|_{L^\infty(B_{2R}(x_0))}\right)\\
		&\qquad	\leqslant C\left(\|u\|_{C^1(\mathbb{R}^n)}+\|u\|_{C^1(\mathbb{R}^n)}^p+C_\delta \|u\|_{L^\infty(\mathbb{R}^n)}+\delta |u_\epsilon|^\prime_{2,\alpha;B_{2R}(x_0)}\right),
		\end{split}
	\end{equation} for some~$ C>0$, depending on~$n$, $s$, $p$ and~$\alpha $.

\subsubsection{Interior $C^{2,\alpha}$-regularity}\label{sec:Proof of Theorem1.5}

The estimate in~\eqref{final equation 2},
coupled with the following statement, will allow us to 
obtain that the~$C^{2,\alpha}$-norm of~$u_\epsilon$ is bounded uniformly in some ball with respect to~$\epsilon$ and thus use Arzel\`a-Ascoli Theorem to complete the proof of Theorem~\ref{th C^2,alpha interior without boundary condition}.
The technical details go as follows:

\begin{Proposition}\label{pro final estimate}(\cite[Proposition~4.3]{SVWZ23})
	Let~$y\in \mathbb{R}^n$, $ d>0$ and~$u\in C^{2,\alpha}(B_{d}(y))$. Suppose that, for any~$\delta > 0 $, there
	exists~$\Lambda_\delta> 0$ such that, for any~$x\in B_{d}(y)$ and any~$r\in(0, d-|x-y|]$, we have that
	\begin{equation}\label{diuwetifgaslfg385730364793-05hrgtfedsxaz}
		|u|^\prime_{2,\alpha;B_{r/8}(x)}\leqslant \Lambda_\delta+ \delta |u|^\prime_{2,\alpha;B_{r/2}(x)}.
	\end{equation}
	
	Then, there exist constants~$\delta_0$, $C > 0$, depending only on~$n$, $\alpha$ and~$ d$, such that
$$
		\|u\|_{C^{2,\alpha}(B_{d/8}(y))}\leqslant C\Lambda_{\delta_0}.
$$
\end{Proposition}

	\begin{proof}[Proof of Theorem~\ref{th C^2,alpha interior without boundary condition}]
		We will use Proposition~\ref{pro final estimate} with~$d:=1/10 $, so that, for every~$ y\in B_{1/2}$,
		$$
			B_{d}(y)\subset B_{3/4} \quad \text{and}\quad  d/4<\frac{1}{20}.	$$

Moreover, we notice that the estimate in~\eqref{final equation 2} tells us that	
formula~\eqref{diuwetifgaslfg385730364793-05hrgtfedsxaz}
is verified in our setting with~$u$ replaced by~$u_\varepsilon$, $r:=4R$ and
$$
\Lambda_\delta:=
\|u\|_{C^1(\mathbb{R}^n)}+\|u\|_{C^1(\mathbb{R}^n)}^p+C_\delta \|u\|_{L^\infty(\mathbb{R}^n)}.
$$
Therefore, we are in a position of exploiting
Proposition~\ref{pro final estimate}, 
thus obtaining that, for every~$y\in B_{1/2}$,
$$
			\|u_\epsilon\|_{C^{2,\alpha}(\overline{B_{1/80}(y) })}\leqslant C\left(\|u\|_{C^1(\mathbb{R}^n)}+\|u\|_{C^1(\mathbb{R}^n)}^p+C_{\delta} \|u\|_{L^\infty(\mathbb{R}^n)}\right).
$$	
		
{F}rom the Arzel\`{a}-Ascoli Theorem, we obtain that~$ u\in C^{2,\alpha}(\overline{B_{1/80}(y)})$, for every~$y\in B_{1/2}$, and
		\begin{equation*}
			\|u\|_{C^{2,\alpha}(\overline{B_{1/80}(y)})}\leqslant C\left(\|u\|_{C^1(\mathbb{R}^n)}+\|u\|_{C^1(\mathbb{R}^n)}^p+ \|u\|_{L^\infty(\mathbb{R}^n)}\right).
		\end{equation*}
Hence, a covering argument and Theorem~\ref{th:regularity},
give that		\begin{equation*}
			\begin{split}
				\|u\|_{C^{2,\alpha}(\overline{B_{1/2}})}&\leqslant C\left(\|u\|_{C^1(\mathbb{R}^n)}+\|u\|_{C^1(\mathbb{R}^n)}^p+ \|u\|_{L^\infty(\mathbb{R}^n)}\right)\\
				&\leqslant C\left(\|u\|_{L^\infty(\mathbb{R}^n)}+\|u\|^p_{L^\infty(\mathbb{R}^n)}+\left(\|u\|_{L^\infty(\mathbb{R}^n)}+\|u\|^p_{L^\infty(\mathbb{R}^n)}\right)^p+ \|u\|_{L^\infty(\mathbb{R}^n)}\right),
			\end{split}
		\end{equation*}
		where the constant~$ C>0 $ depends on~$ n$, $s$, $\alpha$ and~$p $.
		
		 The ball~$B_1$ is centered at the origin, but we may arbitrarily move it around~$\mathbb{R}^n$. Covering~$\mathbb{R}^n$ with these balls, we
		 obtain the desired result.
	\end{proof}
	{

\section{Qualitative properties of positive solutions}\label{sec: Qualitative properties of positive solution}

In this section, we are concerned with the positivity, the decay at infinity and the radial symmetry of classical solutions of~\eqref{Maineq}, as stated in Theorem~\ref{th main theorem}. 

\subsection{Power-type decay of classical positive solutions}

In this part, we shall apply the Maximum Principle and  comparison arguments to obtain the 
decay at infinity of classical positive solutions of~\eqref{Maineq}.

\subsubsection{Existence of classical positive solutions}

We devote this part to establish the existence of classical positive solutions.

	\begin{Theorem}\label{th: u>0}
	Let~$n\geqslant2$ and~$s\in(0,1)$.
	
	Then, problem~\eqref{Maineq} has a classical solution, which satisfies~$u> 0$  in~$\mathbb{R}^n$. Moreover, 
	\begin{equation}\label{dioe3yr2fgufegvwj}
	\lim_{|x|\to+\infty}u(x)= 0.\end{equation}
\end{Theorem}

\begin{proof}
	{F}rom Theorems~\ref{th existence} and~\ref{th C^2,alpha interior without boundary condition}, 
we have that there exists a classical nonnegative solution of~\eqref{Maineq}.	
	
Hence, we can now focus on the proving that~$u>0$ in~$\mathbb{R}^n$. 
For this,
we argue by contradiction and we assume that there exists a global minimum point~$x_0\in\mathbb{R}^n$ at which~$u(x_0)=0$.
Accordingly, we have that~$\Delta u(x_0)\geqslant 0$
and~$(-\Delta)^s u(x_0)<0$. As a result, we deduce from~\eqref{Maineq} that
$$
	 	0=u^p(x_0)=-\Delta u(x_0)+(-\Delta)^s u(x_0)+u(x_0)< 0,
$$
which is a contradiction.

Finally, from Theorem~\ref{th:regularity}, we also have~\eqref{dioe3yr2fgufegvwj}.
\end{proof}

\subsubsection{Power-type decay of classical positive solutions}

In this part, we exploit suitable barriers constructed using the Bessel kernel~$\mathcal{ K}$ to establish the following result:

\begin{Theorem}\label{th:decay}
	Let~$n\geqslant 2$ and~$s\in(0,1)$. Let~$u$ be a positive classical solution of~\eqref{Maineq}. 
	
	Then, there exist constants~$0<C_1\leqslant C_2$ such that, for every~$ |x|\geqslant1$,
	\[ \frac{C_1}{|x|^{n+2s}}\leqslant u(x)\leqslant \frac{C_2}{|x|^{n+2s}}. \]
\end{Theorem}

To prove this result, we start with the following two lemmata, which construct suitable subsolutions and supersolutions.
We start with the analysis of the subsolution:

\begin{Lemma}\label{lemma subsection}
There exists a  function~$\omega\in C^{1,\alpha}(\mathbb{R}^n)$ satisfying
	\begin{equation}\label{x>1}
		\begin{cases}
	-\Delta \omega+(-\Delta)^s \omega+\omega=0 \quad \text{in }{\mathbb{R}}^n\setminus B_1,\\
	\omega>0\quad \text{in }{\mathbb{R}}^n, \\
	\displaystyle \lim_{ |x|\to +\infty}\omega(x)= 0  
\end{cases}
	\end{equation}
and, for every~$ |x|>1$,
\begin{equation}\label{omega }
	\omega(x)\geqslant \frac{c_1}{|x|^{n+2s}},
\end{equation}
	for some constant~$c_1>0$.
\end{Lemma}

\begin{proof}
We consider the function~$\omega:=\mathcal{ K}\ast \chi_{ B_{1/2}}$,
where~$\chi_{B_{1/2}}$ is the characteristic function of the  ball~$B_{1/2}$.

	It is immediate to check that~$\omega>0$ solves
	$$ -\Delta \omega+(-\Delta)^s \omega+\omega=  \chi_{ B_{1/2}}\quad \text{in }{\mathbb{R}}^n, $$
	and therefore the equation in~\eqref{x>1} is satisfied. 
	
Also, in light of the smoothness of~$\mathcal{ K}$ (see e.g. (c) of Theorem~\ref{th properties of k}), we have that~$\omega\in C^{2}(\mathbb{R}^n\setminus B_{1})$. 

In particular, since~$\chi_{B_{1/2}}\in L^q(\mathbb{R}^n) $ for all~$q\geqslant 1$, we have that~$\omega\in   W^{2,q}(\mathbb{R}^n)$ for all~$q\geqslant 1$
(by Lemma~\ref{lemma embedding of X^p}), which implies that~$\omega\in C^{1,\alpha}(\mathbb{R}^n)$ for any~$\alpha\in(0,1)$. 
As a result, the limit in~\eqref{x>1} is also satisfied.
	
	We now check~\eqref{omega }.
Using the lower bound on~$\mathcal{K}$
	 (see e.g.~(b) of Theorem~\ref{th properties of k}), for any~$|x|>1$, one has that 
	\begin{equation*}
			\omega(x)=\int_{B_{1/2}}\mathcal{ K}(x-y)\,dy\geqslant c
			\int_{B_{1/2}}\frac{dy}{|x-y|^{n+2s}}
\geqslant c\int_{B_{1/2}}\frac{dy}{(|x|+|y|)^{n+2s}}
\geqslant \frac{c}{|x|^{n+2s}},
	\end{equation*} up to renaming~$c>0$.
	As a consequence of this, we obtain~\eqref{omega }, as desired.
\end{proof}

The supersolution is constructed exploting
Theorem~\ref{th properties of k} with a parameter~$a>0$ in place of~$1$:

\begin{Lemma}\label{lemma K_a}
	Let~$n\geqslant 1$ and~$s\in(0,1)$. Let~$a>0$ and  
	\[{\mathcal{ K}_a}:=\mathcal{F}^{-1}\left(\frac{1}{a+|\xi|^2+|\xi|^{2s}}\right). \]
	
	Then,
	\begin{itemize}
		\item[(a)] $\mathcal{K}_a$ is positive, radially symmetric, smooth in~$\mathbb{R}^n\setminus \left\{0\right\}$
		and nonincreasing with respect to~$r=|x|$. 
		\item[(b)] There exist positive constants~$C_1$ and~$C_2$ such that,
		 if~$|x|\geqslant 1$,
$$
\frac{C_1}{|x|^{n+2s}}	\leqslant\mathcal{K}_a(x)\leqslant \frac{1}{C_1|x|^{n+2s}}$$ and, if~$|x|\leqslant 1$,
$$\frac{C_2}{|x|^{n-2}}\leqslant\mathcal{K}_a(x)\leqslant\frac1{C_2}
\begin{cases}
			|x|^{2-n} \quad &{\mbox{ if }}n\geqslant 3,\\
			1+|\ln|x|| &{\mbox{ if }}n=2,\\
			1+|x| &{\mbox{ if }}n=1.
		\end{cases}
$$
		\item[(c)] There exists a positive constant~$C$ such that, if~$|x|\geqslant 1$,
$$
			|\nabla\mathcal{K}_a(x)|\leqslant \frac{C}{|x|^{n+2s+1}}\qquad \mbox{and} \qquad |D^2\mathcal{ K}_a(x)|\leqslant \frac{C}{|x|^{n+2s+2}} .
$$
\item[(d)] If~$n\geqslant3$, then~$\mathcal{K}_a\in L^q(\mathbb{R}^n)$ for all~$q\in [1,\frac{n}{n-2})$. If~$n=1$, $2$, then~$\mathcal{K}_a\in L^q(\mathbb{R}^n)$ for all~$q\in [1,+\infty)$.
		\item[(e)] There exists a positive constant~$C$ such that, if~$|x|\geqslant 1$,
		 $$
			\mathcal{K}_a(x)\geqslant \frac{C}{|x|^{n+2s}}.
$$	\end{itemize}
\end{Lemma}

Exploting the kernel given by Lemma~\ref{lemma K_a}, we are able to construct
the supersolution as follows:

\begin{Lemma}\label{lemma supersection}
	There exists~$v\in C^{1,\alpha}(\mathbb{R}^n)$ satisfying
	\begin{equation*}
		\begin{cases}
			-\Delta v+(-\Delta)^s v+\frac{1}{2}v=0 \qquad \text{in } {\mathbb{R}}^n\setminus B_1,\\
		v>0 \quad \text{in }{\mathbb{R}}^n,\\
		\displaystyle\lim_{|x|\to +\infty}	v(x)= 0.
		\end{cases}
			\end{equation*}	
	Also, when~$ |x|>1$ we have that
$$
		0< v(x)\leqslant \frac{c_2}{|x|^{n+2s}},
$$	for some constant~$c_2>0$.
\end{Lemma}

\begin{proof}
	We consider the function~$v:=\mathcal{K}_{1/2}\ast \chi_{B_{1/2}}$, where~$\mathcal{K}_{1/2}$ 
	is defined in Lemma~\ref{lemma K_a} with~$a:=1/2$.
	
	Thus, $v>0$ and it satisfies
$$
		-\Delta v+(-\Delta)^s v+\frac{1}{2}v=\chi_{B_{1/2}}\qquad \text{in }\mathbb{R}^n.
$$

Moreover, from Lemma~\ref{lemma K_a}, when~$|x|>1$ we have that
\begin{equation*}
	v(x)=\int_{B_{1/2}}\mathcal{ K}_{1/2}(x-y)\, dy\leqslant c
	\int_{B_{1/2}}\frac{dy}{|x-y|^{n+2s}} 
	\leqslant c \int_{B_{1/2}}\frac{dy}{(|x|-1/2)^{n+2s}}
	\leqslant \frac{ c}{|x|^{n+2s}},
\end{equation*}
for some positive constant~$c$, that may change from step to step. 

Furthermore, using the smoothness of~$\mathcal{ K}_{1/2}$ (see e.g.~(a) of Lemma~\ref{lemma K_a}), we have that~$v\in C^{2}(\mathbb{R}^n\setminus B_{1})$. 
We claim that~$v\in C^{1,\alpha}(\mathbb{R}^n)$ for any~$\alpha\in(0,1)$. 

Indeed, we point out that~$\chi_{ B_{1/2}}\in L^q(\mathbb{R}^n)$ for all~$q\geqslant 1$ and also we see that
\[  \left(\frac12+|\xi|^2+|\xi|^{2s}\right)\mathcal{F}(v)(\xi)=\mathcal{F}({\chi_{ B_{1/2}}})(\xi).\]
Then, 
\[ \Big(1+|\xi|^2+|\xi|^{2s}\Big)\mathcal{F}(v)(\xi)=\frac{1+|\xi|^2+|\xi|^{2s}}{(1/2+|\xi|^2+|\xi|^{2s})}\mathcal{F}({\chi_{ B_{1/2}}})(\xi)=\hat{g}, \]
where~$g:=\left(\delta+\frac{1}{2}\mathcal{ K}_{1/2}\right)\ast {\chi_{ B_{1/2}}}$. 

Since~$\mathcal{ K}_{1/2}\in L^1(\mathbb{R}^n)$, we have that~$g\in L^q(\mathbb{R}^n)$ for all~$q\ge1$ and 
\[ 	-\Delta v+(-\Delta)^s v+v=g\qquad \text{in }\mathbb{R}^n. \]
These facts and Lemma~\ref{lemma embedding of X^p} imply that~$ v\in W^{2,q}(\mathbb{R}^n) $ for all~$q\ge1$.
By the Sobolev Embedding, we obtain that~$v\in C^{1,\alpha}(\mathbb{R}^n)$ for any~$\alpha\in(0,1)$ and also
\begin{equation*}\lim_{|x|\to+\infty}v(x)= 0.\qedhere\end{equation*}
\end{proof}

\begin{proof}[Proof of Theorem~\ref{th:decay}]
We consider the function~$\omega$ given by Lemma~\ref{lemma subsection}. Utilizing the positivity and continuity of~$u$ and~$\omega$ in~$\mathbb{R}^n$, we can find a constant~$\beta>0$ such that~$h:=u-\beta\omega> 0$ in~$\overline{B_{1}}$. 
	
	Furthermore, from Theorem~\ref{th: u>0} and Lemma~\ref{lemma subsection}, we see that 
	\begin{equation*}
		\begin{cases}
		-\Delta h+(-\Delta )^{s}h\geqslant -h \quad \text{ in }{\mathbb{R}}^n\setminus B_1,	\\
		\displaystyle \lim_{|x|\to +\infty }h(x)= 0.
	\end{cases}
	\end{equation*}
	
We claim that
\begin{equation}\label{di329ru32fuewigfegw}
h(x)\geqslant 0\qquad{\mbox{ for all~$|x|>1$.}}\end{equation}
Indeed, for the sake of contradiction, we  assume that there exists a global strictly negative minimum
point~$x_0\in \mathbb{R}^n\setminus B_{1}$. Since~$h\in C^{2}(\mathbb{R}^n\setminus B_{1})$, we have~$\Delta h(x_0)\geqslant 0$
and~$(-\Delta)^s h(x_0)\le0$, and thus 
\[  0<-h(x_0)\leqslant-\Delta h(x_0)+(-\Delta )^{s}h(x_0)\leqslant 0.\]
This contradiction implies~\eqref{di329ru32fuewigfegw}.

{F}rom~\eqref{di329ru32fuewigfegw}, we have that, for every~$ |x|>1$,
\[ u(x)\geqslant \beta\omega(x)\geqslant \frac{c_1}{|x|^{n+2s}}. \]
This establishes the bound from below in Theorem~\ref{th:decay}.

We now focus on the bound from above. Owing to~\eqref{dioe3yr2fgufegvwj}, we can find some~$R>1$ such that
$$
-\Delta u+(-\Delta )^{s}u+\frac{1}{2}u\leqslant 0\qquad \text{in } \mathbb{R}^n\setminus B_{R}.
$$
In this case, we make use of the function~$v$ given by Lemma~\ref{lemma supersection}.
{F}rom the positivity and continuity of~$u$ and~$v$ in~$\mathbb{R}^n$, there exists~$\gamma>0$ such that~$g:=v-\gamma u> 0$ in~$\overline{B_{R}}$. 

In view of Theorem~\ref{th: u>0} and Lemma~\ref{lemma supersection}, one has that
\begin{equation*}
	\begin{cases}
		-\Delta g+(-\Delta )^{s}g\geqslant -\frac{1}{2}g \quad \text{ in }{\mathbb{R}}^n\setminus B_R,	\\
		\displaystyle \lim_{|x|\to +\infty }g(x)= 0.
	\end{cases}
\end{equation*}
We claim that
\begin{equation}\label{oyfouewgfwo4573485692}
g(x)\geqslant 0\qquad{\mbox{ for all~$|x|>R$.}}\end{equation}
For this, by contradiction, we assume that there exists a global strictly
negative minimum point~$x_1\in \mathbb{R}^n\setminus B_{R}$. Since~$g\in C^{2}(\mathbb{R}^n\setminus B_{1})$, we have that~$\Delta g(x_1)\geqslant 0$
and~$(-\Delta)^s g(x_1)\le0$, and hence 
\[  0<-\frac{1}{2}g(x_1)\leqslant-\Delta g(x_1)+(-\Delta )^{s}g(x_1)\leqslant 0.\]
This contradiction gives~\eqref{oyfouewgfwo4573485692}.

In turn, \eqref{oyfouewgfwo4573485692} implies that, for all~$|x|>R$,
\[ u(x)\leqslant \gamma v(x)\leqslant \frac{c_2}{|x|^{n+2s}}. \]
Thus, the continuity of~$u$ allows us to complete the proof of the bound from
above in Theorem~\ref{th:decay}. 
\end{proof}

\subsection{Radial symmetry of positive solutions}

Our aim is now to prove the radial symmetry of the positive solution found in  Theorem~\ref{th: u>0}. The main statement of this section is the following:

\begin{Theorem}\label{th symmetric}
Let~$n\geqslant 2$ and~$s\in(0,1)$. 

Then, all classical positive solutions of~\eqref{Maineq} are radially symmetric.
\end{Theorem}

To establish Theorem~\ref{th symmetric}, we will exploit the moving planes method.
To begin with, we recall some notation.
For any~$\lambda\in\R$, we set
\begin{eqnarray*}
\Sigma_{\lambda}&:=&\left\{x\in\mathbb{R}^n \;{\mbox{ s.t. }}\; x_1>\lambda \right\}\\ {\mbox{and }}\qquad 
T_\lambda&:=&\left\{x\in\mathbb{R}^n \;{\mbox{ s.t. }}\;x_1=\lambda \right\}.
\end{eqnarray*}
Moreover, for any~$x=(x_1,x_2,\dots,x_n)\in\R^n$, we set~$x^\lambda:=\left(2\lambda-x_1,x_2,\dots,x_n\right)$ and~$u_\lambda(x):=u(x^\lambda)$.

With this notation, we first establish the following:

\begin{Lemma}\label{lemma u-lambda}
	If~$u\geqslant u_\lambda$ and~$u\not\equiv u_\lambda$  in~$\Sigma_\lambda$, then~$ u> u_{\lambda}$ in~$\Sigma_{\lambda}$.
\end{Lemma}

\begin{proof}
	Suppose by contradiction that there exists~$x_0\in \Sigma_{\lambda}$ such that~$ u(x_0)= u_{\lambda}(x_0)$. In the light of Theorem~\ref{th C^2,alpha interior without boundary condition}, we know that~$u\in C^2(\mathbb{R}^n)$. We thus have~$-\Delta( u(x_0)-u_{\lambda}(x_0))\leqslant 0$, and therefore 
	\begin{equation*}
		\begin{split}
			0&=-\Delta( u(x_0)-u_{\lambda}(x_0))+(-\Delta)^s( u(x_0)-u_{\lambda}(x_0))+( u(x_0)-u_{\lambda}(x_0))\\
			&=\int_{ \R^n}\frac{-( u(y)-u_{\lambda}(y))}{|x_0-y|^{n+2s}}\, dy\\
			&=\int_{ \Sigma_\lambda}\frac{-( u(y)-u_{\lambda}(y))}{|x_0-y|^{n+2s}}\, dy
			+\int_{ \R^n\setminus \Sigma_\lambda}\frac{-( u(y)-u_{\lambda}(y))}{|x_0-y|^{n+2s}}\, dy\\
			&\leqslant \int_{ \Sigma_{\lambda}}\frac{-( u(y)-u_{\lambda}(y))}{|x_0-y|^{n+2s}}\, dy+\int_{ \Sigma_{\lambda}}\frac{-( u_\lambda(y)-u(y))}{|x_0-y^\lambda|^{n+2s}}\, dy\\
			&= \int_{ \Sigma_{\lambda}}\big( u(y)-u_{\lambda}(y)\big)
		\left(	\frac1{|x_0-y^\lambda|^{n+2s}}-
\frac{1}{|x_0-y|^{n+2s}}\right)\, dy\\&\le0,
		\end{split}
	\end{equation*}
which implies that~$u\equiv u_\lambda$  in~$\Sigma_\lambda$. This is in contradiction with the assumption and therefore the desired result
is established.
\end{proof}

Now, we define
$$\lambda_0:=\sup\Big\{\lambda\in\R\; {\mbox{ s.t. }}\; u(x)>u_\lambda(x) \text{ for all } x\in \Sigma_\lambda \Big\}.$$

\begin{Lemma}\label{leu7563}
We have that~$\lambda_0>-\infty$.
\end{Lemma}

To check this, we will need the following estimate:

\begin{Lemma}\label{lemma 6.2}
Let~$q>r>\frac{n}{2}$. Then, for any measurable set~$\Omega\subset\mathbb{R}^n$, there exists~$c>0$, depending on~$n$ and~$s$, such that, for all measurable functions~$g:\Omega\to\R$, 
$$
		\left\|\int_{\Omega}\mathcal{ K}(\cdot-y)g(y)\, dy\right\|_{L^q(\Omega)}\leqslant c \,\|g\|_{L^r(\Omega)}.
$$
\end{Lemma}

The proof of Lemma~\ref{lemma 6.2} follows the same line as the one
of~\cite[Lemma~5.2]{MR3002595}. We provide here the details for the facility
of the reader, since some exponents need to be adjusted to our setting.

\begin{proof}[Proof of Lemma~\ref{lemma 6.2}]
Notice that if~$\|g\|_{L^r(\Omega)}$ is unbounded, the claim is obviously true.
Hence, we can suppose that~$\|g\|_{L^r(\Omega)}<+\infty$.

	{F}rom~(d) of Theorem~\ref{th properties of k}
	we know that if~$n\geqslant 3$, then~$\mathcal{ K}\in L^{\gamma}(\mathbb{R}^n)$ for~$\gamma\in[1,\frac{n}{n-2})$ and if~$n=2$, then~$\mathcal{ K}\in L^{\gamma}(\mathbb{R}^n)$ for~$\gamma\in[1,+\infty)$.
	
	We point out that, since~$r>n/2$, setting~$\gamma:=\frac{r}{r-1}$, we
	see that~$\gamma\in\left(1,\frac{n}{n-2}\right)$ when~$n\ge3$.
	Thus, using the H\"{o}lder inequality
	with exponents~$\gamma$ and~$r$, we have that, for every~$x\in \Omega$,
	\[ \left|\int_{\Omega}\mathcal{ K}(x-y)g(y)\, dy\right|\leqslant \|\mathcal{ K}\|_{L^\gamma(\mathbb{R}^n)} \|g\|_{L^{r}(\Omega)}\leqslant C\|g\|_{L^{r}(\Omega)}. \]
	In addition, by Young's convolution inequality,
	\[ \left\|\int_{\Omega}\mathcal{ K}(\cdot-y)g(y)\, dy\right\|_{L^{r}(\Omega)}\leqslant \|\mathcal{ K}\|_{L^1(\mathbb{R}^n)} \|g\|_{L^{r}(\Omega)}\leqslant C\|g\|_{L^{r}(\Omega)}. \]
	 As a result,  we  conclude that 
	 \begin{equation*}
	 	\begin{split}
	 		&\left\|\int_{\Omega}\mathcal{ K}(\cdot-y)g(y)\, dy\right\|_{L^q(\Omega)}\\
	 		=\;&\left(\int_{\Omega}\left(\int_{\Omega}\mathcal{ K}(x-y)g(y)\, dy\right)^{q-r}\left(\int_{\Omega}\mathcal{ K}(x-y)g(y)\, dy\right)^{r}\, dx\right)^{1/q}\\
	 		\leqslant\;& \left\|\int_{\Omega}\mathcal{ K}(x-y)g(y)\, dy\right\|^{1-r/q}_{L^\infty(\Omega)}\,
	 		\left\|\int_{\Omega}\mathcal{ K}(x-y)g(y)\, dy\right\|^{r/q}_{L^r(\Omega)}\\
	 		\leqslant\;& C\|g\|_{L^{r}(\Omega)},
	 	\end{split}
	 \end{equation*}
as desired.
\end{proof}

We will also need the following intermediate estimate:

\begin{Lemma}
Let~$q>\max\{n, 2/(p-1)\}$. Let
\begin{equation}\label{defsigmalambdameno}
\Sigma_{\lambda}^-:=\left\{x\in\Sigma_{\lambda}\;{\mbox{ s.t. }}\;u(x)<u_\lambda(x)\right\}. \end{equation}

Then,
\begin{equation}\label{893r2fgewkfgwekugt43ipoiuytr}
 	\|u_\lambda-u\|_{L^q(\Sigma_\lambda^-)}\leqslant c\|u_\lambda\|^{p-1}_{L^{q(p-1)}(\Sigma_\lambda^-)}\|u_\lambda-u\|_{L^{q}(\Sigma_\lambda^-)}.
\end{equation}
\end{Lemma}

\begin{proof}
Owing to the fact that~$\mathcal{ K}$ is radial symmetric (recall~(a)
of Theorem~\ref{th properties of k}), we deduce that 
	\begin{eqnarray*}
			u(x)&=&\int_{\mathbb{R}^n}\mathcal{ K}(x-y) u^p(y)\, dy
			= \int_{\Sigma_\lambda}\mathcal{ K}(x-y) u^p(y)\, dy+\int_{\Sigma_\lambda}\mathcal{ K}(x^\lambda-y) u^p_\lambda(y)\, dy\\
{\mbox{and }}\qquad
	u_\lambda(x)&=&\int_{\mathbb{R}^n}\mathcal{ K}(x-y) u_\lambda^p(y)\, dy
		= \int_{\Sigma_\lambda}\mathcal{ K}(x-y) u^p_\lambda(y)\, dy+\int_{\Sigma_\lambda}\mathcal{ K}(x^\lambda-y) u^p(y)\, dy.
\end{eqnarray*} 
Moreover, we point out that~$|y-x^\lambda|> |y-x|$ for every~$x$, $y\in\Sigma_\lambda$. Collecting these pieces of information and
recalling that~$\mathcal{ K}$ is nonincreasing with respect to~$r=|x|$,
we conclude that, for every~$ x\in\Sigma_\lambda$, 
\begin{equation*}
	\begin{split}
		u_\lambda(x)-u(x)&=\int_{\Sigma_\lambda}\left(\mathcal{ K}(x-y)-\mathcal{ K}(x^\lambda-y)\right) \left(u^p_\lambda(y)-u^p(y)\right)\, dy\\
		&\leqslant \int_{\Sigma^-_\lambda}\left(\mathcal{ K}(x-y)-\mathcal{ K}(x^\lambda-y)\right) \left(u^p_\lambda(y)-u^p(y)\right)\, dy\\
		&\leqslant \int_{\Sigma^-_\lambda}\mathcal{ K}(x-y) \left(u^p_\lambda(y)-u^p(y)\right)\, dy.
	\end{split}
\end{equation*}
Therefore, by the positivity of~$\mathcal{ K}$ and the Mean Value Theorem, we  derive that 
\begin{equation*}
		|u_\lambda(x)-u(x)|\leqslant\int_{\Sigma^-_\lambda}\mathcal{ K}(x-y) \left|u^p_\lambda(y)-u^p(y)\right|\, dy
		\leqslant p\int_{\Sigma^-_\lambda}\mathcal{ K}(x-y)u^{p-1}_\lambda(y) \left|u_\lambda(y)-u(y)\right|\, dy.
\end{equation*}

Since~$q>\max\{n, 2/(p-1)\}$, we can use Lemma~\ref{lemma 6.2} with~$r:=q/2$, $\Omega:=\Sigma_\lambda^-$ and~$g:=u^{p-1}_\lambda (u_\lambda-u)$ and we find that
$$
	\|u_\lambda-u\|_{L^q(\Sigma_\lambda^-)}
	\le p\left\|\int_{\Sigma^-_\lambda}\mathcal{ K}(\cdot-y)u^{p-1}_\lambda(y) \left|u_\lambda(y)-u(y)\right|\, dy\right\|_{L^q(\Sigma_\lambda^-)}
	\leqslant c \left\|u^{p-1}_\lambda (u_\lambda-u)\right\|_{L^{q/2}(\Sigma_\lambda^-)},
$$
for some~$c>0$ depending on~$n$, $s$ and~$p$.

{F}rom this and the H\"{o}lder inequality, we obtain~\eqref{893r2fgewkfgwekugt43ipoiuytr}.

We point out that, for all~$m\in[2,+\infty)$,
\begin{eqnarray*}
\|u\|_{L^m(\R^n)}^m= \int_{\R^n}|u(x)|^m\,dx
\le\|u\|_{L^\infty(\R^n)}^{m-2}\int_{\R^n}|u(x)|^2\,dx<+\infty,
\end{eqnarray*}
thanks to Theorem~\ref{Theorem :holder regularity}. This gives that the norms
in~\eqref{893r2fgewkfgwekugt43ipoiuytr} are finite.
\end{proof}

\begin{proof}[Proof of Lemma~\ref{leu7563}]
We recall the definition of~$\Sigma_\lambda^-$ in~\eqref{defsigmalambdameno}
and we
observe that, for every~$m\in[2,+\infty)$,
\begin{eqnarray*}&&
\lim_{\lambda\to-\infty}\|u_\lambda\|^{m}_{L^{m}(\Sigma_\lambda^-)}
=\lim_{\lambda\to-\infty}\int_{\Sigma_\lambda^-} |u_\lambda(x)|^m\,dx\\
&&\qquad=\lim_{\lambda\to-\infty}
\int_{\{x_1>\lambda\}\cap \{u(x^\lambda)>u(x)\}} |u(x^\lambda)|^m\,dx
\le\lim_{\lambda\to-\infty}\int_{\{y_1<\lambda\}} |u(y)|^m\,dy=0.
\end{eqnarray*}
Thus, there exists~$\lambda^-\in\R$ sufficiently negative 
such that~$c\|u_\lambda\|^{p-1}_{L^{q(p-1)}(\Sigma_\lambda^-)}\le1/2$
for all~$\lambda\le\lambda^-$. Therefore, using this information into~\eqref{893r2fgewkfgwekugt43ipoiuytr}, we find that,
for all~$\lambda\le\lambda^-$,
$$
	\|u_\lambda-u\|_{L^q(\Sigma_\lambda^-)}\leqslant 
	c\|u_\lambda\|^{p-1}_{L^{q(p-1)}(\Sigma_\lambda^-)}
	\|u_\lambda-u\|_{L^{q}(\Sigma_\lambda^-)}\le
	\frac{1}{2}\|u_\lambda-u\|_{L^{q}(\Sigma_\lambda^-)}.
$$
This implies that, for all~$\lambda\le\lambda^-$,
we have that~$|\Sigma_\lambda^-|=0$, and therefore,
for all~$x\in\Sigma_\lambda$, we see that~$u(x)\geqslant u_\lambda(x)$.

Now, from Lemma~\ref{lemma u-lambda} it follows that,
for all~$\lambda\le\lambda^-$,
\begin{equation}\label{esdfghj098765432eirguoi0}
{\mbox{either~$ u> u_{\lambda}$  in~$\Sigma_{\lambda}$
or~$u\equiv u_\lambda$ in~$\R^n$.}}\end{equation}

We claim that
\begin{equation}\label{esdfghj098765432eirguoi}
{\mbox{if~$u\equiv u_{\lambda_1}\equiv u_{\lambda_2}$ in~$\R^n$ for some~$\lambda_1$, $\lambda_2\in\R$,
then~$\lambda_1=\lambda_2$.}}
\end{equation}
Indeed, suppose that~$\lambda_2>\lambda_1$ and let~$T:=2(\lambda_2-\lambda_1)$. Then, for all~$x\in\R^n$,
$$ u(x+Te_1)=u((x+Te_1)^{\lambda_2})
=u(2\lambda_2 e_1-(x+Te_1))=u(2\lambda_1 e_1-x)=u(x).
$$
That is, $u$ is periodic in the first coordinate of period~$T$, violating
the fact that~$u$ is non-trivial and with finite~$L^2$-norm.
This establishes~\eqref{esdfghj098765432eirguoi}.

{F}rom~\eqref{esdfghj098765432eirguoi0} and~\eqref{esdfghj098765432eirguoi},
we thereby conclude that~$ u> u_{\lambda}$  in~$\Sigma_{\lambda}$,
for all~$\lambda\le\lambda^-$ (up to taking~$\lambda^-$ smaller if needed).
This implies that~$\lambda_0\ge \lambda^-$, as desired.
\end{proof}

\begin{Lemma}\label{leu75632}
We have that~$\lambda_0<+\infty$.
\end{Lemma}

\begin{proof}
Suppose by contradiction that there exists a
sequence~$\lambda_k\to+\infty$ such that~$u(x)>u_{\lambda_k}(x)$ for all~$x\in\Sigma_{\lambda_k}$. Then, thanks to~\eqref{dioe3yr2fgufegvwj},
$$0=\lim_{k\to+\infty} u(2\lambda_k e_1)\ge 
\lim_{k\to+\infty} u((2\lambda_k e_1)^{\lambda_k})=u(0).
$$
This is a contradiction with the fact that~$u>0$ in~$\R^n$, in light of Theorem~\ref{th: u>0}.
\end{proof}

In light of Lemma~\ref{leu7563}, we have that
$$ \Big\{\lambda\in\R\; {\mbox{ s.t. }}\; u(x)>u_\lambda(x) \text{ for all } x\in \Sigma_\lambda \Big\}\neq \varnothing $$
and moreover
\begin{equation}\label{asdfghjqwertyui12345678}
{\mbox{$u(x)\ge u_{\lambda_0}(x)$ for all~$x\in\Sigma_{\lambda_0}$.}}
\end{equation}

With this preliminary work, we can prove the following:

\begin{Lemma}\label{lemma:201ry}
We have that~$u_{\lambda_0}(x)=u(x)$ for all~$x\in\Sigma_{\lambda_0}$.
\end{Lemma}

\begin{proof}
Suppose by contradiction that~$u\not\equiv u_{\lambda_0}$ in~$\Sigma_{\lambda_0}$. 
Then, thanks to~\eqref{asdfghjqwertyui12345678} and Lemma~\ref{lemma u-lambda}, we have that
\begin{equation}\label{mnbvcxzkjhgfd987654}
u> u_{\lambda_0}\qquad {\mbox{ in }}\Sigma_{\lambda_0}.\end{equation}
	
	We claim that there exists~$\epsilon>0$, depending on~$n$, $s$, $p$ and~$u$,
	such that
	 \begin{equation}\label{nthbrver685vr38icr6tb8fny}
	 {\mbox{$u\geqslant u_{\lambda}$ in~$\Sigma_{\lambda}$,
	 for all~$\lambda \in [\lambda_0,\lambda_0+\epsilon)$.}}\end{equation}
	Indeed, we recall the definition of~$\Sigma_\lambda^-$ in~\eqref{defsigmalambdameno} 
and that, for any~$q>\max\{n, 2/(p-1)\}$,
the estimate in~\eqref{893r2fgewkfgwekugt43ipoiuytr} holds true, namely
	 \begin{equation}\label{eq: lambda0}
		\|u_\lambda-u\|_{L^q(\Sigma_\lambda^-)}\leqslant c \|u_\lambda\|^{p-1}_{L^{q(p-1)}(\Sigma_\lambda^-)}\|u_\lambda-u\|_{L^{q}(\Sigma_\lambda^-)}.
	\end{equation}

Now we pick~$R_{0}>0$ large enough such that, for all~$\lambda\in[\lambda_0,\lambda_0+1]$, 
\begin{equation}\label{fuerfgegileagwo4y584396}
c\|u_\lambda\|^{p-1}_{L^{q(p-1)}(\Sigma_\lambda^-\cap (\R^n\setminus B_{R_0}))}\leqslant \frac{1}{3}. \end{equation}

Moreover, from~\eqref{mnbvcxzkjhgfd987654}
we deduce that,
for~$\epsilon>0$ small enough,
we can find~$c_\epsilon>0$  such that 
\[ u(x)- u_{\lambda_0}(x)\geqslant c_\epsilon \qquad \text{for every }x\in \Sigma_{\lambda_0+\epsilon}\cap B_{R_0}.\]
In addition, by the continuity of~$u$, up to taking~$\epsilon$ smaller,
we have that, for all~$\lambda\in[\lambda_0,\lambda_0+\epsilon)$,
\[ u(x)\geqslant u_{\lambda}(x) \qquad \text{for every } x\in \Sigma_{\lambda_0+\epsilon}\cap B_{R_0}.\]
This implies that, for all~$\lambda\in[\lambda_0,\lambda_0+\delta_\epsilon)$,
$$\Sigma_\lambda^- \cap B_{R_0}\subset (\Sigma_\lambda\setminus \Sigma_{\lambda_0+\epsilon}) \cap B_{R_0}. $$
Thus, for all~$\lambda\in[\lambda_0,\lambda_0+\epsilon)$, we have that
\[ |\Sigma_\lambda^- \cap B_{R_0}|\leqslant |(\Sigma_\lambda\setminus \Sigma_{\lambda_0+\epsilon} )\cap B_{R_0}|\leqslant c_{n}R_0^{n-1}\epsilon. \]
As a result, up to taking~$\epsilon$ even smaller, we find that, for all~$\lambda\in[\lambda_0,\lambda_0+\delta_{\epsilon_0})$,
\[ c \|u_\lambda\|^{p-1}_{L^{q(p-1)}(\Sigma_\lambda^- \cap B_{R_0})}\leqslant \frac{1}{3}. \]

As a consequence of this and~\eqref{fuerfgegileagwo4y584396},
for all~$\lambda\in[\lambda_0,\lambda_0+\delta_{\epsilon_0})$,
$$ c \|u_\lambda\|^{p-1}_{L^{q(p-1)}(\Sigma_\lambda^- )}\leqslant \frac{2}{3}.$$
Pluggin this information into~\eqref{eq: lambda0}, we thereby obtain that
$$		\|u_\lambda-u\|_{L^q(\Sigma_\lambda^-)}\leqslant \frac23\|u_\lambda-u\|_{L^{q}(\Sigma_\lambda^-)}.$$
This, in turn, implies that~$|\Sigma_\lambda^-|=0$ for all~$\lambda\in[\lambda_0,\lambda_0+\delta_{\epsilon_0})$, which gives the desired claim in~\eqref{nthbrver685vr38icr6tb8fny}.

Owing to~\eqref{nthbrver685vr38icr6tb8fny} and Lemma~\ref{lemma u-lambda},
we conclude that, for all~$\lambda\in[\lambda_0,\lambda_0+\delta_{\epsilon_0})$,
either~$u> u_{\lambda}$ in~$\Sigma_{\lambda}$
or~$u\equiv u_\lambda$ in~$\R^n$.
This and~\eqref{esdfghj098765432eirguoi}
give that the second possibility cannot occur, and therefore, for all~$\lambda\in[\lambda_0,\lambda_0+\delta_{\epsilon_0})$,
we have that~$u> u_{\lambda}$ in~$\Sigma_{\lambda}$.

This is a contradiction with the definition of~$\lambda_0$ and thus the desired
result is established.
\end{proof}

\begin{proof}[Proof of Theorem~\ref{th symmetric}]
	By translation, we may suppose that~$\lambda_0=0$. Hence, we have that~$u$ is symmetric with respect to the~$x_1$-axis, i.e. $u(x_1,x')=u(-x_1,x')$. Using the same approach in any arbitrary direction, we obtain that~$u$ is radially symmetric (see e.g.~\cite[Lemma~8.2.1]{ZZLIB} for full details of this standard argument).
\end{proof}

\begin{proof}[Proof of Theorem~\ref{th main theorem}]
Theorem~\ref{th main theorem} now follows from Theorems~\ref{th: u>0}, \ref{th:decay} and~\ref{th symmetric}.
\end{proof}

\begin{appendix}

\section{Properties of the heat kernel}\label{th properties of h1}

	In this section, we focus on the fundamental properties of the heat kernel
	introduced in~\eqref{DBSD-1}, with the aim of establishing Theorem~\ref{th properties of h}. For this, we observe that the main step to establish Theorem~\ref{th properties of h} 
	is to investigate the inverse Fourier transform and the
	asymptotic behavior of~$\mathcal{ H}$ by scaling techniques.
	We thus need to overcome the additional difficulties caused by the
	fact that the operator that we take into account is not scale invariant, and therefore our analysis
cannot rely entirely on either the purely classical or the purely fractional counterparts, as given in~\cite{MR119247,MR3002595}.

To start with,  
 let us define a ``two-scales" function  for~$s\in(0,1)$, $t_1$, $t_2>0$ and~$x\in\mathbb{R}^n$ as 
\begin{equation}\label{definition of H}
	\mathcal{H}(x,t_1,t_2):=\int_{\mathbb{R}^n}e^{-(t_1|\xi|^{2s}+t_2|\xi|^{2})}e^{2\pi ix\cdot \xi}\,d\xi,
\end{equation} which satisfies the following rescaling properties
\begin{equation}\label{scaling  2}
		\mathcal{H}(x,t,t)=t^{-\frac{n}{2s}}\mathcal{H}(t^{-\frac{1}{2s}}x,1,t^{1-\frac{1}{s}})
		=t^{-\frac{n}{2}}\mathcal{H}(t^{-\frac{1}{2}}x,t^{1-s},1).
\end{equation}
We notice that
if~$t\in(1,+\infty)$, then~$t^{1-\frac{1}{s}}\in(0,1)$, while
if~$t\in(0,1)$, then~$t^{1-s}\in(0,1)$.
As a consequence of this, we shall take~$2s$-scaling for~$t\in(1,+\infty)$ and
take~$2$-scaling for~$t\in(0,1)$ in order to discuss the properties of the heat kernel~$\mathcal{ H}$.

\subsection{Nonnegativity of heat kernel }\label{subsec:Non-negativity  of kernel h}

We perform some auxiliary analysis on the kernel~$\mathcal{ H}(x,t)$ defined in~\eqref{DBSD-1}.
For this sake, we recall a result contained in~\cite[Lemma~A.2]{MR3002595}.

\begin{Lemma}\label{convolution}
	If~$f$ and~$g \in L^1(\mathbb{R}^n)$ are radially symmetric, nonnegative and decreasing in~$r=|x|$, then~$f\ast g$ is radially symmetric and decreasing in~$r$. 
\end{Lemma}

Then, we have:

\begin{Lemma}\label{lemma  H1}
	Let~$n\geqslant 1$ and~$s\in(0,1)$. Let~$\mathcal{ H}$ be defined as in~\eqref{DBSD-1}. 
	
Then, 
$\mathcal{H}$ is nonnegative, radially symmetric, and nonincreasing  with respect to~$r=|x|$.
\end{Lemma}

\begin{proof}
	We point out that, being the Fourier transform of a radially symmetric function, $\mathcal{H}(x,t)$ is radially symmetric in~$x$.  To prove the nonnegativity of~$\mathcal{H}$,
	we adapt the arguments in~\cite{MR3002595}. 
	
	For all~$s\in(0,1]$, we define the radially symmetric nonnegative functions~$f_s$ as
	\begin{equation*}
		f_s(x):= A_s\left(\frac{1}{|x|^{n+2s}}\,\chi_{\mathbb{R}^n\setminus B_{1}}(x)+\chi_{ B_{1}}(x)\right),
	\end{equation*}
	and, for any~$a>0$,
	\begin{equation*}
		g^a(x):= A_1a^{-\frac{n}{2}}\pi^{\frac{n}{2}}e^{-\frac{\pi^2}{a}|x|^2},
	\end{equation*}
	where~$A_s$ and~$A_1$ are such that~$\int_{\mathbb{R}^n}f_s(x)\, dx=1$ and~$\int_{\mathbb{R}^n}g^a(x)\, dx=1$, respectively.
	
	In this way, we have that 
	\begin{equation*}
		\begin{split}
			\mathcal{F}(f_s)(\xi)&=1+\int_{\mathbb{R}^n}\left(\cos(2\pi\xi\cdot x)-1\right)f_s(x)\, dx\\
			&=1+A_s|\xi|^{2s}\int_{|y|\geqslant|\xi|}\frac{\cos\left(2\pi y_1\right)-1}{|y|^{n+2s}}\,dy+A_s|\xi|^{-n}\int_{|y|\leqslant|\xi|}{\cos\left(2\pi y_1\right)-1}\,dy.
		\end{split}
	\end{equation*}
	Let also 
	\[ c:=-A_s\int_{\mathbb{R}^n}\frac{\cos\left(2\pi y_1\right)-1}{|y|^{n+2s}}\,dy. \]
	Then, it is immediate to check that
	$$\mathcal{F}(f_s)(\xi)=1-c|\xi|^{2s}(1+\omega(\xi)), $$
	where~$\omega(\xi)\to 0$ if~$\xi\to 0$. 
	
	We now define, for all~$k\in\mathbb{N}$,
	\[ f_k(x):=k^{\frac{n}{2s}}(f_s\ast f_s\ast \cdots \ast f_s)(k^{\frac{1}{2s}}x) ,\]
	the convolution product being taken~$k$ times. 
	
	By the properties of the Fourier transform with respect to the convolution product, one has that, for all~$\xi\in\mathbb{R}^n$,
	\begin{equation}\label{fk}
		\mathcal{F}(f_k)(\xi)=\left(\mathcal{F}(f_s)\left(\frac{\xi}{k^{1/(2s)}}\right)\right)^k=\left(1-c\frac{|\xi|^{2s}}{k}\left(1+\omega\left(\frac{\xi}{k^{1/(2s)}}\right)\right)\right)^k.
	\end{equation}

	We define, for all~$k\in \mathbb{N}$,
	\[ g_k^a(x):=f_k\ast g^a(x) \]
	and, by combining~\eqref{fk} and the fact that~$ \mathcal{F}(g^a)(\xi)=e^{-a|\xi|^2}$, we see that
	\[ \mathcal{F}(g_k^a)(\xi)= \left(1-c\frac{|\xi|^{2s}}{k}\left(1+\omega\left(\frac{\xi}{k^{1/(2s)}}\right)\right)\right)^ke^{-a|\xi|^2}.\]
	We note that the right-hand-side of the above equation converges to~$e^{-c|\xi|^{2s}-a|\xi|^{2}}$ pointwise as~$k\to+\infty$. Furthermore, since, for all~$\xi\in\mathbb{R}^n$,
	 \[ |\mathcal{F}(g_k^a)(\xi)|\leqslant \|g_k^a\|_{L^1(\mathbb{R}^n)}=1,
	 \]
	this convergence also holds in~$\mathcal{S}'(\mathbb{R}^n)$, that is,
	as~$k\to+\infty$,
	\[ \mathcal{F}(g_k^a)(\xi)\to e^{-c|\xi|^{2s}-a|\xi|^{2}}\qquad \text{in } \mathcal{S}'(\mathbb{R}^n). \]
	As a result, taking the inverse Fourier transform, we see that~$g_k^a$ converges in~$\mathcal{S}'(\mathbb{R}^n)$ to~$\mathcal{H}(x,c,a)$ for any~$a>0$. 
	
	Since~$g_k^a$ is nonnegative for all~$k$ and~$a>0$, we deduce the nonnegativity of~$\mathcal{H}(x,c,a)$. 
	
	By scaling and exploiting the continuity of~$\mathcal{H}(x,c,a)$ with respect to~$x$, one derives the nonnegativity of~$\mathcal{H}(x,t,t)$ for all~$t>0$. 

	The monotonicity of the heat kernel~$\mathcal{ H}$ can be also deduced by the fact that~$f_s$ and~$g^a$ are nonincreasing in~$r=|x|$ and Lemma~\ref{convolution}.
\end{proof}

\subsection{Asymptotic formulae for the heat kernel }\label{subsec:Bound on the kernel h}

To obtain the bounds on~$\mathcal{ H}$, we point out two  asymptotic formulae for~$\mathcal{ H}(x,t_1,t_2)$ as defined in~\eqref{definition of H}. These asymptotics are described in the next Lemmata~\ref{lemma 2s limit} and~\ref{lemma 2 limit}.

\begin{Lemma}\label{lemma 2s limit}
Let~$n\geqslant 1$, $ s\in(0,1)$, $\eta\in(0,1)$ and~$\mathcal{H}(x,1,\eta)$ be as defined by~\eqref{definition of H}.

Then, for any~$ \epsilon>0$, there exists~$M>0$ independent of~$\eta$
such that, for every~$|x|>M $ and~$\eta\in(0,1)$,
$$
	\left||x|^{n+2s}\mathcal{H}(x,1,\eta)-2^{n+2s}\pi^{\frac{n}{2}-1}s\sin(\pi s)\Gamma\left(\frac{n}{2}+s\right)\Gamma(s)\right|<\epsilon.
$$
That is,
$$	\lim\limits_{|x|\to+\infty}|x|^{n+2s}\mathcal{H}(x,1,\eta)=2^{n+2s}\pi^{\frac{n}{2}-1}s\sin(\pi s)\Gamma\left(\frac{n}{2}+s\right)\Gamma(s),
$$
and the limit is uniform with respect to~$\eta\in(0,1)$.
\end{Lemma}

	\begin{proof}
		 {F}rom the definition of~$\mathcal{H}(x,t_1,t_2)$, one has that
$$
			\mathcal{H}(x,1,\eta)=\int_{\mathbb{R}^n}e^{-|\xi|^{2s}-\eta|\xi|^{2}}e^{2\pi ix\cdot \xi}\,d\xi.
$$	By using the Fourier Inversion Theorem for radial functions, see e.g.~\cite[Chapter~\uppercase\expandafter{\romannumeral 2}]{MR31582}, and the Bessel functions in~\cite{MR58756}, it is immediate to obtain the $1$-dimensional integral representation of~$\mathcal{H}(x,1,\eta)$, namely that
		\begin{equation}\label{k(r)}
		\mathcal{H}(x,1,\eta)=\frac{(2\pi)^{\frac{n}{2}}}{|x|^{\frac{n}{2}-1}}\int_{0}^{+\infty} e^{-r^{2s}-\eta r^{2}} r^{\frac{n}{2}} J_{\frac{n}{2}-1}(|x|r)\, dr,
	\end{equation}
where~$J_{\frac{n}{2}-1}$ denotes the Bessel function of first kind of order~$ \frac{n}{2}-1$.

We also recall that, for all~$ z\in \mathbb{C}$,
	\begin{equation*}
		\frac{d}{dz}\left(z^{\frac{n}{2}}J_{\frac{n}{2}}(z)\right)=z^{\frac{n}{2}}J_{\frac{n}{2}-1}(z) 
	\end{equation*} 
	and \begin{equation*}
		J_{\frac{n}{2}}(z)\to \sqrt{\frac{2}{\pi z}}\cos \left(z-\frac{n\pi}{4}-\frac{\pi}{4}\right) \qquad {\mbox{when~$|z|\to+\infty\;$  and~$\;\arg z<\frac{\pi}{2}$.}}
	\end{equation*}
	Consequently,
	\begin{equation}
		\begin{split}\label{representation of H}
			\mathcal{H}(x,1,\eta)&=\frac{(2\pi)^{\frac{n}{2}}}{|x|^{\frac{n}{2}-1}}\int_{0}^{+\infty} e^{-r^{2s}-\eta r^{2}} r^{\frac{n}{2}} J_{\frac{n}{2}-1}(|x|r)\, dr\\
			&=\frac{(2\pi)^{\frac{n}{2}}}{|x|^n}\int_{0}^{+\infty} e^{-\frac{t^{2s}}{|x|^{2s}}-\eta\frac{t^{2}}{|x|^{2}}} \frac{d}{dt}\left(t^{\frac{n}{2}} J_{\frac{n}{2}}(t)\right)\, dt\\
			&={(2\pi)^{\frac{n}{2}}}{|x|^{-n}} e^{-\frac{t^{2s}}{|x|^{2s}}-\eta\frac{t^{2}}{|x|^{2}}} t^{\frac{n}{2}} J_{\frac{n}{2}}(t)\bigg|_{t=0}^{t=+\infty}\\
			&\qquad +{(2\pi)^{\frac{n}{2}}}{|x|^{-n}}\int_{0}^{+\infty} e^{-\frac{t^{2s}}{|x|^{2s}}-\eta\frac{t^{2}}{|x|^{2}}} \left(\frac{2st^{\frac{n}{2}+2s-1}}{|x|^{2s}}+\eta\frac{2t^{\frac{n}{2}+1}}{|x|^{2}}\right) J_{\frac{n}{2}}(t)\, dt\\
			&={(2\pi)^{\frac{n}{2}}}{|x|^{-n}}\int_{0}^{+\infty} e^{-\frac{t^{2s}}{|x|^{2s}}-\eta\frac{t^{2}}{|x|^{2}}} \left(\frac{2st^{\frac{n}{2}+2s-1}}{|x|^{2s}}+\eta\frac{2t^{\frac{n}{2}+1}}{|x|^{2}}\right) J_{\frac{n}{2}}(t)\, dt.
		\end{split}
	\end{equation}
		
	As a result,
	\begin{equation}
		\begin{split}\label{H(x,1,eta)}
			|x|^{n+2s}	\mathcal{H}(x,1,\eta)&={(2\pi)^{\frac{n}{2}}}\int_{0}^{+\infty} e^{-\frac{t^{2s}}{|x|^{2s}}-\eta\frac{t^{2}}{|x|^{2}}} \left({2st^{\frac{n}{2}+2s-1}}+2\eta{t^{\frac{n}{2}+1}}{|x|^{2s-2}}\right) J_{\frac{n}{2}}(t)\, dt\\
			&= {(2\pi)^{\frac{n}{2}}} \text{Re}\int_{0}^{+\infty} e^{-\frac{t^{2s}}{|x|^{2s}}-\eta\frac{t^{2}}{|x|^{2}}} \left({2st^{\frac{n}{2}+2s-1}}+2\eta{t^{\frac{n}{2}+1}}{|x|^{2s-2}}\right) H^{(1)}_{\frac{n}{2}}(t)\, dt,
		\end{split}
	\end{equation}
	where~$ H^{(1)}_{\frac{n}{2}}(z)$ is the Bessel function of the third kind 
	and Re~$A$ denotes the real part of~$A$.  
	
	We now choose
	a straight line~{$L_1$} running from~$0$ to~$+\infty$ in the upper half-plane and making a~$\frac{\pi}{6}$ angle with the positive real axis, that is
	\begin{equation}\label{definition of L1}
		L_1:=\left\{z\in\mathbb{C} : \arg z=\frac{\pi}{6}\right\}.
	\end{equation}
	Furthermore, applying the Residue Theorem, we find that 
	\begin{equation}
		\begin{split}\label{L1}
			&\int_{0}^{+\infty} e^{-\frac{t^{2s}}{|x|^{2s}}-\eta\frac{t^{2}}{|x|^{2}}} \left({2st^{\frac{n}{2}+2s-1}}+2\eta{t^{\frac{n}{2}+1}}{|x|^{2s-2}}\right) H^{(1)}_{\frac{n}{2}}(t)\, dt\\
			=&\int_{L_1}e^{-\frac{z^{2s}}{|x|^{2s}}-\eta\frac{z^{2}}{|x|^{2}}} \left({2sz^{\frac{n}{2}+2s-1}}+2\eta{z^{\frac{n}{2}+1}}{|x|^{2s-2}}\right) H^{(1)}_{\frac{n}{2}}(z)\, dz\\
			&\quad - \lim\limits_{R\to+\infty}\int_{C_R}e^{-\frac{z^{2s}}{|x|^{2s}}-\eta\frac{z^{2}}{|x|^{2}}} \left({2sz^{\frac{n}{2}+2s-1}}+2\eta{z^{\frac{n}{2}+1}}{|x|^{2s-2}}\right) H^{(1)}_{\frac{n}{2}}(z)\, dz,
		\end{split}
	\end{equation}
	where~$C_R$ denotes the arc
	$$C_R:=\left\{z\in \mathbb{C}\;{\mbox{ s.t. }}\; z=Re^{i\theta}, \;{\mbox{ with }}\theta\in\left(0,\frac{\pi}{6}\right)\right\}.$$
	
We claim that, for every~$|x|>1$ and~$\eta\in(0,1)$,
	\begin{equation}\label{R infty 0}
		\lim\limits_{R\to+\infty}\int_{C_R}e^{-\frac{z^{2s}}{|x|^{2s}}-\eta\frac{z^{2}}{|x|^{2}}} \left({2sz^{\frac{n}{2}+2s-1}}+2\eta{z^{\frac{n}{2}+1}}{|x|^{2s-2}}\right) H^{(1)}_{\frac{n}{2}}(z)\, dz=0.
	\end{equation}
Indeed, since~$z\in C_R$, one has that 
	\begin{equation}
		\begin{split}\label{step 1 estimate of H1}
			&\left|\int_{C_R}e^{-\frac{z^{2s}}{|x|^{2s}}-\eta\frac{z^{2}}{|x|^{2}}} \left({2sz^{\frac{n}{2}+2s-1}}+2\eta{z^{\frac{n}{2}+1}}{|x|^{2s-2}}\right) H^{(1)}_{\frac{n}{2}}(z)\, dz\right|\\
			=\;&\left|\int_{0}^{\frac{\pi}{6}}e^{-\frac{R^{2s}e^{i2s\theta}}{|x|^{2s}}-\eta\frac{R^{2}e^{i2\theta}}{|x|^{2}}} \left({2s(Re^{i\theta})^{\frac{n}{2}+2s-1}}+2\eta{(Re^{i\theta})^{\frac{n}{2}+1}}{|x|^{2s-2}}\right) H^{(1)}_{\frac{n}{2}}(Re^{i\theta})iRe^{i\theta}\, d\theta\right|\\
			\leqslant\;&\int_{0}^{\frac{\pi}{6}}e^{-\frac{R^{2s}\cos2s\theta}{|x|^{2s}}-\eta\frac{R^{2}\cos2\theta}{|x|^{2}}} \left({2sR^{\frac{n}{2}+2s}}+2\eta{R^{\frac{n}{2}+2}}{|x|^{2s-2}}\right) \left|H^{(1)}_{\frac{n}{2}}(Re^{i\theta})\right|\, d\theta\\
			\leqslant\;&e^{-\frac{R^{2s}}{2|x|^{2s}}-\eta\frac{R^{2}}{2|x|^{2}}} \left({2sR^{\frac{n}{2}+2s}}+2\eta{R^{\frac{n}{2}+2}}{|x|^{2s-2}}\right) \int_{0}^{\frac{\pi}{6}}\left|H^{(1)}_{\frac{n}{2}}(Re^{i\theta})\right|\, d\theta.
		\end{split}
	\end{equation}
	Also, from the fact that 
$$
		H^{(1)}_{\frac{n}{2}}(z)\to \sqrt{\frac{2}{\pi z}}e^{i(z-\frac{n\pi}{4}-\frac{\pi}{4})}\qquad {\mbox{when~$|z|\to+\infty\;$ and~$\;\arg z<\frac{\pi}{2}$,}}
$$	one obtains that
	\begin{equation}\label{estimate of H1}
		\lim_{R\to+\infty}\left|H^{(1)}_{\frac{n}{2}}(Re^{i\theta})\right|\leqslant
		\lim_{R\to+\infty} \left|\sqrt{\frac{2}{\pi R e^{i\theta}}}\right|= 0.
	\end{equation}
By combining~\eqref{estimate of H1} with~\eqref{step 1 estimate of H1}, we conclude that 
$$
		\lim_{R\to+\infty}\left|\int_{C_R}e^{-\frac{z^{2s}}{|x|^{2s}}-\eta\frac{z^{2}}{|x|^{2}}} \left({2sz^{\frac{n}{2}+2s-1}}+2\eta{z^{\frac{n}{2}+1}}{|x|^{2s-2}}\right) H^{(1)}_{\frac{n}{2}}(z)\, dz\right| = 0,
$$	which implies~\eqref{R infty 0}, as desired. 
	
	As a  consequence,	recalling~\eqref{H(x,1,eta)}, and inserting~\eqref{R infty 0} into~\eqref{L1}, we find that
	\begin{equation}
		\begin{split}\label{x H(x,1,eta)}
			&	|x|^{n+2s}	\mathcal{H}(x,1,\eta)
			= {(2\pi)^{\frac{n}{2}}} \text{Re}\int_{L_1} e^{-\frac{z^{2s}}{|x|^{2s}}-\eta\frac{z^{2}}{|x|^{2}}} \left({2sz^{\frac{n}{2}+2s-1}}+2\eta{z^{\frac{n}{2}+1}}{|x|^{2s-2}}\right) H^{(1)}_{\frac{n}{2}}(z)\, dz\\
			=\;&(2\pi)^{\frac{n}{2}}\text{Re}\int_{0}^{+\infty}e^{-\frac{\left(re^{i\frac{\pi}{6}}\right)^{2s}}{|x|^{2s}}-\eta\frac{\left(re^{i\frac{\pi}{6}}\right)^{2}}{|x|^{2}}} \left({2s\left(re^{i\frac{\pi}{6}}\right)^{\frac{n}{2}+2s-1}}+2\eta{\left(re^{i\frac{\pi}{6}}\right)^{\frac{n}{2}+1}}{|x|^{2s-2}}\right) H^{(1)}_{\frac{n}{2}}\left(re^{i\frac{\pi}{6}}\right)e^{i\frac{\pi}{6}}\, dr.
		\end{split}
	\end{equation}
	
Now we claim that for any~$ \epsilon>0$  there exists~$M>0$ independent of~$\eta$ such that, for every~$|x|>M$ and~$\eta\in(0,1)$,
	\begin{equation*}
		\left||x|^{n+2s}	\mathcal{H}(x,1,\eta)- (2\pi)^{\frac{n}{2}}\text{Re}\int_{L_1} {2sz^{\frac{n}{2}+2s-1}} H^{(1)}_{\frac{n}{2}}\left(z\right)\, dz\right|<\epsilon.
	\end{equation*}
	Indeed, 
	we first make use of the properties of~$ H^{(1)}_{\frac{n}{2}}\left(z\right) $ on~$L_1$ defined in~\eqref{definition of L1} to prove that the integral in the right-hand-side of~\eqref{x H(x,1,eta)} is uniformly bounded with respect to~$|x|>1$ and~$\eta\in(0,1)$. 
	
Let~$z\in L_1$, from the definition of~$ H^{(1)}_{\frac{n}{2}}\left(z\right) $
(see e.g.~\cite[page~21]{MR58756}), one has 
	\begin{equation}
		\begin{split}\label{estimate of H}
			&\left|H^{(1)}_{\frac{n}{2}}\left(z\right)\right|=\left|H^{(1)}_{\frac{n}{2}}\left(re^{i\frac{\pi}{6}}\right)\right|=\left|-ie^{-i\frac{n\pi}{4}}\int_{-\infty}^{+\infty}e^{ire^{i\frac{\pi}{6}}\frac{e^t+e^{-t}}{2}} e^{-\frac{nt}{2}}\, dt\right|\\
			&\qquad\leqslant \int_{0}^{+\infty}e^{-\frac{r}{4}\left(e^t+e^{-t}\right)}\left(e^{\frac{nt}{2}}+e^{\frac{-nt}{2}}\right)\,dt\leqslant 2\int_{0}^{+\infty}e^{-\frac{r}{4}e^t}e^{\frac{nt}{2}}\,dt.
		\end{split}
	\end{equation}
	Furthermore, we  derive that, for every~$\eta\in(0,1)$ and~$|x|>1$,
	\begin{equation}\label{uniform bounded in eta}
		\begin{split}
			&	\left|\int_{0}^{+\infty}e^{-\frac{\left(re^{i\frac{\pi}{6}}\right)^{2s}}{|x|^{2s}}-\eta\frac{\left(re^{i\frac{\pi}{6}}\right)^{2}}{|x|^{2}}} \left({2s\left(re^{i\frac{\pi}{6}}\right)^{\frac{n}{2}+2s-1}}+2\eta{\left(re^{i\frac{\pi}{6}}\right)^{\frac{n}{2}+1}}{|x|^{2s-2}}\right) H^{(1)}_{\frac{n}{2}}\left(re^{i\frac{\pi}{6}}\right)e^{i\frac{\pi}{6}}\, dr\right|\\
			\leqslant \;& \int_{0}^{+\infty} \left({2sr^{\frac{n}{2}+2s-1}}+2 r^{\frac{n}{2}+1}\right) \left|H^{(1)}_{\frac{n}{2}}\left(re^{i\frac{\pi}{6}}\right)\right|\, dr\\
			\leqslant\; & 4s\int_{0}^{+\infty} e^{\frac{nt}{2}} \int_{0}^{+\infty}e^{-\frac{r}{4}e^t} {r^{\frac{n}{2}+2s-1}} \, dr  \,dt+ 4\int_{0}^{+\infty} e^{\frac{nt}{2}} \int_{0}^{+\infty}e^{-\frac{r}{4}e^t} {r^{\frac{n}{2}+1}} \, dr  \,dt\\
			\leqslant\;& s4^{\frac{n}{2}+2s+1}\Gamma(\frac{n}{2}+2s)\int_{0}^{+\infty}e^{-2st}\,dt+ 4^{\frac{n}{2}+2}\Gamma(\frac{n}{2}+2)\int_{0}^{+\infty}e^{-2t}\,dt\\
			\leqslant \;& 2^{n+4s+1}\Gamma(\frac{n}{2}+2s)+2^{n+3}\Gamma(\frac{n}{2}+2).
		\end{split}
	\end{equation}
By combining~\eqref{x H(x,1,eta)} with~\eqref{uniform bounded in eta}, one has
	\begin{equation}
		\begin{split}\label{A1 A2}
			&\left||x|^{n+2s}	\mathcal{H}(x,1,\eta)-(2\pi)^{\frac{n}{2}}\text{Re}\int_{L_1} {2sz^{\frac{n}{2}+2s-1}} H^{(1)}_{\frac{n}{2}}\left(z\right)\,dz\right|\\
			=\;&	\left|{(2\pi)^{\frac{n}{2}}} \text{Re}\int_{L_1} \left(e^{-\frac{z^{2s}}{|x|^{2s}}-\eta\frac{z^{2}}{|x|^{2}}}-1\right) {2sz^{\frac{n}{2}+2s-1}} H^{(1)}_{\frac{n}{2}}(z)+e^{-\frac{z^{2s}}{|x|^{2s}}-\eta\frac{z^{2}}{|x|^{2}}}2\eta{z^{\frac{n}{2}+1}}{|x|^{2s-2}}H^{(1)}_{\frac{n}{2}}(z)\, dz\right|\\
			\leqslant \;&\left|{(2\pi)^{\frac{n}{2}}} \text{Re}\int_{L_1} \left(e^{-\frac{z^{2s}}{|x|^{2s}}-\eta\frac{z^{2}}{|x|^{2}}}-1\right) {2sz^{\frac{n}{2}+2s-1}} H^{(1)}_{\frac{n}{2}}(z)\, dz\right|+(2\pi)^{\frac{n}{2}}2^{n+3}\Gamma\left(\frac{n}{2}+2\right)|x|^{2s-2}\\
			=:\;&A_1+A_2.
		\end{split}
	\end{equation}
	
	Let us estimate the term~$A_1$. For some~$R>0$ large enough, one has
	\begin{equation}
		\begin{split}\label{A11}
			A_1&\leqslant (2\pi)^{\frac{n}{2}}\int_{0}^{+\infty}\left|e^{-\frac{\left(re^{i\frac{\pi}{6}}\right)^{2s}}{|x|^{2s}}-\eta\frac{\left(re^{i\frac{\pi}{6}}\right)^{2}}{|x|^{2}}}-1\right|{2sr^{\frac{n}{2}+2s-1}}\left|H^{(1)}_{\frac{n}{2}}\left(re^{i\frac{\pi}{6}}\right)\right|\, dr\\
			&\leqslant (2\pi)^{\frac{n}{2}}\int_{0}^{R}\left|e^{-\frac{\left(re^{i\frac{\pi}{6}}\right)^{2s}}{|x|^{2s}}-\eta\frac{\left(re^{i\frac{\pi}{6}}\right)^{2}}{|x|^{2}}}-1\right|{2sr^{\frac{n}{2}+2s-1}}\left|H^{(1)}_{\frac{n}{2}}\left(re^{i\frac{\pi}{6}}\right)\right|\, dr\\
			&\qquad + (2\pi)^{\frac{n}{2}}\int_{R}^{+\infty}\left|e^{-\frac{\left(re^{i\frac{\pi}{6}}\right)^{2s}}{|x|^{2s}}-\eta\frac{\left(re^{i\frac{\pi}{6}}\right)^{2}}{|x|^{2}}}-1\right|{2sr^{\frac{n}{2}+2s-1}}\left|H^{(1)}_{\frac{n}{2}}\left(re^{i\frac{\pi}{6}}\right)\right|\, dr,\\
			&=:B_1^{R}+B_2^{R}.
		\end{split}
	\end{equation}
	{F}rom~\eqref{uniform bounded in eta},	we observe that 
	\begin{equation*}
		\begin{split}
			B_2^{R}&\leqslant (2\pi)^{\frac{n}{2}}\int_{R}^{+\infty}{2sr^{\frac{n}{2}+2s-1}}\left|H^{(1)}_{\frac{n}{2}}\left(re^{i\frac{\pi}{6}}\right)\right|\, dr\leqslant (2\pi)^{\frac{n}{2}}4s\int_{0}^{+\infty} e^{\frac{nt}{2}} \int_{R}^{+\infty}e^{-\frac{r}{4}e^t} {r^{\frac{n}{2}+2s-1}} \, dr  \,dt\\
			&\leqslant (2\pi)^{\frac{n}{2}}4s\int_{0}^{+\infty} e^{\frac{nt}{2}} \int_{R}^{+\infty}e^{-\frac{r}{4}(t+1)} {r^{\frac{n}{2}+2s-1}} \, dr  \,dt\leqslant (2\pi)^{\frac{n}{2}}4s\int_{0}^{+\infty} e^{\frac{nt}{2}-\frac{Rt}{4}} \int_{R}^{+\infty}e^{-\frac{r}{4}} {r^{\frac{n}{2}+2s-1}} \, dr  \,dt\\&\leqslant\frac{(2\pi)^{\frac{n}{2}}s2^{n+4s+4}\Gamma(\frac{n}{2}+2s)}{R-2n}.
		\end{split}
	\end{equation*}
	As a result, for any~$\epsilon>0$, taking \begin{equation}\label{R_0}
		R_0:=\frac{4(2\pi)^{\frac{n}{2}}s2^{n+4s+4}\Gamma(\frac{n}{2}+2s)}{\epsilon}+2n,
	\end{equation}
one has that
	\begin{equation}\label{B2}
		B_2^{R_0}\leqslant \frac{\epsilon}{4}.
	\end{equation}
	We now turn to estimating~$B_1^{R_0}$. For this, we take~$R_0$ as in~\eqref{R_0} and we remark that
	\begin{equation*}
		\begin{split}
			B_1^{R_0}&\leqslant(2\pi)^{\frac{n}{2}}\int_{0}^{R_0}\bigg(\left|e^{-\frac{r^{2s}\cos\frac{s\pi}{3}}{|x|^{2s}}-\eta\frac{r^{2}\cos\frac{\pi}{3}}{|x|^{2}}}\cos\left(-\frac{r^{2s}\cos\frac{s\pi}{3}}{|x|^{2s}}-\eta\frac{r^{2}\cos\frac{\pi}{3}}{|x|^{2}}\right)-1\right|\\
			&\qquad+ \left|\sin\left(-\frac{r^{2s}\cos\frac{s\pi}{3}}{|x|^{2s}}-\eta\frac{r^{2}\cos\frac{\pi}{3}}{|x|^{2}}\right)\right|\bigg){2sr^{\frac{n}{2}+2s-1}}\left|H^{(1)}_{\frac{n}{2}}\left(re^{i\frac{\pi}{6}}\right)\right|\, dr.
		\end{split}
	\end{equation*}
	Note also that, for any~$r<R_0$,  $|x|>R_0^2$ large enough and~$\eta\in(0,1)$, 
	\begin{equation}
		\begin{split}\label{R_0 estimate}
			&\left|e^{-\frac{r^{2s}\cos\frac{s\pi}{3}}{|x|^{2s}}-\eta\frac{r^{2}\cos\frac{\pi}{3}}{|x|^{2}}}\cos\left(-\frac{r^{2s}\cos\frac{s\pi}{3}}{|x|^{2s}}-\eta\frac{r^{2}\cos\frac{\pi}{3}}{|x|^{2}}\right)-1\right|+\left|\sin\left(\frac{r^{2s}\cos\frac{s\pi}{3}}{|x|^{2s}}+\eta\frac{r^{2}}{2|x|^{2}}\right)\right|\\
			\leqslant \;&1-e^{-\frac{R_0^{2s}\cos\frac{s\pi}{3}}{|x|^{2s}}-\frac{R_0^{2}}{2|x|^{2}}}\cos\left(\frac{R_0^{2s}}{|x|^{2s}}+\frac{R_0^{2}}{2|x|^{2}}\right)+\sin\left(\frac{R_0^{2s}}{|x|^{2s}}+\frac{R_0^{2}}{2|x|^{2}}\right).
		\end{split}
	\end{equation}
	By combining~\eqref{uniform bounded in eta} with ~\eqref{R_0 estimate}, we see that for every~$\epsilon>0$, there exists~$T_{R_0,\epsilon}>0$  such that, if~$|x|>T_{R_0,\epsilon}$,
$$
		1-e^{-\frac{R_0^{2s}\cos\frac{s\pi}{3}}{|x|^{2s}}-\frac{R_0^{2}}{2|x|^{2}}}\cos\left(\frac{R_0^{2s}}{|x|^{2s}}+\frac{R_0^{2}}{2|x|^{2}}\right)+\sin\left(\frac{R_0^{2s}}{|x|^{2s}}+\frac{R_0^{2}}{2|x|^{2}}\right)\leqslant \frac{\epsilon}{2^{n+4s+3}(2\pi)^{\frac{n}{2}}\Gamma(\frac{n}{2}+2s)},
$$	which implies that 
	\begin{equation}\label{B3}
		B_1^{R_0}<\frac{\epsilon}{4}.
	\end{equation}
	Recalling~\eqref{A11}, by  combining~\eqref{B2} with~\eqref{B3}, we deduce that,
	for every~$\epsilon>0$, there exists~$M_1>0$,
	depending only on~$n$, $s$ and~$\epsilon$,
	such that, for every~$|x|>M_1$,
	\[ A_1\leqslant \frac{\epsilon}{2}. \]
	
	Regarding~$A_2$, from~\eqref{A1 A2}, we see that, for every~$\epsilon>0$, there exists~$M_2>0$, depending only on~$n$, $s$ and~$\epsilon$,  such that, for every~$|x|>M_2$,
	\[ A_2\leqslant \frac{\epsilon}{2}. \]
	 We take~$M:=\max\left\{M_1,M_2\right\}$ and we use~\eqref{A1 A2} to find that,  for every~$|x|>M$,
$$
		\left||x|^{n+2s}	\mathcal{H}(x,1,\eta)-(2\pi)^{\frac{n}{2}}\text{Re}\int_{L_1} {2sz^{\frac{n}{2}+2s-1}} H^{(1)}_{\frac{n}{2}}\left(z\right)\,dz\right|<\epsilon.
$$	That is,
	\begin{equation}\label{L21}
		\lim\limits_{|x|\to+\infty}|x|^{n+2s}	\mathcal{H}(x,1,\eta)=(2\pi)^{\frac{n}{2}}\text{Re}\int_{L_1} {2sz^{\frac{n}{2}+2s-1}} H^{(1)}_{\frac{n}{2}}\left(z\right)\,dz,
	\end{equation}
	where the limit is uniform with respect to~$\eta$.
	
	Applying the Residue Theorem again to the region composed of~$L_1$ and the straight line
	\begin{equation}\label{L2}
		L_2:=\left\{z\in\mathbb{C}\;{\mbox{ s.t. }}\;\arg z=\frac{\pi}{2}\right\},
	\end{equation}
 we obtain that 
$$
		\int_{L_1} {2sz^{\frac{n}{2}+2s-1}} H^{(1)}_{\frac{n}{2}}\left(z\right)\,dz=\int_{L_2} {2sz^{\frac{n}{2}+2s-1}} H^{(1)}_{\frac{n}{2}}\left(z\right)\,dz- \lim\limits_{R\to+\infty}\int_{C_R} {2sz^{\frac{n}{2}+2s-1}} H^{(1)}_{\frac{n}{2}}\left(z\right)\,dz,
$$	where $$
	C_R:=\left\{z\in\mathbb{C}\;{\mbox{ s.t. }} z=Re^{i\theta},\;{\mbox{ with }} \theta\in\left(\frac{\pi}{6},\frac{\pi}{2}\right)\right\}.$$ 
	
	We now claim that 
	\[  \lim\limits_{R\to+\infty}\int_{C_R} {2sz^{\frac{n}{2}+2s-1}} H^{(1)}_{\frac{n}{2}}\left(z\right)\,dz=0.\]
	Indeed, for~$R>0$ large enough, we have that
	\begin{equation*}
		\begin{split}
		\lim_{R\to+\infty}	\left|\int_{C_R} {2sz^{\frac{n}{2}+2s-1}} H^{(1)}_{\frac{n}{2}}\left(z\right)\,dz\right|&=\lim_{R\to+\infty}\left|\int_{\frac{\pi}{6}}^{\frac{\pi}{2}}iRe^{i\theta}2s\left(Re^{i\theta}\right)^{\frac{n}{2}+2s-1} H^{(1)}_{\frac{n}{2}}\left(Re^{i\theta}\right)\,d\theta\right|\\
			&\leqslant \lim_{R\to+\infty}\int_{\frac{\pi}{6}}^{\frac{\pi}{2}}2sR^{\frac{n}{2}+2s} 2\int_{0}^{+\infty}e^{-\frac{R}{4}e^t}e^{\frac{nt}{2}}\,dt\,d\theta\\
			&\leqslant\lim_{R\to+\infty} \frac{4\pi s}{3}R^{\frac{n}{2}+2s}\int_{0}^{+\infty}e^{-\frac{R}{4}(t+1)}e^{\frac{nt}{2}}\,dt=0.
		\end{split}
	\end{equation*}
	Hence, recalling~\eqref{L21} and exploiting the modified
Bessel function of the third kind~$K(x)$ (see e.g.~\cite[page~5]{MR58756}), 
	one finds that
	\begin{equation*}
		\begin{split}
			\lim\limits_{|x|\to+\infty}|x|^{n+2s}	\mathcal{H}(x,1,\eta)&=(2\pi)^{\frac{n}{2}}\text{Re}\int_{L_2} {2sz^{\frac{n}{2}+2s-1}} H^{(1)}_{\frac{n}{2}}\left(z\right)\,dz\\&=(2\pi)^{\frac{n}{2}}\text{Re}\int_{0}^{+\infty}2sr^{\frac{n}{2}+2s-1}e^{i\frac{\pi n}{4}}e^{i\pi s}H^{(1)}_{\frac{n}{2}}\left(ir\right)\,dr\\
			&=(2\pi)^{\frac{n}{2}}\text{Re}\int_{0}^{+\infty}2sr^{\frac{n}{2}+2s-1}(-i)e^{i\pi s}\frac{2}{\pi}K_{\frac{n}{2}}(r)\, dr\\
			&	=2^{\frac{n}{2}+2}\pi^{\frac{n}{2}-1}s\sin(\pi s)\int_{0}^{+\infty}r^{\frac{n}{2}+2s-1}K_{\frac{n}{2}}(r)\, dr\\
			&=2^{n+2s}\pi^{\frac{n}{2}-1}s\sin(\pi s)\Gamma\left(\frac{n}{2}+s\right)\Gamma(s).
		\end{split}
	\end{equation*}
	This last integral is evaluated on~\cite[page~51]{MR58756}, allowing us to complete the proof of Lemma~\ref{lemma 2s limit}. 
\end{proof}

\begin{Lemma}\label{lemma 2 limit}
	Let~$n\geqslant 1$, $s\in(0,1)$, $\eta\in(0,1)$ and~$\mathcal{H}(x,\eta,1)$ be as defined by~\eqref{definition of H}.
	
	Then,
	for any~$ \epsilon>0$, there exists~$M>0$ independent of~$\eta$ such that, for every~$|x|>M$ and~$\eta\in(0,1)$, 
$$
		\left||x|^{n+2s}\mathcal{H}(x,\eta,1)-2^{n+2s}\pi^{\frac{n}{2}-1}\eta s\sin(\pi s)\Gamma\left(\frac{n}{2}+s\right)\Gamma(s)\right|<\epsilon.
$$ That is
\begin{equation}\label{t<1}
	\lim\limits_{|x|\to+\infty}|x|^{n+2s}\mathcal{H}(x,\eta,1)=2^{n+2s}\pi^{\frac{n}{2}-1}\eta s\sin(\pi s)\Gamma\left(\frac{n}{2}+s\right)\Gamma(s),
\end{equation} 
and the limit is uniform with respect to~$\eta\in(0,1)$.
\end{Lemma}

\begin{proof}
	As  in the  proof of Lemma~\ref{lemma 2s limit},  first of all, the 1-dimensional integral representation of~$\mathcal{H}(x,\eta,1)$ follows from~\eqref{representation of H}, that is
$$
		\mathcal{H}(x,\eta,1)
		={(2\pi)^{\frac{n}{2}}}{|x|^{-n}}\int_{0}^{+\infty} e^{-\eta\frac{t^{2s}}{|x|^{2s}}-\frac{t^{2}}{|x|^{2}}} \left(\eta\frac{2st^{\frac{n}{2}+2s-1}}{|x|^{2s}}+\frac{2t^{\frac{n}{2}+1}}{|x|^{2}}\right) J_{\frac{n}{2}}(t)\, dt.
$$	It thus follows that 
	\begin{equation*}
		\begin{split}
			|x|^{n+2s}	\mathcal{H}(x,\eta,1)&={(2\pi)^{\frac{n}{2}}}\int_{0}^{+\infty} e^{-\eta\frac{t^{2s}}{|x|^{2s}}-\frac{t^{2}}{|x|^{2}}} \left({2\eta st^{\frac{n}{2}+2s-1}}+2{t^{\frac{n}{2}+1}}{|x|^{2s-2}}\right) J_{\frac{n}{2}}(t)\, dt\\
			&= {(2\pi)^{\frac{n}{2}}} \text{Re}\int_{0}^{+\infty} e^{-\eta\frac{t^{2s}}{|x|^{2s}}-\frac{t^{2}}{|x|^{2}}} \left({2\eta st^{\frac{n}{2}+2s-1}}+2{t^{\frac{n}{2}+1}}{|x|^{2s-2}}\right) H^{(1)}_{\frac{n}{2}}(t)\, dt.
		\end{split}
	\end{equation*}
	Exploiting  the Residue Theorem once again, and referring to~\eqref{L1} and~\eqref{R infty 0}, we infer that 
	\begin{equation*}
		\begin{split}
			&|x|^{n+2s}	\mathcal{H}(x,\eta,1)
			= {(2\pi)^{\frac{n}{2}}} \text{Re}\int_{L_1} e^{-\eta\frac{z^{2s}}{|x|^{2s}}-\frac{z^{2}}{|x|^{2}}} \left({2\eta sz^{\frac{n}{2}+2s-1}}+2{z^{\frac{n}{2}+1}}{|x|^{2s-2}}\right) H^{(1)}_{\frac{n}{2}}(z)\, dz\\
			=\;&(2\pi)^{\frac{n}{2}}\text{Re}\int_{0}^{+\infty}e^{-\eta\frac{\left(re^{i\frac{\pi}{6}}\right)^{2s}}{|x|^{2s}}-\frac{\left(re^{i\frac{\pi}{6}}\right)^{2}}{|x|^{2}}} \left({2\eta s\left(re^{i\frac{\pi}{6}}\right)^{\frac{n}{2}+2s-1}}+2{\left(re^{i\frac{\pi}{6}}\right)^{\frac{n}{2}+1}}{|x|^{2s-2}}\right) H^{(1)}_{\frac{n}{2}}\left(re^{i\frac{\pi}{6}}\right)e^{i\frac{\pi}{6}}\, dr.
		\end{split}
	\end{equation*}
	
	Moreover, owing to~\eqref{uniform bounded in eta}, we observe that,  for every~$\eta\in(0,1)$,
	\begin{equation}
		\begin{split}\label{Q1 Q2}
			&\left||x|^{n+2s}	\mathcal{H}(x,\eta,1)-(2\pi)^{\frac{n}{2}}\text{Re}\int_{L_1} {2\eta sz^{\frac{n}{2}+2s-1}} H^{(1)}_{\frac{n}{2}}\left(z\right)\,dz\right|\\
			=\;&	\left|{(2\pi)^{\frac{n}{2}}} \text{Re}\int_{L_1} \left(e^{-\eta\frac{z^{2s}}{|x|^{2s}}-\frac{z^{2}}{|x|^{2}}}-1\right) {2\eta sz^{\frac{n}{2}+2s-1}} H^{(1)}_{\frac{n}{2}}(z)+e^{-\eta\frac{z^{2s}}{|x|^{2s}}-\frac{z^{2}}{|x|^{2}}}2{z^{\frac{n}{2}+1}}{|x|^{2s-2}}H^{(1)}_{\frac{n}{2}}(z)\, dz\right|\\
			\leqslant\; &\left|{(2\pi)^{\frac{n}{2}}} \text{Re}\int_{L_1} \left(e^{-\eta\frac{z^{2s}}{|x|^{2s}}-\frac{z^{2}}{|x|^{2}}}-1\right) {2\eta sz^{\frac{n}{2}+2s-1}} H^{(1)}_{\frac{n}{2}}(z)\, dz\right|+(2\pi)^{\frac{n}{2}}2^{n+3}\Gamma\left(\frac{n}{2}+2\right)|x|^{2s-2}\\
			=:\;&Q_1+Q_2.
		\end{split}
	\end{equation}
	
	We now focus on the term~$Q_1$. For some~$R>0$ large enough, one deduces that
		\begin{equation}
		\begin{split}\label{Q1}
			Q_1&\leqslant (2\pi)^{\frac{n}{2}}\int_{0}^{+\infty}\left|e^{-\eta\frac{\left(re^{i\frac{\pi}{6}}\right)^{2s}}{|x|^{2s}}-\frac{\left(re^{i\frac{\pi}{6}}\right)^{2}}{|x|^{2}}}-1\right|{2\eta sr^{\frac{n}{2}+2s-1}}\left|H^{(1)}_{\frac{n}{2}}\left(re^{i\frac{\pi}{6}}\right)\right|\, dr\\
			&= (2\pi)^{\frac{n}{2}}\int_{0}^{R}\left|e^{-\eta\frac{\left(re^{i\frac{\pi}{6}}\right)^{2s}}{|x|^{2s}}-\frac{\left(re^{i\frac{\pi}{6}}\right)^{2}}{|x|^{2}}}-1\right|{2\eta sr^{\frac{n}{2}+2s-1}}\left|H^{(1)}_{\frac{n}{2}}\left(re^{i\frac{\pi}{6}}\right)\right|\, dr\\
			&\qquad + (2\pi)^{\frac{n}{2}}\int_{R}^{+\infty}\left|e^{-\eta\frac{\left(re^{i\frac{\pi}{6}}\right)^{2s}}{|x|^{2s}}-\frac{\left(re^{i\frac{\pi}{6}}\right)^{2}}{|x|^{2}}}-1\right|{2\eta sr^{\frac{n}{2}+2s-1}}\left|H^{(1)}_{\frac{n}{2}}\left(re^{i\frac{\pi}{6}}\right)\right|\, dr\\
			&=:Y_1^{R}+Y_2^{R}.
		\end{split}
	\end{equation}
	As for~$Y_2^R$, 	we notice that, utilizing~\eqref{uniform bounded in eta}, for every~$\eta\in(0,1)$,
$$
		Y_2^R\leqslant (2\pi)^{\frac{n}{2}}\int_{R}^{+\infty}{2sr^{\frac{n}{2}+2s-1}}\left|H^{(1)}_{\frac{n}{2}}\left(re^{i\frac{\pi}{6}}\right)\right|\, dr\leqslant\frac{(2\pi)^{\frac{n}{2}}s2^{n+4s+4}\Gamma(\frac{n}{2}+2s)}{R-2n}.
$$	Therefore, for any~$\epsilon>0$, picking 
	\begin{equation}\label{R_1}
	R_1:=\frac{4(2\pi)^{\frac{n}{2}}s2^{n+4s+4}\Gamma(\frac{n}{2}+2s)}{\epsilon}+2n,
	\end{equation}
one finds that
	\begin{equation}\label{Y2}
		Y_2^{R_1}\leqslant \frac{\epsilon}{4}.
	\end{equation}
	As regards~$Y_1^{R_1}$, if~$R_1$ is as defined in~\eqref{R_1} we have that
	\begin{equation*}
		\begin{split}
			Y_1^{R_1}&\leqslant(2\pi)^{\frac{n}{2}}\int_{0}^{R_1}\bigg(\left|e^{-\eta\frac{r^{2s}\cos\frac{s\pi}{3}}{|x|^{2s}}-\frac{r^{2}\cos\frac{\pi}{3}}{|x|^{2}}}\cos\left(-\eta\frac{r^{2s}\cos\frac{s\pi}{3}}{|x|^{2s}}-\frac{r^{2}\cos\frac{\pi}{3}}{|x|^{2}}\right)-1\right|\\
			&\qquad+ \left|\sin\left(-\eta\frac{r^{2s}\cos\frac{s\pi}{3}}{|x|^{2s}}-\frac{r^{2}\cos\frac{\pi}{3}}{|x|^{2}}\right)\right|
			{2 sr^{\frac{n}{2}+2s-1}}\left|H^{(1)}_{\frac{n}{2}}\left(re^{i\frac{\pi}{6}}\right)\right|\, dr.
		\end{split}
	\end{equation*}
	We notice that, for any~$r<R_1$, $|x|>R_1^2$ large enough and~$\eta\in(0,1)$, 
	\begin{equation}
		\begin{split}\label{R_1 estimate}
			&\left|e^{-\eta\frac{r^{2s}\cos\frac{s\pi}{3}}{|x|^{2s}}-\frac{r^{2}\cos\frac{\pi}{3}}{|x|^{2}}}\cos\left(-\eta\frac{r^{2s}\cos\frac{s\pi}{3}}{|x|^{2s}}-\frac{r^{2}\cos\frac{\pi}{3}}{|x|^{2}}\right)-1\right|+\left|\sin\left(\eta\frac{r^{2s}\cos\frac{s\pi}{3}}{|x|^{2s}}+\frac{r^{2}}{2|x|^{2}}\right)\right|\\
			\leqslant \;&1-e^{-\frac{R_1^{2s}\cos\frac{s\pi}{3}}{|x|^{2s}}-\frac{R_1^{2}}{2|x|^{2}}}\cos\left(\frac{R_1^{2s}}{|x|^{2s}}+\frac{R_1^{2}}{2|x|^{2}}\right)+\sin\left(\frac{R_1^{2s}}{|x|^{2s}}+\frac{R_1^{2}}{2|x|^{2}}\right),
		\end{split}
	\end{equation}
	Utilizing~\eqref{R_1 estimate}, we  obtain that for every~$\epsilon>0$, there exists~$T_{R_1,\epsilon}>0$  such that, for every~$|x|>T_{R_1,\epsilon}$,
$$
		1-e^{-\frac{R_1^{2s}\cos\frac{s\pi}{3}}{|x|^{2s}}-\frac{R_1^{2}}{2|x|^{2}}}\cos\left(\frac{R_1^{2s}}{|x|^{2s}}+\frac{R_1^{2}}{2|x|^{2}}\right)+\sin\left(\frac{R_1^{2s}}{|x|^{2s}}+\frac{R_1^{2}}{2|x|^{2}}\right)\leqslant \frac{\epsilon}{2^{n+4s+3}(2\pi)^{\frac{n}{2}}\Gamma(\frac{n}{2}+2s)}.
$$	As a consequence of this and~\eqref{uniform bounded in eta}, one obtains that
	\begin{equation}\label{Y3}
		Y_1^{R_1}<\frac{\epsilon}{4}.
	\end{equation}
	Recalling~\eqref{Q1}, by combining~\eqref{Y2} with~\eqref{Y3}, we deduce that
	for every~$\epsilon>0$, there exists~$M_3>0$, depending
	only on~$n$, $s$ and~$\epsilon$, such that, for every~$|x|>M_3$,
	\[ Q_1\leqslant \frac{\epsilon}{2}. \]
	
Regarding~$Q_2$, thanks to~\eqref{Q1 Q2}, we  see that 	for every~$\epsilon>0$, there exists~$M_4>0$, depending only on~$n$, $s$ and~$\epsilon$,
such that, for every~$|x|>M_4$,
	\[ Q_2\leqslant \frac{\epsilon}{2}. \]
Taking~$M:=\max\left\{M_3, M_4\right\}$, it follows from~\eqref{Q1 Q2} that,  for every~$|x|>M$ and~$\eta\in(0,1)$,
$$
		\left||x|^{n+2s}	\mathcal{H}(x,\eta,1)-(2\pi)^{\frac{n}{2}}\text{Re}\int_{L_1} {2\eta sz^{\frac{n}{2}+2s-1}} H^{(1)}_{\frac{n}{2}}\left(z\right)\,dz\right|<\epsilon.
$$	That is,
	\begin{equation*}
		\lim\limits_{|x|\to+\infty}|x|^{n+2s}	\mathcal{H}(x,\eta,1)=(2\pi)^{\frac{n}{2}}\text{Re}\int_{L_1} {2\eta sz^{\frac{n}{2}+2s-1}} H^{(1)}_{\frac{n}{2}}\left(z\right)\,dz,
	\end{equation*}
and the limit is uniform with respect to~$\eta$.

As a consequence of this, 
	by using the modified
	Bessel function of the third kind~$K(x)$ again (see e.g.~\cite[page~5]{MR58756}), 
	it is not difficult  to check that
	\begin{equation*}
		\begin{split}
			\lim\limits_{|x|\to+\infty}|x|^{n+2s}	\mathcal{H}(x,\eta,1)&=(2\pi)^{\frac{n}{2}}\text{Re}\int_{L_2} {2\eta sz^{\frac{n}{2}+2s-1}} H^{(1)}_{\frac{n}{2}}\left(z\right)\,dz\\
			&=2^{n+2s}\pi^{\frac{n}{2}-1}\eta s\sin(\pi s)\Gamma\left(\frac{n}{2}+s\right)\Gamma(s).
		\end{split}
	\end{equation*}
	Thus, we obtain~\eqref{t<1}, as desired.
\end{proof}

\subsection{Bounds on the heat kernel }\label{subsec:Proof of Lemma{lemma h b}}

In this section, we establish the following result related to the upper and lower bounds on the heat  kernel~$\mathcal{ H}$ by using the two asymptotic formulae in Lemmata~\ref{lemma 2s limit} and~\ref{lemma 2 limit}.

\begin{Lemma}\label{lemma h b}
	Let~$n\geqslant 1$ and~$s\in(0,1)$. Let~$\mathcal{H}$ be as defined in~\eqref{DBSD-1}.
	
	 Then,
	there exist positive constants~$C_1$ and~$C_2$ such that
	\begin{equation}\label{decay of H11 }
		\mathcal{H}(x,t)\leqslant C_1 \left\{\frac{t}{|x|^{n+2s}}\vee \frac{t^s}{|x|^{n+2s}}\right\}\wedge \left\{ t^{-\frac{n}{2s}}\wedge  t^{-\frac{n}{2}}\right\};
	\end{equation}
	\begin{equation}\label{decay of H22 }
		\text{and } \quad		\mathcal{H}(x,t)\geqslant C_2\begin{cases}
			\frac{t}{|x|^{n+2s}} \qquad&\mbox{if }\; 1<t<|x|^{2s} ;\\
			e^{\frac{\pi|x|^2}{t}}t^{-\frac{n}{2}} \qquad&\mbox{if }\; |x|^2<t<|x|^{2s}<1,
		\end{cases}
	\end{equation}
where~$a\wedge b:=\min\left\{a,b\right\}$ and~$a\vee b:=\max\left\{a,b\right\}$.
\end{Lemma}

\begin{proof}
We observe that,	by the definition of~$\mathcal{H}$ in~\eqref{DBSD-1}, for every~$ t>0$,
	\begin{equation}\label{h1}
		\mathcal{ H}(x,t)\leqslant c_1\left\{t^{-\frac{n}{2}}\wedge t^{-\frac{n}{2s}}\right\}  
	\end{equation} 
	for some constant~$c_1$ depending only on~$n$ and~$s$.

In addition,	in the light of Lemma~\ref{lemma 2s limit}, we know that there exists~$M>0$ independent of~$\eta$ such that,
for all~$\eta\in(0,1)$ and~$ |x|>M$,
	\begin{equation}\label{alpha 1}
		\frac{\alpha}{2}<|x|^{n+2s}\mathcal{H}(x,1,\eta)<2\alpha,
	\end{equation}
	where~$\alpha=2^{n+2s}\pi^{\frac{n}{2}-1}s\sin(\pi s)\Gamma(\frac{n}{2}+s)\Gamma(s)$.
	
	Furthermore, for every~$t\in(1,+\infty)$, one has that~$t^{1-\frac{1}{s}}\in(0,1)$. Hence, by combining~\eqref{scaling  2} with~\eqref{alpha 1}, one concludes that
$$
		\frac{\alpha}{2}<	|t^{-\frac{1}{2s}}x|^{n+2s}\mathcal{H}(t^{-\frac{1}{2s}}x,1,t^{1-\frac{1}{s}})<2\alpha \quad \text{if } |x|>Mt^{\frac{1}{2s}}.
$$
That is,
$$
		\frac{\alpha t^{\frac{n}{2s}+1}}{2|x|^{n+2s}}<	\mathcal{H}(t^{-\frac{1}{2s}}x,1,t^{1-\frac{1}{s}})<\frac{2\alpha t^{\frac{n}{2s}+1}}{|x|^{n+2s}} \quad \text{if } |x|>Mt^{\frac{1}{2s}}.
$$	Exploiting~\eqref{scaling  2}, one sees that, if~$ t>1$ and~$|x|>Mt^{\frac{1}{2s}}$,
	\begin{equation}\label{t>1 x>m}
		\frac{\alpha t}{2|x|^{n+2s}}<	\mathcal{H}(x,t)<\frac{2\alpha t}{|x|^{n+2s}}.
	\end{equation}
Similarly, for every~$t\in(0,1)$, we have that~$t^{1-s}\in(0,1)$. 
In this way, using the nonnegativity of~$\mathcal{ H}$, by combining~\eqref{scaling  2} with Lemma~\ref{lemma 2 limit},   we know that there exists~$M>0$ independent of~$\eta$ such that, for every~$\eta\in(0,1)$, 
$t\in(0,1)$ and~$ |x|>Mt^{\frac{1}{2}}$,
	\begin{equation}\label{t<1 x>m}
			0\leqslant	\mathcal{H}(x,t)<\frac{2\alpha t^s}{|x|^{n+2s}}.
		\end{equation}
Gathering~\eqref{h1}, \eqref{t>1 x>m} and~\eqref{t<1 x>m}, we obtain~\eqref{decay of H11 }, as desired.

Let us now  estimate~\eqref{decay of H22 }.  We first claim that for any~$\epsilon>0$, there exists~$\delta\in(0,1)$ such that, for every~$|x|<\delta$ and~$|x|^2\leqslant t\leqslant |x|^{2s}$,  
\begin{equation}\label{H}
	\left|	\frac{\mathcal{H}(x,t)}{e^{-\pi\frac{|x|^2}{t}}t^{-\frac{n}{2}}} -\frac{\omega_n\Gamma(\frac{n}{2})}{2}\right|<\epsilon.
\end{equation}
Indeed, 
from the definition of~$\mathcal{H}$ it follows that 
\begin{equation*}
	\begin{split}&
		\mathcal{H}(x,t)= \int_{\mathbb{R}^n}e^{-t(|\xi|^{2s}+|\xi|^{2})+2\pi ix\cdot \xi}\,d\xi
		=e^{-\frac{\pi|x|^2}{t}}\int_{\mathbb{R}^n}e^{-\left(t^{\frac{1}{2}}\xi-i\pi xt^{-\frac{1}{2}}\right)^2}e^{-t|\xi|^{2s}}\,d\xi\\
		&\qquad=e^{-\frac{\pi|x|^2}{t}}t^{-\frac{n}{2}}\int_{\mathbb{R}^n}e^{-|y-i\pi xt^{-\frac{1}{2}}|^2}e^{-t^{1-s}|y|^{2s}}\,dy.
	\end{split}
\end{equation*}
Accordingly, recalling the fact that~$|x|^2\leqslant t\leqslant |x|^{2s}$ and  taking~$\delta=\left(\frac{\epsilon}{\omega_{n}\Gamma(\frac{n}{2}+s) e^{\pi^2}}\right)^{\frac{1}{2s(1-s)}}$,  one  derives that
	\begin{equation*}
		\begin{split}
			&\left|	\frac{\mathcal{H}(x,t)}{e^{-\pi\frac{|x|^2}{t}}t^{-\frac{n}{2}}} -\int_{\mathbb{R}^n}e^{-|y-i\pi xt^{-\frac{1}{2}}|^2}
			\, dy\right|=\left|\int_{\mathbb{R}^n}e^{-|y-i\pi xt^{-\frac{1}{2}}|^2}\left(1-e^{-t^{1-s}|y|^{2s}}\right)\,dy\right|\\
			&\qquad\qquad\leqslant \int_{\mathbb{R}^n}e^{-|y|^2+\pi^2}t^{1-s}|y|^{2s}\,dy\leqslant \omega_{n}\Gamma\left(\frac{n}{2}+s\right) t^{1-s}e^{\pi^2} \\
			&\qquad \qquad\leqslant \omega_{n}\Gamma\left(\frac{n}{2}+s\right) |x|^{2s(1-s)}e^{\pi^2}\leqslant\epsilon. 
		\end{split}
	\end{equation*}
Subsequently, applying the Residue Theorem, we observe that
	\begin{equation*}
\int_{\mathbb{R}^n}e^{-|y-i\pi xt^{-\frac{1}{2}}|^2}
			\, dy=\int_{\mathbb{R}^n-i\pi xt^{-\frac{1}{2}}}e^{-|z|^2}
			\, dz= \int_{\mathbb{R}^n}e^{-|z|^2}
			\, dz=\frac{\omega_n\Gamma(\frac{n}{2})}{2}
	\end{equation*}
	and we obtain~\eqref{H}, as desired.
	
	Taking now~$\epsilon:=\frac{\omega_n\Gamma(\frac{n}{2})}{4}$ in~\eqref{H}
	and utilizing~\eqref{t>1 x>m}, we establish the desired result in~\eqref{decay of H22 }.
\end{proof}

\section{Properties of the Bessel kernel}\label{sce properties of k}

The main aim of this section is to prove Theorem~\ref{th properties of k} by exploiting the properties of the heat kernel~$\mathcal{ H}$ defined in~\eqref{DBSD-1}.

To start with, in view of the definition of~$\mathcal{ K}$, we observe that,
from the nonnegativity, radial symmetry  and monotonicity of~$\mathcal{H}(\cdot,t)$ for all~$t>0$, it follows that the kernel~$\mathcal{ K}$ is nonnegative,  radially symmetric  and nonincreasing in~$r=|x|$.

\subsection{Decay of the Bessel kernel}\label{subsec:Decay of the kernel k}

In this part, our goal  is to 
employ the lower and upper bounds on the heat kernel~$\mathcal{ H}$
(see e.g.~\eqref{decay of H11 } and~\eqref{decay of H22 }) to obtain the decay of~$\mathcal{K}$, as given by the following Lemmata~\ref{proposition t>1} and~\ref{decay of K_2}.

\begin{Lemma}\label{proposition t>1}
	Let~$n\geqslant 1$ and~$s\in(0,1)$. Let~$\mathcal{K}$ be as defined in~\eqref{definition of K}.
	
	Then, if~$|x|\geqslant 1$,
	\begin{equation}\label{decay of K_1}
		\mathcal{K}(x)\leqslant \frac{c_1}{|x|^{n+2s}}
	\end{equation}
and, if~$|x|\leqslant 1$,
$$
	 \mathcal{K}(x)\leqslant c_2\begin{cases}
		|x|^{2-n}\qquad &\text{ if } n\geqslant 3,\\
		1+|\ln|x|| &\text{ if } n=2,\\
		1 & \text{ if }n=1,
	\end{cases}
$$	for some positive constants~$c_1$ and~$c_2$ depending on~$n$ and~$s.$ 		
\end{Lemma}

\begin{proof}
{F}rom~\eqref{decay of H11 }, we derive that 
\begin{eqnarray*}
	&&	0\leqslant	\mathcal{H}(x,t)\leqslant \frac{1}{C_1} \left\{t^{-\frac{n}{2s}}\wedge \frac{t}{|x|^{n+2s}}\right\}  \qquad {\mbox{for all~$t>1$,}}\\
{\mbox{and }}\quad
&&	0\leqslant	\mathcal{H}(x,t)\leqslant \frac{1}{C_1} \left\{t^{-\frac{n}{2}}\wedge \frac{t^s}{|x|^{n+2s}}\right\} 
\qquad  {\mbox{for all~$t\in(0,1)$.}}
\end{eqnarray*}
{F}rom these formulas we conclude that, if~$|x|>1$,
	\begin{equation}
		\begin{split}\label{K_1 x>m}
				\mathcal{K}(x)&=\int_{1}^{+\infty} e^{-t} \mathcal{H}(x,t) \,dt+\int_{0}^{1} e^{-t} \mathcal{H}(x,t) \,dt\\
			&\leqslant \frac{1}{C_1} \left(\int_{1}^{\left|{x}\right|^{2s}}e^{-t}\frac{ t}{|x|^{n+2s}}\, dt+\int_{\left|{x}\right|^{2s}}^{+\infty}e^{-t}t^{-\frac{n}{2s}}\, dt+\int_{0}^{1}e^{-t}\frac{ t^s}{|x|^{n+2s}}\, dt\right)\\
			&\leqslant \frac{2\Gamma(2)+\Gamma(s+1)}{C_1|x|^{n+2s}}
		\end{split}
	\end{equation}
	and, if~$|x|\leqslant 1$,
	\begin{equation}
		\begin{split}\label{K_1 x<m}
			\mathcal{K}(x)&=\int_{1}^{+\infty} e^{-t} \mathcal{H}(x,t) \,dt+\int_{0}^{1} e^{-t} \mathcal{H}(x,t) \,dt\\
		&\leqslant \frac{1}{C_1} \left(\int_{1}^{+\infty}e^{-t}t^{-\frac{n}{2s}}\, dt+\int_{0}^{\left|{x}\right|^{2}} e^{-t} \frac{ t^s}{|x|^{n+2s}}\,dt+ \int_{\left|{x}\right|^{2}}^{1}e^{-t}t^{-\frac{n}{2}}\, dt\right)\\
			&\leqslant
			\begin{cases}
			 \frac{n+2 }{C_1(n-2)|x|^{n-2}}\qquad\qquad & {\mbox{ if }}
			 n\geqslant 3\\
			 4(1+|\ln|x|| ) & {\mbox{ if }}n=2\\
			 1+\Gamma\left(\frac{1}{2}\right)+|x| & {\mbox{ if }}n=1.
		 \end{cases}
		\end{split}
	\end{equation}
By combining~\eqref{K_1 x>m} with~\eqref{K_1 x<m}, we obtain~\eqref{decay of K_1}.
\end{proof}


We shall make use of the nonnegativity of~$\mathcal{ K}$
and~\eqref{decay of H22 }  to prove the strict positivity of~$\mathcal{K}$ and  a lower bound on the behaviour of~$\mathcal{K}$.

\begin{Lemma}\label{decay of K_2}
	Let~$n\geqslant 1 $ and~$s\in(0,1)$. Let~$\mathcal{K}$ be as in~\eqref{definition of K}.
	
	Then, $\mathcal{ K}$ is positive and 
$$
		\mathcal{K}(x)\geqslant \frac{c_3}{|x|^{n+2s}}\quad\text{ if } |x|>1\qquad \text{and} \qquad \mathcal{K}(x)\geqslant \frac{c_4}{|x|^{n-2}}\quad \text{ if } |x|\leqslant 1
$$
	for some constants~$c_3$ and~$c_4$ depending only on~$n$ and~$s.$ 
\end{Lemma}

\begin{proof}
{F}rom~\eqref{decay of H22 }, we know that 
$$	\mathcal{H}(x,t)\geqslant C_2
		\frac{t}{|x|^{n+2s}} \qquad \mbox{if }\; 1<t<|x|^{2s}. 
$$	Recalling the fact that~$\mathcal{H}$ is nonnegative, one  deduces that, for any~$|x|>2$,
	\begin{equation}\label{low bounded}
			\mathcal{K}(x)=\int_{0}^{+\infty} e^{-t}\mathcal{H}(x,t)\, dt\geqslant C_2\int_{1}^{\left|{x}\right|^{2s}}e^{-t}\mathcal{H}(x,t)\, dt
			\geqslant C_2\int_{1}^{2}e^{-t}\frac{ t}{|x|^{n+2s}}\, dt\geqslant \frac{ e^{-2}}{|x|^{n+2s}}. 
	\end{equation}
	
	Moreover, in light of~\eqref{decay of H22 }, for every~$|x|^2<t<|x|^{2s}<1$,
	\[ {\mathcal{H}(x,t)} >C_2 e^{-\pi\frac{|x|^2}{t}}t^{-\frac{n}{2}}. \]
	As a consequence, by the definition of~$\mathcal{K},$  one has that, for every~$|x|<\frac{1}{2}$, 
	\begin{equation}\begin{split}\label{low bound}
			&\mathcal{K}(x)\geqslant \int_{|x|^2}^{|x|^{2s}} C_2 e^{-\pi\frac{|x|^2}{t}}t^{-\frac{n}{2}} e^{-t}\, dt
			\geqslant e^{-1} C_2\int_{|x|^{2-2s}}^{1}e^{-\pi y}|x|^{-n+2}y^{\frac{n}{2}-2}\, dy\\
			&\qquad\qquad\geqslant \frac{e^{-1} C_2}{|x|^{n-2}}\int_{{\frac{1}{2}}^{2-2s}}^{1} e^{-\pi y}y^{\frac{n}{2}-2}\,dy=\frac{c_{n,s}}{|x|^{n-2}}.
		\end{split}
	\end{equation}

	In addition, we recall that~$\mathcal{K}\geqslant 0$, and~$\mathcal{K}$ is nonincreasing in~$r=|x|$. According to~\eqref{low bounded} and~\eqref{low bound}, one concludes that~$\mathcal{K}(x)>0 $ for every~$x\in\mathbb{R}^n$.  
\end{proof}

\subsection{Smoothness of the Bessel kernel}\label{subsec:Smoothness of kernel}

In this section, we focus on the smoothness of the kernel~$\mathcal{ K}$ given by~\eqref{definition of K}. These properties of~$\mathcal{ K}$  are useful in the proofs of our regularity results.

\begin{Lemma}\label{sommothness of K}
	Let~$n\ge1$ and~$s\in(0,1)$. Let~$\mathcal{K}$ be as in~\eqref{definition of K}.
	
	Then, there exists~$C>0$, depending only on~$n$ and~$s$, such that, for all~$ |x|\geqslant 1$,
	\begin{equation*}
		|\nabla\mathcal{K}(x)|\leqslant \frac{C}{|x|^{n+2s+1}}\quad \text{and} \quad |D^2\mathcal{ K}(x)|\leqslant \frac{C}{|x|^{n+2s+2}}.
	\end{equation*}
\end{Lemma}

\begin{proof}
	Since~$\mathcal{K}$ is radially symmetric, with a slight abuse of notation we write~$\mathcal{K}(x)=\mathcal{K}(r)$, with~$r=|x|$. Furthermore, recalling the definition of~$\mathcal{K}$ and employing~\eqref{k(r)}, we infer that 
$$
		\mathcal{K}(r)=\frac{(2\pi)^{\frac{n}{2}}}{r^{\frac{n}{2}-1}}\int_{0}^{+\infty}\int_{0}^{+\infty}e^{-t(1+\tau^{2s}+\tau^2)}\tau^{\frac{n}{2}}J_{\frac{n}{2}-1}(r\tau)\, d\tau\, dt.
$$	Taking the derivative of the expression above with respect to~$r$, one finds that 
	\begin{equation}
		\begin{split}\label{K'}
			\mathcal{K}'(r)=&\frac{-(2\pi)^{\frac{n}{2}}(n-2)}{2r}\mathcal{K}(r)+\frac{(2\pi)^{\frac{n}{2}}}{r^{\frac{n}{2}}}\int_{0}^{+\infty}\int_{0}^{+\infty}e^{-t(1+\tau^{2s}+\tau^2)}\tau^{\frac{n}{2}}J'_{\frac{n}{2}-1}(r\tau)\tau r\, d\tau\, dt\\
			:=&I_1(r)+I_2(r).
		\end{split}
	\end{equation}
Concerning~$I_1$, owing to~\eqref{decay of K_1}, one has that, for every~$r>1$,
\begin{equation*}
	|I_1(r)|\leqslant \frac{c_1(2\pi)^{\frac{n}{2}}|n-2|}{2r^{n+2s+1}}.
\end{equation*}
We now estimate~$I_2$. For this, integrating by parts in~$\tau$, we see that 
	\begin{equation*}
		\begin{split}
			I_2(r)&=\frac{(2\pi)^{\frac{n}{2}}}{r^{\frac{n}{2}}}\int_{0}^{+\infty}e^{-t(1+\tau^{2s}+\tau^2)}\tau^{\frac{n}{2}+1}J_{\frac{n}{2}-1}(r\tau)\bigg|_{\tau=0}^{\tau=+\infty}\, dt\\
			& \quad-\frac{(2\pi)^{\frac{n}{2}}}{r^{\frac{n}{2}}}\int_{0}^{+\infty}\int_{0}^{+\infty}e^{-t(1+\tau^{2s}+\tau^2)}J_{\frac{n}{2}-1}(r\tau)\left(\left(\frac{n}{2}+1\right)\tau^{\frac{n}{2}}-t(2s\tau^{2s}+2\tau^2)\tau^{\frac{n}{2}}\right)\, d\tau\, dt\\
			&=-\frac{n+2}{2r}\mathcal{K}(r)+\frac{(2\pi)^{\frac{n}{2}}}{r^{\frac{n}{2}}}\int_{0}^{+\infty}\int_{0}^{+\infty}e^{-t(1+\tau^{2s}+\tau^2)}\tau^{\frac{n}{2}}J_{\frac{n}{2}-1}(r\tau)\left(2ts\tau^{2s}+2t\tau^2\right)\, d\tau\, dt.
		\end{split}
	\end{equation*}
	Furthermore, integrating by parts in~$t$,  
	\begin{equation*}
		\begin{split}
			&2s\frac{(2\pi)^{\frac{n}{2}}}{r^{\frac{n}{2}}}\int_{0}^{+\infty}\int_{0}^{+\infty}e^{-t(1+\tau^{2s}+\tau^2)}\tau^{\frac{n}{2}}J_{\frac{n}{2}-1}(r\tau)t\tau^{2s}\, d\tau\, dt\\
			=	\;& \frac{(2\pi)^{\frac{n}{2}}}{r^{\frac{n}{2}}}\int_{0}^{+\infty}\frac{te^{-t(1+\tau^{2})}e^{-t\tau^{2s}}}{-\tau^{2s}}\tau^{\frac{n}{2}}J_{\frac{n}{2}-1}(r\tau)2s\tau^{2s}\bigg|_{t=0}^{t=+\infty}\, d\tau\\
			&\quad +2s\frac{(2\pi)^{\frac{n}{2}}}{r^{\frac{n}{2}}}\int_{0}^{+\infty}\int_{0}^{+\infty}\frac{e^{-t\tau^{2s}}}{\tau^{2s}}\left(e^{-t(1+\tau^{2})}-te^{-t(1+\tau^{2})}(1+\tau^{2})\right)
			\tau^{\frac{n}{2}+2s}J_{\frac{n}{2}-1}(r\tau)\, d\tau\, dt\\
			=\;& 2s\frac{(2\pi)^{\frac{n}{2}}}{r^{\frac{n}{2}}}\int_{0}^{+\infty}\int_{0}^{+\infty}{e^{-t\tau^{2s}}}\left(e^{-t(1+\tau^{2})}-te^{-t(1+\tau^{2})}(1+\tau^{2})\right)
			\tau^{\frac{n}{2}}J_{\frac{n}{2}-1}(r\tau)\, d\tau\, dt\\
			=\;& \frac{2s}{r}\mathcal{K}(r)-2s\frac{(2\pi)^{\frac{n}{2}}}{r^{\frac{n}{2}}}\int_{0}^{+\infty}\int_{0}^{+\infty}te^{-t(1+\tau^{2s}+\tau^2)}(1+\tau^{2})
			\tau^{\frac{n}{2}}J_{\frac{n}{2}-1}(r\tau)\, d\tau\, dt\\
			=\;&\frac{2s}{r}\mathcal{K}(r)-\frac{2s}{r}\int_{0}^{+\infty} te^{-t}\mathcal{H}(r,t,t)\, dt -2s\frac{(2\pi)^{\frac{n}{2}}}{r^{\frac{n}{2}}}\int_{0}^{+\infty}\int_{0}^{+\infty}e^{-t(1+\tau^{2s}+\tau^2)}
			\tau^{\frac{n}{2}}J_{\frac{n}{2}-1}(r\tau)t\tau^2\, d\tau\, dt.
		\end{split}
	\end{equation*}
	As a consequence,
	\begin{equation}\label{G}
		\begin{split}
			&\frac{(2\pi)^{\frac{n}{2}}}{r^{\frac{n}{2}}}\int_{0}^{+\infty}\int_{0}^{+\infty}e^{-t(1+\tau^{2s}+\tau^2)}\tau^{\frac{n}{2}}J_{\frac{n}{2}-1}(r\tau)\left(2ts\tau^{2s}+2t\tau^2\right)\, d\tau\, dt\\
			=\;&(2s-2)\frac{(2\pi)^{\frac{n}{2}}}{r^{\frac{n}{2}}}\int_{0}^{+\infty}\int_{0}^{+\infty}e^{-t(1+\tau^{2s}+\tau^2)}\tau^{\frac{n}{2}}J_{\frac{n}{2}-1}(r\tau)t\tau^{2s}\, d\tau\, dt +\frac{2}{r}\mathcal{K}(r)-\frac{2}{r}\int_{0}^{+\infty} te^{-t}\mathcal{H}(r,t,t)\, dt\\
			=:\;& (2s-2)\mathcal{G}(r)+\frac{2}{r}\mathcal{K}(r)-\frac{2}{r}\int_{0}^{+\infty} te^{-t}\mathcal{H}(r,t,t)\, dt.
		\end{split}
	\end{equation}
	{F}rom~\eqref{definition of H} and~\eqref{decay of H11 }, it follows that, for any~$r>1$,
	\begin{equation}\label{g3}
			\int_{0}^{+\infty} te^{-t}\mathcal{H}(r,t,t)\, dt=\int_{0}^{+\infty}te^{-t}\mathcal{H}(r,t)\, dt\leqslant \frac{c}{r^{n+2s}}
	\end{equation}
	where the constant~$c$ depends only on~$n$ and~$s$.
	
Let us now estimate the integral~$\mathcal{G}(r)$. For this,
for every~$t_1$, $t_2>0$ and~$x\in\mathbb{R}^n\setminus\left\{0\right\}$ we set
	$$\Phi(x,t_1,t_2):=\frac{1}{|x|}\int_{\mathbb{R}^n}e^{-t_1|\xi|^{2s}-t_2|\xi|^{2}}|\xi|^{2s} e^{2\pi ix\cdot\xi}\, d\xi.$$
	Using again the Fourier Inversion Theorem for radial functions, we see that 
	\begin{equation}\label{Phi}
		\Phi(x,t_1,t_2)=\frac{(2\pi)^{\frac{n}{2}}}{|x|^{\frac{n}{2}}}\int_{0}^{+\infty} e^{-t_1r^{2s}-t_2r^{2}} r^{\frac{n}{2}+2s} J_{\frac{n}{2}-1}(|x|r)\, dr,
	\end{equation}
where~$J_{\frac{n}{2}-1}$ denotes the Bessel function of first kind of order~$ \frac{n}{2}-1$. Furthermore, we observe that~$\Phi$ satisfies
the following scaling properties:	\begin{equation}\label{phi}
		\Phi(x,t,t)=t^{-\frac{n}{2s}-1}\Phi(t^{-\frac{1}{2s}}x,1,t^{1-\frac{1}{s}})
		=t^{-\frac{n}{2}-s}\Phi(t^{-\frac{1}{2}}x,t^{1-s},1).
\end{equation}
By the definition of~$\mathcal{G}$ in~\eqref{G}, we see that 
	\begin{equation*}
		\mathcal{G}(r)=\int_{1}^{+\infty}te^{-t}\Phi(r,t,t)\, dt+\int_{0}^{1}te^{-t}\Phi(r,t,t)\, dt
		=:\mathcal{G}_1(r)+\mathcal{G}_2(r).
	\end{equation*}
	We now focus on estimating~$\mathcal{G}_1$ and~$\mathcal{G}_2$.
	\smallskip
	
	{ Step~1. } We estimate~$\mathcal{G}_1$.
	Let us start by calculating the value of~$\lim\limits_{|x|\to+\infty}|x|^{n+2s+1}\Phi(x,1,\eta)$ for every~$\eta\in(0,1)$. {F}rom~\eqref{Phi}, we  infer that 
$$
			\Phi(x,1,\eta)=\frac{(2\pi)^{\frac{n}{2}}}{|x|^{n+2s+1}}\int_{0}^{+\infty} e^{-\frac{t^{2s}}{|x|^{2s}}-\eta\frac{t^{2}}{|x|^{2}}}t^{\frac{n}{2}+2s} J_{\frac{n}{2}-1}(t)\, dt.
$$	Thus, owing to the fact that~$ H^{(1)}_{\frac{n}{2}-1}$ is the Bessel function of the third kind, we have that
	\begin{equation*}
		|x|^{n+2s+1}\Phi(x,1,\eta)={(2\pi)^{\frac{n}{2}}}\text{Re}\int_{0}^{+\infty}e^{-\frac{t^{2s}}{|x|^{2s}}-\eta\frac{t^{2}}{|x|^{2}}}t^{\frac{n}{2}+2s} H^{(1)}_{\frac{n}{2}-1}(t)\, dt.
	\end{equation*}
Using again the Residue Theorem, we  deduce that 
	\begin{equation*}
		\begin{split}
			|x|^{n+2s+1}\Phi(x,1,\eta)
			&={(2\pi)^{\frac{n}{2}}}\text{Re}\int_{L_1}e^{-\frac{z^{2s}}{|x|^{2s}}-\eta\frac{z^{2}}{|x|^{2}}}z^{\frac{n}{2}+2s} H^{(1)}_{\frac{n}{2}-1}(z)\, dz
		\end{split}
	\end{equation*}
	where $L_1$ is defined in~\eqref{definition of L1}. 
	
We now claim that for any~$ \epsilon>0$  there exists~$M>0$ independent of~$\eta$ such that, for every~$|x|>M$ and~$\eta\in(0,1)$,
	\begin{equation*}
		\left|\,|x|^{n+2s+1}	\Phi(x,1,\eta)- (2\pi)^{\frac{n}{2}}\text{Re}\int_{L_1} {z^{\frac{n}{2}+2s}} H^{(1)}_{\frac{n}{2}-1}\left(z\right)\, dz\right|<\epsilon.
	\end{equation*}
For this purpose, we notice that, for~$R>0$ large enough,  \begin{equation*}
		\begin{split}
			&\left|\,|x|^{n+2s+1}	\Phi(x,1,\eta)-(2\pi)^{\frac{n}{2}}\text{Re}\int_{L^1} {z^{\frac{n}{2}+2s}} H^{(1)}_{\frac{n}{2}-1}\left(z\right)\,dz\right|\\
			=\;&	\left|{(2\pi)^{\frac{n}{2}}} \text{Re}\int_{L_1} \left(e^{-\frac{z^{2s}}{|x|^{2s}}-\eta\frac{z^{2}}{|x|^{2}}}-1\right) {z^{\frac{n}{2}+2s}} H^{(1)}_{\frac{n}{2}-1}(z)\, dz\right|\\
			\leqslant \;& (2\pi)^{\frac{n}{2}}\int_{0}^{R}
			\left|e^{-\frac{\left(re^{i\frac{\pi}{6}}\right)^{2s}}{|x|^{2s}}-\eta\frac{\left(re^{i\frac{\pi}{6}}\right)^{2}}{|x|^{2}}}-1\right|{r^{\frac{n}{2}+2s}}\left|H^{(1)}_{\frac{n}{2}-1}\left(re^{i\frac{\pi}{6}}\right)\right|\, dr\\
			&\qquad + (2\pi)^{\frac{n}{2}}\int_{R}^{+\infty}\left|e^{-\frac{\left(re^{i\frac{\pi}{6}}\right)^{2s}}{|x|^{2s}}-\eta\frac{\left(re^{i\frac{\pi}{6}}\right)^{2}}{|x|^{2}}}-1\right|{r^{\frac{n}{2}+2s}}\left|H^{(1)}_{\frac{n}{2}-1}\left(re^{i\frac{\pi}{6}}\right)\right|\, dr\\
			=:\;&\mathcal{B}^{R}_1+\mathcal{B}^{R}_2.
		\end{split}
	\end{equation*}
To estimate~$\mathcal{B}_2^R$, we note that, by the definition of~$ H^{(1)}_{\frac{n}{2}-1}\left(z\right) $ (see e.g.~\cite[page~21]{MR58756}), 
	\begin{equation*}
		\begin{split}
			&\mathcal{B}^{R}_2\leqslant (2\pi)^{\frac{n}{2}}\int_{R}^{+\infty}{r^{\frac{n}{2}+2s}}\left|H^{(1)}_{\frac{n}{2}-1}\left(re^{i\frac{\pi}{6}}\right)\right|\, dr\leqslant (2\pi)^{\frac{n}{2}}2\int_{0}^{+\infty} e^{\left(\frac{n}{2}-1\right)t} \int_{R}^{+\infty}e^{-\frac{r}{4}e^t} {r^{\frac{n}{2}+2s}} \, dr  \,dt\\
			&\qquad\leqslant (2\pi)^{\frac{n}{2}}2\int_{0}^{+\infty} e^{\left(\frac{n}{2}-1\right)t} \int_{R}^{+\infty}e^{-\frac{r}{4}(t+1)} {r^{\frac{n}{2}+2s}} \, dr  \,dt\\
			&\qquad\leqslant (2\pi)^{\frac{n}{2}}2\int_{0}^{+\infty} e^{\left(\frac{n}{2}-1-\frac{R}{4}\right)t} \int_{R}^{+\infty}e^{-\frac{r}{4}} {r^{\frac{n}{2}+2s}} \, dr  \,dt\leqslant\frac{(2\pi)^{\frac{n}{2}}2^{n+4s+5}\Gamma(\frac{n}{2}+2s+1)}{R-2n+2}.
		\end{split}
	\end{equation*}
	As a consequence, given~$\epsilon>0$, taking
	\begin{equation}\label{R_3}
		R_3:=\frac{2(2\pi)^{\frac{n}{2}}2^{n+4s+5}\Gamma(\frac{n}{2}+2s+1)}{\epsilon}+2n-2,
	\end{equation}
one has that
	\begin{equation}\label{B-2}
		\mathcal{B}^{R_3}_2\leqslant \frac{\epsilon}{2}.
	\end{equation}
  	We now consider~$\mathcal{B}^{R_3}_1$ where~$R_3$ is defined in~\eqref{R_3}.  Recalling~\eqref{uniform bounded in eta}, we have that
	\begin{equation*}
		\begin{split}
			\mathcal{B}^{R_3}_1&\leqslant(2\pi)^{\frac{n}{2}}\int_{0}^{R_3}\bigg(\left|e^{-\frac{r^{2s}\cos\frac{s\pi}{3}}{|x|^{2s}}-\eta\frac{r^{2}\cos\frac{\pi}{3}}{|x|^{2}}}\cos\left(-\frac{r^{2s}\cos\frac{s\pi}{3}}{|x|^{2s}}-\eta\frac{r^{2}\cos\frac{\pi}{3}}{|x|^{2}}\right)-1\right|\\
			&\qquad+ \left|\sin\left(-\frac{r^{2s}\cos\frac{s\pi}{3}}{|x|^{2s}}-\eta\frac{r^{2}\cos\frac{\pi}{3}}{|x|^{2}}\right)\right|\bigg){r^{\frac{n}{2}+2s}}\left|H^{(1)}_{\frac{n}{2}-1}\left(re^{i\frac{\pi}{6}}\right)\right|\, dr\\
			&\leqslant\frac{(2\pi)^{\frac{n}{2}}\Gamma\left(\frac{n}{2}+2s+1\right)2^{n+4s+2}}{s+1}\left(1-e^{-\frac{R_3^{2s}\cos\frac{s\pi}{3}}{|x|^{2s}}-\frac{R_3^{2}}{2|x|^{2}}}\cos\left(\frac{R_3^{2s}}{|x|^{2s}}+\frac{R_3^{2}}{2|x|^{2}}\right)+\sin\left(\frac{R_3^{2s}}{|x|^{2s}}+\frac{R_3^{2}}{2|x|^{2}}\right)\right).
		\end{split}
	\end{equation*}
	It follows that for any~$\epsilon>0$ there exists~$T_{R_3,\epsilon}>0$  such that, for every~$|x|>T_{R_3,\epsilon}$,
$$
		1-e^{-\frac{R_3^{2s}\cos\frac{s\pi}{3}}{|x|^{2s}}-\frac{R_3^{2}}{2|x|^{2}}}\cos\left(\frac{R_3^{2s}}{|x|^{2s}}+\frac{R_3^{2}}{2|x|^{2}}\right)+\sin\left(\frac{R_3^{2s}}{|x|^{2s}}+\frac{R_3^{2}}{2|x|^{2}}\right)\leqslant \frac{(s+1)\epsilon}{2^{n+4s+3}(2\pi)^{\frac{n}{2}}\Gamma(\frac{n}{2}+2s+1)},
$$
which implies that 
	\begin{equation*}
		\mathcal{B}^{R_3}_1<\frac{\epsilon}{2}.
	\end{equation*}
	By combining this and~\eqref{B-2}, we deduce that
	for any~$\epsilon>0$ there exists~$M>0$, depending only on $n$, $s$ and~$\epsilon$,  such that, for every~$|x|>M$,
	\[ \left|\,|x|^{n+2s+1}	\Phi(x,1,\eta)-(2\pi)^{\frac{n}{2}}\text{Re}\int_{L_1} {z^{\frac{n}{2}+2s}} H^{(1)}_{\frac{n}{2}-1}\left(z\right)\,dz\right|\leqslant {\epsilon}. \]
	That is,
	\begin{equation*}
		\lim\limits_{|x|\to+\infty}|x|^{n+2s+1}	\Phi(x,1,\eta)=(2\pi)^{\frac{n}{2}}\text{Re}\int_{L_1} {z^{\frac{n}{2}+2s}}, H^{(1)}_{\frac{n}{2}-1}\left(z\right)\,dz
	\end{equation*}
	where the limit is independent of~$\eta$.
	
	Moreover, applying the Residue Theorem again to the region composed of~$L^1$ and the straight line $L_2$ given in~\eqref{L2},  and using the modified
	Bessel function of the third kind~$K(x)$ (see e.g.~\cite[page~5]{MR58756}), it follows that
		\begin{equation}
		\begin{split}\label{egre}
			\lim\limits_{|x|\to+\infty}|x|^{n+2s+1}	\Phi(x,1,\eta)&=(2\pi)^{\frac{n}{2}}\text{Re}\int_{L_2} {z^{\frac{n}{2}+2s}} H^{(1)}_{\frac{n}{2}-1}\left(z\right)\,dz\\&=(2\pi)^{\frac{n}{2}}\text{Re}\int_{0}^{+\infty}ir^{\frac{n}{2}+2s}e^{i\frac{\pi n}{4}}e^{i\pi s}H^{(1)}_{\frac{n}{2}-1}\left(ir\right)\,dr\\
			&=(2\pi)^{\frac{n}{2}}\text{Re}\int_{0}^{+\infty}r^{\frac{n}{2}+2s}ie^{i\pi s}\frac{2}{\pi}K_{\frac{n}{2}-1}(r)\, dr\\
			&=-2^{n+2s}\pi^{\frac{n}{2}-1}\sin(\pi s)\Gamma\left(\frac{n}{2}+s\right)\Gamma(s+1):=-\beta.
		\end{split}
	\end{equation}
	{F}rom this, we infer that there exists~$M_1>0$, depending only on~$n $ and~$s$, such that, for every~$ |x|>M_1$,
	\begin{equation*}
		\frac{\beta}{2}<|x|^{n+2s}|\Phi(x,1,\eta)|<2\beta.
	\end{equation*}
	
	{F}rom~\eqref{phi}, we  deduce that, if~$t>1$ and~$ |x|>M_1t^{\frac{1}{2s}}$,
	\begin{equation}\label{t>1 x>m 1} 
		\frac{\beta t^{\frac{1}{2s}}}{2|x|^{n+2s+1}}<	|\Phi(x,t,t)|<\frac{2\beta t^{\frac{1}{2s}}}{|x|^{n+2s+1}} .
	\end{equation} 
	Moreover, according to the definition of~$\Phi(x,t,t)$, one has that
	\begin{equation}\label{t>1 x<m 1}
		0\leqslant	|\Phi(x,t,t)|\leqslant \Lambda_{n,s} |x|^{-1}\left(t^{-\frac{n}{2s}-1}\wedge t^{-\frac{n}{2}-s} \right).
	\end{equation}
	Owing to~\eqref{t>1 x>m 1} and~\eqref{t>1 x<m 1},   we  conclude that, {for every }$|x|>M_1$,
	\begin{equation}
		\begin{split}\label{phi_1 x>m}
			|\mathcal{G}_1(x)|&\leqslant\int_{1}^{+\infty} te^{-t} |\Phi(x,t,t)|\,dt\leqslant \int_{1}^{\left|\frac{x}{M_1}\right|^{2s}}te^{-t}\frac{2\beta t^{\frac{1}{2s}}}{|x|^{n+2s+1}}\, dt+\int_{\left|\frac{x}{M_1}\right|^{2s}}^{+\infty}te^{-t}t^{-\frac{n}{2s}-1}|x|^{-1}\Lambda_{n,s}\, dt\\
			&\leqslant \frac{2\beta\Gamma(2+\frac{1}{2s})}{|x|^{n+2s+1}}+\frac{\Lambda_{n,s}\Gamma(2)M_1^{n+2s}}{|x|^{n+2s+1}}
		\end{split}
	\end{equation}
	and,  for every~$1<|x|\leqslant M_1$,
	\begin{equation*}
		\begin{split}
			|\mathcal{G}_1(x)|\leqslant\int_{1}^{+\infty} te^{-t} |\Phi(x,t,t)|\,dt\leqslant \int_{\left|\frac{x}{M_1}\right|^{2s}}^{+\infty}te^{-t}t^{-\frac{n}{2s}-1}\Lambda_{n,s}|x|^{-1}\, dt\leqslant \frac{\Lambda_{n,s}\Gamma(2)M_1^{n+2s}}{|x|^{n+2s+1}}.
		\end{split}
	\end{equation*}
	As a consequence of this and~\eqref{phi_1 x>m}, we have that, for every~$ |x|>1$,
	\begin{equation}\label{G_1}
		|\mathcal{G}_1(x)|\leqslant \frac{1}{|x|^{n+2s+1}}\left(2\beta\,\Gamma\left(2+\frac{1}{2s}\right)+\Lambda_{n,s}\Gamma(2)M_1^{n+2s}\right).
	\end{equation}
	\smallskip
	
 {Step~2.} We now estimate the~$\mathcal{G}_2$. To this end, we remark that 
	\begin{equation*}
		\lim\limits_{|x|\to+\infty}|x|^{n+2s+1}	\Phi(x,\eta,1)=-2^{n+2s}\pi^{\frac{n}{2}-1}\sin(\pi s)\Gamma\left(\frac{n}{2}+s\right)\Gamma(s+1)=-\beta,
	\end{equation*}
	where~$\beta$ is defined in~\eqref{egre} and
	the limit is independent of~$\eta\in(0,1)$. 
	
	Recalling~\eqref{phi}, we see that there exists~$M_2>0$, depending only on~$n$ and~$s$, such that, for all~$t\in(0,1)$ and~$ |x|>M_2t^{\frac{1}{2}}$,
	\begin{equation}\label{t>1 x>m 2} 
		\frac{\beta t^{\frac{1}{2}}}{2|x|^{n+2s+1}}<	|\Phi(x,t,t)|<\frac{2\beta t^{\frac{1}{2}}}{|x|^{n+2s+1}}.
	\end{equation} 
	Owing to~\eqref{t>1 x<m 1} and ~\eqref{t>1 x>m 2},   we  conclude that, if~$1<|x|<M_2$,
	\begin{equation}
		\begin{split}\label{phi_2 x>m}
			&|\mathcal{G}_2(x)|\leqslant\int_{0}^{1} te^{-t} |\Phi(x,t,t)|\,dt\leqslant \int_{0}^{\left|\frac{x}{M_2}\right|^{2}}te^{-t}\frac{2\beta t^{\frac{1}{2}}}{|x|^{n+2s+1}}\, dt+\int_{\left|\frac{x}{M_2}\right|^{2}}^{1}te^{-t}t^{-\frac{n}{2}-s}|x|^{-1}\Lambda_{n,s}\, dt\\
			&\qquad\qquad\leqslant \frac{2\beta\Gamma(\frac{5}{2})}{|x|^{n+2s+1}}+\frac{\Lambda_{n,s}\Gamma(2)M_2^{n+2s}}{|x|^{n+2s+1}}
		\end{split}
	\end{equation}
	and, if~$|x|> M_2$,
	\begin{equation*}
		\begin{split}
			|\mathcal{G}_2(x)|\leqslant\int_{0}^{1} te^{-t} |\Phi(x,t,t)|\,dt\leqslant \int_{0}^{\left|\frac{x}{M_2}\right|^{2}}te^{-t}\frac{2\beta t^{\frac{1}{2}}}{|x|^{n+2s+1}}\, dt\leqslant \frac{2\beta\Gamma(\frac{5}{2})}{|x|^{n+2s+1}}.
		\end{split}
	\end{equation*}
	By~\eqref{phi_2 x>m}, one sees that, for every~$ |x|>1$,
	\begin{equation}\label{G_2}
		|\mathcal{G}_2(x)|\leqslant \frac{1}{|x|^{n+2s+1}}\left(2\beta\,\Gamma\left(\frac{5}{2}\right)+\Lambda_2\Gamma(2)M_2^{n+2s}\right).
	\end{equation}
\smallskip

	{Step~3.}
	Combining~\eqref{G_1} and~\eqref{G_2}, we  conclude that, for every~$|x|>1$,
	\begin{equation}\label{decay of G}
		\mathcal{G}(x)\leqslant\frac{c_{n,s}}{|x|^{n+2s+1}}.
	\end{equation}
	
	Finally, recalling~\eqref{K'}, we have that
	$$
		\mathcal{K}'(r)=\frac{-(2\pi)^{\frac{n}{2}}(n-2)}{2r}\mathcal{K}(r)-\frac{n+2}{2r}\mathcal{K}(r)+(2s-2)\mathcal{G}(r)+\frac{2}{r}\mathcal{K}(r)-\frac{2}{r}\int_{0}^{+\infty} te^{-t}\mathcal{H}(r,t,t)\, dt.
	$$
	Owing to the decay of~$\mathcal{ K}$, by combining~\eqref{g3} with~\eqref{decay of G}, we deduce that, for every~$|x|>1$,
	$$
		|\nabla\mathcal{ K}(x)|\leqslant|\mathcal{ K}'(r)|\leqslant \frac{C}{r^{n+2s+1}}=\frac{C}{|x|^{n+2s+1}}
$$
	for some constant~$C$ depending only on~$n$ and~$s$.  
	
	Using similar arguments, we  infer that, for every~$|x|>1$,
	$$
		|D^2\mathcal{ K}(x)|\leqslant\frac{C}{|x|^{n+2s+2}}
$$	for some constant~$C$ depending only on~$n$ and~$s$. 
\end{proof}

\subsection{Proof of Theorem~\ref{th properties of k}}\label{fhuoewyt9843t0987654}

We observe that, gathering
	together the results of  Sections~\ref{subsec:Decay of the kernel k}
	and~\ref{subsec:Smoothness of kernel}, we  obtain the desired claims in~(a)-(d) of Theorem~\ref{th properties of k}.
Thus, it suffices to prove  property (e) of Theorem~\ref{th properties of k}. To this end, we introduce the following two ancillary results:

\begin{Lemma} \label{FU2}Let~$\phi:\R^n\to\R$ be in the Schwartz space and~$t>0$.
	Then, the function
	$$\R^n\times\R^n\ni(x,\xi)\mapsto e^{-t(|\xi|^{2s}+|\xi|^{2})+2\pi ix\cdot \xi}\,\phi(x)$$ belongs to~$L^1(\R^n\times\R^n,\,\C)$.
\end{Lemma}

\begin{proof} We have that
	$$ \iint_{\R^n\times\R^n} \left|e^{-t(|\xi|^{2s}+|\xi|^{2})+2\pi ix\cdot \xi}\,\phi(x)\right|\,dx\,d\xi
	=\iint_{\R^n\times\R^n} e^{-t(|\xi|^{2s}+|\xi|^{2})}\,|\phi(x)|\,dx\,d\xi=\|\phi\|_{L^1(\R^n)}
	\int_{\R^n} e^{-t(|\xi|^{2s}+|\xi|^{2})}\,d\xi,$$ which is finite.
\end{proof}

\begin{Lemma}
	Let~$\Phi:\R^n\to\C$ be in the Schwartz space.
	Then, the function
	$$\R^n\times(0,+\infty)\ni(\xi,t)\mapsto e^{-t(1+|\xi|^{2s}+|\xi|^{2})}\,\Phi(\xi)$$ belongs to~$L^1(\R^n\times(0,+\infty),\,\C)$ and
	\begin{equation}\label{FU7} \iint_{\R^n\times(0,+\infty)} e^{-t(1+|\xi|^{2s}+|\xi|^{2})}\,\Phi(\xi)\,d\xi\,dt=
		\int_{\R^n}\frac{\Phi(\xi)}{1+|\xi|^{2s}+|\xi|^{2}}\,d\xi.\end{equation}
\end{Lemma}

\begin{proof} We have that
	$$ \int_{0}^{+\infty} e^{-t(1+|\xi|^{2s}+|\xi|^{2})}\,dt=\frac1{1+|\xi|^{2s}+|\xi|^{2}}.
	$$
	Therefore, by Fubini-Tonelli's Theorem,
	$$\iint_{\R^n\times(0,+\infty)}\left|e^{-t(1+|\xi|^{2s}+|\xi|^{2})}\,\Phi(\xi)\right|\,d\xi\,dt
	=\iint_{\R^n\times(0,+\infty)} e^{-t(1+|\xi|^{2s}+|\xi|^{2})}\,|\Phi(\xi)|\,d\xi\,dt=
	\int_{\R^n}
	\frac{|\Phi(\xi)|}{1+|\xi|^{2s}+|\xi|^{2}}\,d\xi,$$
	which is finite.
\end{proof}

With this preparatory work, we can now check that the Bessel kernel is the fundamental solution of the mixed order operator~$- \Delta+  (-\Delta)^s$, as clarified by the following result:

\begin{Lemma}\label{C.6} The Fourier transform of~${\mathcal{K}}$ equals~$\frac{1}{1+|\xi|^{2s}+|\xi|^{2}}$
in the sense of distribution.
	
	More explicitly,
	for every~$\phi:\R^n\to\R$ in the Schwartz space, we have that
	\begin{equation}\label{FU9} \int_{\R^n}{\mathcal{K}}(x)\,\phi(x)\,dx=\int_{\R^n}\frac{\hat\phi(\xi)}{1+|\xi|^{2s}+|\xi|^{2}}\,d\xi.\end{equation}
\end{Lemma}

\begin{proof} By~\eqref{definition of K} and Fubini-Tonelli's Theorem
	(whose validity is a consequence here of Lemma~\ref{decay of K_2}), we have that
	\begin{equation*}
		\int_{\R^n}{\mathcal{K}}(x)\,\phi(x)\,dx=\int_{\R^n}\left(
		\int_{0}^{+\infty} e^{-t} \,\mathcal{H}(x,t) \,dt\right)\phi(x)\,dx=\int_{0}^{+\infty} 
		\left(\int_{\R^n}
		e^{-t} \,\mathcal{H}(x,t)\,\phi(x) \,dx\right)\,dt.
	\end{equation*}
	
	Hence, by~\eqref{DBSD-1},
	\begin{equation*}
		\int_{\R^n}{\mathcal{K}}(x)\,\phi(x)\,dx=\int_{0}^{+\infty} 
		\left(\int_{\R^n}\left(\int_{\mathbb{R}^n}e^{-t(1+|\xi|^{2s}+|\xi|^{2})+2\pi ix\cdot \xi}\,\phi(x)\,d\xi\right)\,dx\right)\,dt.
	\end{equation*}
	
	This and Fubini-Tonelli's Theorem
	(whose validity is a consequence here of Lemma~\ref{FU2}) yield that
	\begin{equation*}\begin{split}
			\int_{\R^n}{\mathcal{K}}(x)\,\phi(x)\,dx&=\int_{0}^{+\infty} 
			\left(\int_{\R^n}\left(\int_{\mathbb{R}^n}e^{-t(1+|\xi|^{2s}+|\xi|^{2})+2\pi ix\cdot \xi}\,\phi(x)\,dx\right)\,d\xi\right)\,dt\\&=
			\int_{0}^{+\infty} 
			\left(\int_{\R^n}e^{-t(1+|\xi|^{2s}+|\xi|^{2})}\,\hat\phi(\xi)\,d\xi\right)\,dt.\end{split}
	\end{equation*}
	{F}rom this and~\eqref{FU7} (utilized here with~$\Phi:=\hat\phi$) we arrive at~\eqref{FU9}, as desired.
\end{proof}

\begin{proof}[Proof of (e) of Theorem~\ref{th properties of k}]
Let~$ q\in [1,+\infty)$ and~$f\in L^q(\mathbb{R}^n)$. Since~$\mathcal{S}(\mathbb{R}^n)$ is dense in~$L^q(\mathbb{R}^n)$,  there exists a sequence~$f_m\in \mathcal{S}(\mathbb{R}^n)$ such that 
  \begin{equation}\label{vsdvds}
  		f_m\rightarrow f \quad \text{in } L^q(\mathbb{R}^n). 
  \end{equation}
  	
  	{F}rom Lemma~\ref{C.6}, it follows that the function~$u_m:=\mathcal{K}\ast f_m$ is a solution of
  	$$
  	- \Delta u_m +  (-\Delta)^s u_m+u_m = f_m \quad \text{ in }  \mathbb{R}^n.  
  	$$ Namely, for every~$\varphi\in C_c^\infty(\mathbb{R}^n)$,
  	\begin{equation}\label{bcjsdhf}
  		\int_{\mathbb{R}^n}- \Delta \varphi\, u_m + \int_{\mathbb{R}^n} (-\Delta)^s \varphi\, u_m+\int_{\mathbb{R}^n}u_m\, \varphi = \int_{\mathbb{R}^n}f_m\, \varphi.
  	\end{equation}
  		Moreover, owing to~\eqref{vsdvds}
  		and taking into account the fact that~$\mathcal{ K} \in L^1(\mathbb{R}^n)$, we infer that~$ u_m\rightarrow u:=\mathcal{ K}\ast f $  in~$ L^q(\mathbb{R}^n).$  
  		
  		By passing to the limit in~\eqref{bcjsdhf}, we conclude that 
  		the function~$u=\mathcal{K}\ast f$ is a solution of $$
  		- \Delta u +  (-\Delta)^s u+u = f \quad \text{ in }  \mathbb{R}^n.    $$ 
  	
  	As a consequence of this, by Lemmata~\ref{lemma  H1}, \ref{proposition t>1}, \ref{decay of K_2} and~\ref{sommothness of K}, we obtain~(e) of Theorem~\ref{th properties of k}.
 \end{proof}

\end{appendix}

\bibliographystyle{is-abbrv}

\bibliography{manuscript}
\vfill
\end{document}